\documentclass[11pt]{amsart}
\usepackage{amssymb, amscd, amsmath, amsthm, epsf,  latexsym,color ,  mathtools,
}

\usepackage[lmargin=1in,rmargin=1in,tmargin=1in,bmargin=1in]{geometry}
\usepackage{setspace} 
\usepackage[pdfpagelabels,bookmarksdepth=3]{hyperref}
\hypersetup{colorlinks=true,citecolor=black,urlcolor=black,linkcolor=black}
\usepackage{pinlabel}
\usepackage{pb-diagram}
\usepackage{soul}
 \usepackage{tikz-cd}
\usepackage{comment}
\usepackage{float}
\usepackage{graphicx}
\usepackage{mathrsfs}
\usepackage{wrapfig}
\restylefloat{figure}
\usepackage[font=small,labelfont=bf]{caption}
\usepackage{subcaption}
\usepackage{relsize}

 \usepackage{color}
\usepackage{xcolor}
 
\newtheorem{theoremABC}{Theorem}

\numberwithin{equation}{section}
\newtheorem{theorem}{Theorem}[section]
\newtheorem*{theorem*}{Theorem}
\newtheorem{corollary}[theorem]{Corollary}
\newtheorem{lemma}[theorem]{Lemma}

\newtheorem{proposition}[theorem]{Proposition}

\theoremstyle{definition}
\newtheorem{definition}[theorem]{Definition}

\newtheorem*{remark*}{Remark}

\newcommand{\ep}{\epsilon}

\newcommand{\bbi}{{{\bf i}}}
\newcommand{\bbj}{{{\bf j}}}
\newcommand{\bbk}{{{\bf k}}}

\newcommand{\ZZ}{{\mathbb Z}}
\newcommand{\RR}{{\mathbb R}}

\newcommand{\WNb}{{\widetilde{\mathcal{N}}}}
\newcommand{\Nb}{{\mathcal{N}}}

\newcommand{\rin}{{\operatorname{1}}}
\newcommand{\rout}{{\operatorname{0}}}

\newcommand{\RN}[1]{%
  \textup{\uppercase\expandafter{\romannumeral#1}}%
}

\newcommand{\nat}{\natural}

\DeclareMathOperator{\Real}{Re}
\DeclareMathOperator{\Ima}{Im}

\DeclareMathOperator{\Hom}{Hom}

\DeclareMathOperator{\rank}{rank}

\graphicspath{ {figures/} }

 \renewcommand{\qed}{\hfill$\square$}

\newcommand{\ba}{\bar{a}}
\newcommand{\bb}{\bar{b}}
\newcommand{\bc}{\bar{c}}
\newcommand{\be}{\bar{e}}
\newcommand{\bff}{\bar{f}}
\newcommand{\bq}{\bar{q}}
\newcommand{\bp}{\bar{p}}
\newcommand{\bh}{\bar{h}}
\newcommand{\bg}{\bar{g}}

\newcommand{\tracelessTwoSphere}{{\mathbb{S}_\bbi}} 

\newcommand{\FPS}{{(S^2, 4)}} 

\newcommand{\preScaledNatural}[2]{\raisebox{1pt}{\scalebox{2.3}[1]{$#1\natural$}}}
\newcommand{\scaledNatural}{{\mathpalette\preScaledNatural\relax}}
\DeclareMathOperator{\NAT}{{\scaledNatural}}
\newcommand{\tNAT}{\widetilde{\scaledNatural}}

\newcommand{\hNAT}{\widehat{\scaledNatural}}

\DeclareMathOperator{\Lag}{\mathrm{Lag}}

\author{Christopher M. Herald}
\address{Department of Mathematics and Statistics, University of Nevada,  Reno, NV 89557} 
\email{herald@unr.edu}

 \author{Paul Kirk}
\address{Department of Mathematics, Indiana University, Bloomington, IN 47405} 
\email{pkirk@indiana.edu}

 \date{\today}

\thanks{CH is supported by a Simons Collaboration Grant for Mathematicians. PK acknowledges support from the Fourier Institute in Grenoble and the Max Planck Institute in Bonn, where this work was carried out. }


\begin{document}

\title{An endomorphism on immersed curves in the pillowcase}

\begin{abstract} 
We examine the holonomy-perturbed traceless $SU(2)$ character variety of the trivial four-stranded tangle $\{p_1,p_2,p_3,p_4\} \times [0,1]$ in $S^2 \times [0,1]$ equipped  with a strong marking, either an {\em earring} or a {\em bypass}.  Viewing these marked tangles as endomorphisms in the cobordism category from the four-punctured sphere to itself, we identify the images of these endomorphisms in the Weinstein symplectic partial category under the partially defined holonomy-perturbed traceless character variety functor. We express these endomorphisms on immersed curves in the pillowcase in terms of doubling and figure eight operations
and prove they have the same image.

\end{abstract}

\maketitle


\section{Introduction}\label{intro}

The operation of introducing a particular {\em strong marking} (in the sense of \cite[Section 2.4]{KMweb}), which we called the {\em earring},  on  the product tangle 
  $(S^2\times I,\{p_1,p_2,p_3,p_4\}\times I)$ induces an endomorphim on the set of immersed curves in the traceless $SU(2)$ character variety   of a four-punctured two-sphere. The main result in this article completes the description, begun in \cite{CHKK}, of   this endomorphism. We also extend the result to another strong marking which we call the {\em bypass}.  These two markings are  illustrated on the left and right in  Figure \ref{atomfig}.
  Such markings are used in instanton gauge theory to ensure that 
  critical points of the Chern-Simons function are separated from  the orbit singularities in 
  its domain.

We interpret the operation of introducing a marking on a link or tangle in a 3-manifold as  a surgery operation. First, choose a {\em Conway sphere}  (an embedded 2-sphere meeting  the link of tangle  transversely in four points, denoted in this article by $(S^2,4)$).  Then, remove a tubular neighborhood of the Conway sphere (such a neighborhood is illustrated in the middle in Figure \ref{atomfig2}). Finally,  glue in either  the {\em earring tangle} or the {\em  bypass tangle}  illustrated in Figure \ref{atomfig2} (these tangles  have trivalent vertices).

 From a TQFT perspective, this semi-local operation replaces the product tangle, which represents the identity endomorphism of the Conway sphere  in the cobordism category, by the 
 earring or bypass tangle.  Our main results identify the image of this morphism in the Weinstein symplectic partial category \cite{weinstein} under the (partially defined) {\em holonomy perturbed traceless character variety}  functor.

 \medskip

  The traceless $SU(2)$ character variety $P$ of  the Conway sphere  is a {\em pillowcase}, analytically isomorphic to  the quotient of a torus by its elliptic involution. It is homeomorphic to a  two-sphere, with four distinguished orbifold points (called {\em corners}) which correspond to the four fixed points of the elliptic involution.  The complement of these four points is a smooth open variety which we denote by $P^*$; it is endowed with its Atiyah-Bott-Goldman symplectic form.

  The boundary of the earring (resp. bypass) tangle is a disjoint union of two Conway spheres, with character variety a {\em product} $P_0\times P_1$ of two pillowcases.   In the earlier article \cite{CHKK}, the $s$-holonomy perturbed traceless character variety of the  earring tangle (denoted  $\NAT_s$)   is identified, for small non-zero $s$,  as a smooth genus three surface. Furthermore,  the restriction map $$u_s:\NAT_s\to P_0^*\times P_1^*,$$
 a Lagrangian immersion, is   determined {\em up to homology}.  The tools used to prove the  results of \cite{CHKK}  are insufficient to determine what endomorphism is induced on the set of  regular homotopy classes of immersed curves in the pillowcase  by  Weinstein composition  \cite{weinstein} with $u_s$.

 This article provide a more precise geometric description of the Lagrangian immersion $u_s$, sufficient to determine its action  on regular homotopy classes of immersed curves in the pillowcase.  We also identify  the corresponding Lagrangian immersion $u_s'$ for the bypass tangle and prove that the endomorphisms induced  by $u_s$ and $u_s'$ are equivalent.
\medskip

Our main results  are the following. Theorem \ref{thm6.1} states:  
 
\begin{theoremABC}\label{thmA}  For small nonzero perturbation parameter $s$, the restriction-to-the-boundary map from the holonomy perturbed traceless character variety $\NAT_s$ (resp.~$\NAT_s'$), a closed genus three surface, to a product of pillowcases
$$u_s:\NAT_s \to P_0^{*}\times P_1^*$$  is a Lagrangian bifold immersion into the smooth top stratum with respect to the Atiyah-Bott-Goldman symplectic structure. The  compositions ${\rm proj}_0\circ u_s$ and ${\rm proj}_1\circ u_s$  have the same critical set, a disjoint union of four circles which separate $\NAT_s$ into two four-punctured two-spheres. Under these maps to the factors, the four critical circles smoothly embed and  encircle  the four corners of $P_0$ and $P_1$ and each complementary four-punctured two-sphere in $\NAT_s$ embeds into  $P_0$ and $P_1$. The identical statement holds for $u_s':\NAT_s'\to P_0^*\times P_1^*$.
\end{theoremABC}

In this theorem, a {\em Lagrangian bifold} refers to a Lagrangian immersion of a surface into a product of two symplectic surfaces whose projection to each factor has only fold singularities. We give a precise definition of this notion, which may be of  independent interest,   in Section \ref{bifolds}.   

\medskip

In preparation for the statement of our second main result, let $\Lag(P^*)$  denote  the set of good immersed curves in   $P^*$.   Also,  $[\Lag(P^*)]$ denotes the set of regular homotopy classes and $[g]$ denotes the regular homotopy class of $g\in \Lag(P^*)$.   
 These notions are given precise definitions in Section \ref{LaginP}.
 
In Section \ref{operations}, we define the following operations on regular homotopy classes. First,
 $D([g])$  refers to the  {\em double} of   $[g]\in[\Lag (P^*)] $    (Definition  \ref{double}).  If $g$ is a good immersion of an interval, $F8([g])$ denotes the regular homotopy class of the {\em Figure Eight curve  supported by}  $g$   (Definition \ref{deffig8} and Figure \ref{Fig5fig}).  The symplectomorphism $\Psi:P_0^*\to P_1^*$, defined in 
Equation (\ref{Psi}),  is the  induced by the cylindrical identification of the two boundary components of the tangle.

\begin{theoremABC}  \label{thm7.1}   The Lagrangian immersions $u_s:\NAT_s\to P_0^*\times P_1^*$
induce a well defined function 
 $$
\lim_{s\to 0}(u_s)_*:[\Lag(P_0^*)] \to [\Lag (P_1^*)],
$$
determined  by the following two assertions:
\begin{enumerate}
\item If  $g:J\to P_0^*$ is a good immersion of an open interval, then $$\lim_{s\to 0}(u_s)_*([g])=\Psi(F8([g])).$$
\item If  $g:S^1\to P_0^*$ is an immersion, then 
$$\lim_{s\to 0}(u_s)_*([g])=\Psi(D([g])).$$
\end{enumerate}

The identical statement holds for $\NAT_s'$.
\end{theoremABC} 
\begin{proof}
 Assertion (1) is Theorem \ref{FIG8baby}.  Assertion (2) is Theorem \ref{interiorc}.
\end{proof}
 
 Limits are used to streamline  notation. An alternative formulation is:   for any $g\in \Lag(P_0^*)$, there exists an $s(g)>0$ so that for all $0<|s|<s(g)$, $g$  is Weinstein-composable with $u_s$ 
and satisfies $[(u_s)_*([g])]=\Psi(F8([g]))$ or $\Psi(D([g]))$,  and similarly for $u_s'$.

\medskip

We can now explain precisely why Theorem \ref{thm7.1} strengthens  the corresponding result in \cite{CHKK}.
In that article, a  family of  Lagrangian immersions  $v_\delta:\NAT_s\to P_0^*\times P_1^*, \delta>0$, was constructed, called the {\em model maps}, which satisfy
\begin{itemize}
\item $v_\delta$ is $\delta$-homotopic to $u_s$,
\item $\lim_{\delta\to 0}(v_\delta)_*$ satisfies the conclusion of Theorem  \ref{thm7.1}.
\end{itemize}

All the conclusions in \cite[Sections 9--11]{CHKK} concerning $v_\delta$  depend only on its satisfying the conclusions of Theorem \ref{thm7.1}. Hence:

\begin{theoremABC}\label{CHKK1} All assertions in \cite[Sections 9--11]{CHKK}
concerning the model map $v_\delta:\NAT_s\to P_0^*\times P_1^*$  are valid for the map $u_s$.
\end{theoremABC}

\subsection{Relation to  \cite{CHKK}}

This article is self-contained and does not presuppose that the reader is familiar with the results or notation of \cite{CHKK}; we recall and reprove here those results we need.  However,  many of the mathematical ideas explored in this article were known to all four authors of \cite{CHKK}, as highlighted by Theorem \ref{CHKK1} above.  

What is new in the present article can be summarized as follows.  First, we carry out several asymptotic calculations, notably Lemmas \ref{estimate2} and \ref{fuglylemma}, which provide enough control  over $u_s$ to conclude it is a Lagrangian bifold with embedded fold locus. Second,  we use symmetries and {\em one} calculation, Lemma  \ref{magicformula}, to determine the effect of the Lagrangian correspondence on {\em all} immersed curves up to   {\em regular homotopy}, rather than up to {\em homology} as  in  \cite{CHKK}.   Third, all results are extended to the bypass marking, using   similarities between the fundamental groups of the earring and bypass tangle complements, as calculated in Propositions \ref{summ} and \ref{summ2}. 

\medskip

A quick overview of the results of this article can be obtained by glancing at  Figure \ref{atomfig2}, which illustrates the tangles in $S^2\times I$ whose holonomy perturbed character variety this article is about, and the figures in Section \ref{IBE}, which illustrate the statement of Theorem \ref{thm7.1} for the example of the traceless character varieties associated to a tangle decomposition of the $(3,7)$ torus knot.

\subsection{Context}
A reader wishing to understand the project which this article completes   may want to first read the introduction to \cite{CHKK}, then read this article, and finally read   Sections 10 and 11 of \cite{CHKK} {\em with Theorem \ref{CHKK1} of the current article understood}.   These sections of \cite{CHKK} discuss  the low-dimensional topology context, the relationship to  bounding cochains, and Figure Eight bubbling in immersed Floer theory \cite{bottman, bottman-wehrheim, fukaya}.

One aim of  this article and the earlier articles \cite{HHK,HHHK,HK,FKP} is  to develop an {\em immersed curve calculus in the  pillowcase} for Kronheimer-Mrowka {\em reduced singular  instanton  homology}  \cite{KMweb, KM1, KM2}.)    A calculus for immersed curves in  once-punctured tori for    {\em Heegaard-Floer} homology is constructed by Hanselman-Rasmussen-Watson \cite{HRW} and in pillowcases for {\em Khovanov} homology  by Kotelskiy-Watson-Zibrowius \cite{KWZ}.  These approaches leverage the fact that the Riemann mapping theorem permits a construction  of Lagrangian-Floer theory in surfaces which avoids the use of moduli spaces of $J$-holomorphic polygons \cite{abou}.  These approaches further make use of the theorem of Haiden-Katzarkov-Kontsevich \cite{HaKaKo} concerning representability,  by immersed curves, of objects of the triangulated envelope of the Fukaya category of a punctured surface.

Our results should be viewed in the context of the Atiyah-Floer conjecture and, more broadly,  {\em Floer field theory} \cite{WW2,WW1,WW3, Caz,bottman-wehrheim,fukaya,AJ, FOOO1}.  From this perspective, our paper calculates the value of the character variety Floer field theory functor on the non-trivial endomorphisms of the Conway sphere in the cobordism category given by the earring and bypass tangles.   None of the proofs in this article depend mathematically on any analytic or algebraic aspects of Lagrangian-Floer theory or Floer field theory.  Although we borrow terminology  from symplectic geometry,  the proofs we provide rely only on classical  differential topology, the geometry of $SU(2)$, and the result that $u_s$ and $u_s'$ are Lagrangian immersions, whose proof can be found in \cite{CHK}.

\medskip

As a study  of $SU(2)$ character varieties of  three-manifolds,
Theorems \ref{thmA} and \ref{thm7.1} echo the results of Klassen's 1993 article \cite{Klassen} which contains a calculation of the character variety of the Whitehead link complement.  While that article was not focused on symplectic topology, restriction to the character variety to the boundary (two disjoint tori) gives a Lagrangian immersion into the product of two pillowcases, just like our $u_s, u'_s$. Boozer's article \cite{Boozer1}, which  discusses 1-stranded tangles in  solid tori, and    Smith's article,  \cite{KaiSmith} which examines the {\em tangle addition} tangle, whose boundary is a union of three Conway spheres, also describe holonomy perturbed character varieties  and their boundary restriction maps to symplectic character varieties of dimension greater than two.

 \begin{figure}[ht]
\begin{center}
\def\svgwidth{5.5in}
\includegraphics[width=0.7\textwidth]{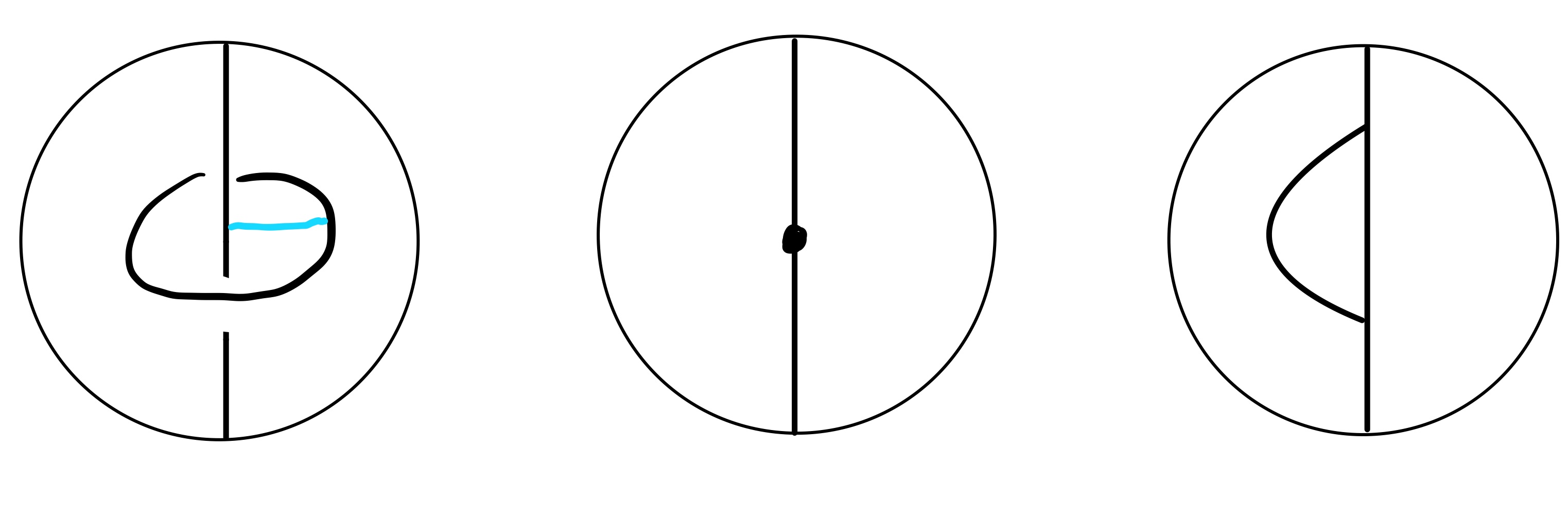}
 \caption{The middle illustrates a  three-ball neighborhood of a base point chosen on a link in a  three-manifold. The left picture illustrates the result of introducing the earring atom to the link at this base point, with the blue arc Poincar\'e dual to the second Stiefel-Whitney class of an SO(3) bundle over the tangle complement.  The right illustrates the result of introducing the bypass  atom to the link at this point, and working in the trivial  bundle.\label{atomfig}}
\end{center}
\end{figure}

 \begin{figure}[ht]
\begin{center}
\def\svgwidth{5.5in}
\includegraphics[width=.8\textwidth]{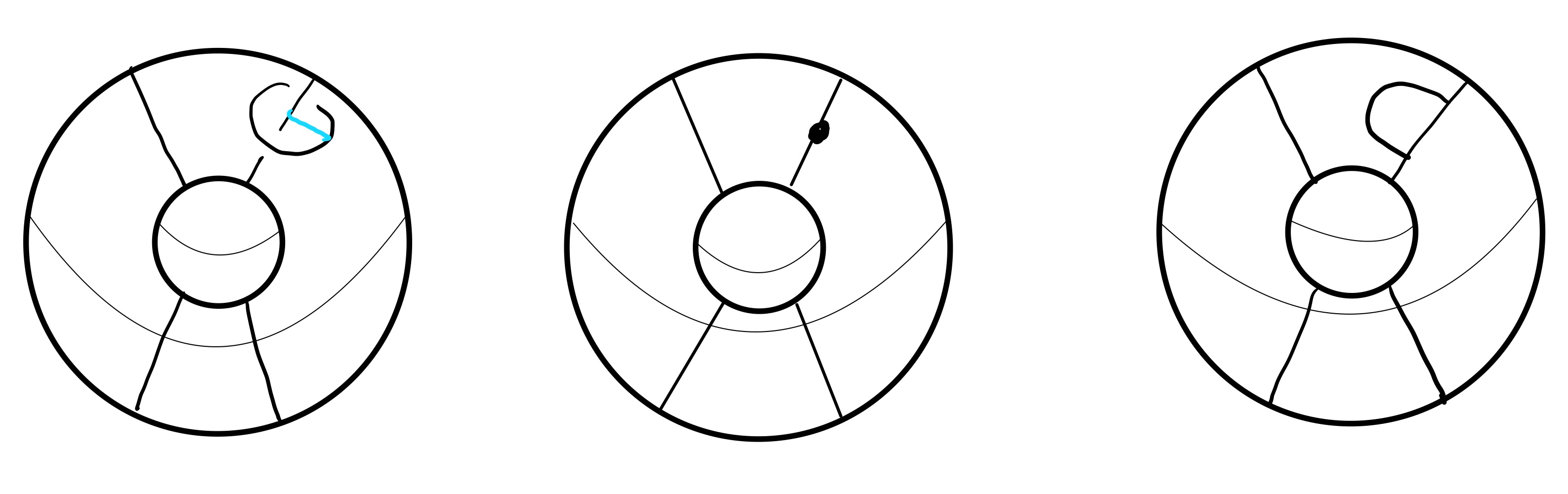}
 \caption{The middle illustrates the product tangle $(S^2,4)\times [0,1]$, with a base point placed on one strand. The left picture illustrates the  earring tangle and the right illustrates   the bypass  tangle.\label{atomfig2}}
\end{center}
\end{figure}

 \noindent{\em Acknowledgement.} We thank G.~ Cazassus and A.~ Kotelskiy  for their many contributions to this project.

\section{Summary of \cite{CHKK} and extension to the bypass tangle}\label{sumsec}

 In this section, we recall and reformulate  those calculations and  notation of \cite{CHKK} about the earring tangle needed for  the present article and we extend them to the bypass tangle.

 \begin{figure}[ht]
\begin{center}
\includegraphics[width=.8\textwidth]{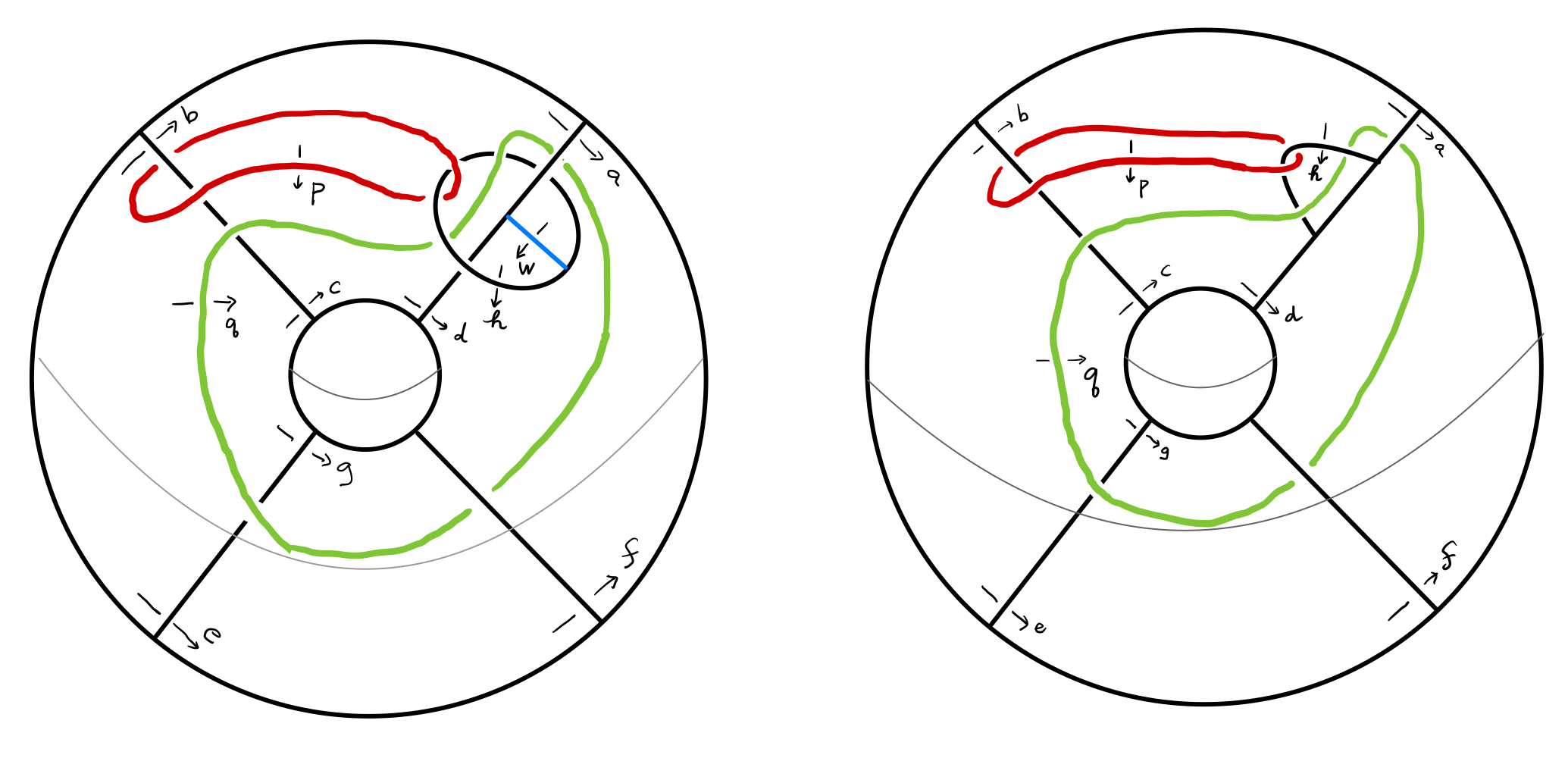}
\\
(a)\hskip2.9in (b)   
 \caption{
 The earring and bypass tangles in $S^2\times I$, enhanced with perturbation curves.  Each labelled arrow represents an element of the fundamental group; a loop based at a point above the page is obtained by concatenating a straight path from the base point to the tail of the arrow, the arrow, and a straight path from the head of the arrow back to the base point.
  \label{earpertfig}}
\end{center}
\end{figure}

Figure \ref{earpertfig}.a illustrates the earring tangle,  enhanced by the addition of two embedded {\em perturbation curves}.  Also illustrated is a  collection of  based loops in the complement, $a,b,e,f,c,d,g,h,p,q,$ and $w$. (Figure \ref{earpertfig}.b is the corresponding illustration for the bypass tangle, which does not appear in \cite{CHKK}, but is analyzed in this article as well.)

The following proposition, which sharpens Proposition 3.1  \cite{CHKK},  provides a  presentation   of the fundamental group $\Pi$  of the complement of this enhanced graph.   It can be proved either using Wirtinger's algorithm  \cite{rolfsen}  or by  direct topological arguments.  The loops $c, e, g, d, $ and $w$ are not needed to generate $\Pi$, but naming them simplifies notation throughout the article.
 Here, and throughout this article,  the inverse of any group element $g$ is denoted  $\bar{g}$.

  \begin{proposition}\label{summ}
The set of loops
$\{ a,b,  f,p,q, h\}$ generates $\Pi$, and $\Pi$ has the presentation
$$\Pi=\langle  a,b,  f,p,q, h~|~ [p,\lambda_p], ~ [q,\lambda_q]\rangle,
$$
where  $\lambda_p=bh\text { and }\lambda_q=f\bar{a} h$ denote 
the longitudes of the perturbation curves.

The remaining labelled loops in Figure \ref{earpertfig}.a  are expressed in terms of these as 
$$c=\bq\bp b p q, ~ e=ba\bff, ~g= \bq e q =\bq ba \bff q,~d=\bc g f=\bq\bp \bb p  ba\bff q f
,$$
and
$$w=[\ba h \bq \bp  , h].
 $$

The inclusions of the boundary components $$ \FPS^i\subset  S^2\times I,~ i=\rout,\rin$$ induce peripheral subgroups $$\Pi^\rout=\langle a,b,e,f~|~ ba\bff \be\rangle, \qquad \Pi^\rin=\langle d,c ,g,f~|~ cd\bff \bg\rangle,$$ 
with the inclusion homomorphisms $\Pi^\rout \to \Pi, ~ \Pi^\rin \to \Pi$ preserving the labeling of the generators.\qed

\end{proposition}
 The geometric explanation for the expression for $w$ as a commutator is that the boundary of the tubular neighborhood of the earring is a torus punctured once by the arc $w$, with meridian $h$ and longitude $\ba h \bq \bp $.

\medskip

\subsection{Traceless $SU(2)$ representations}
Identify the Lie group $SU(2)$ with the  three-sphere of  unit quaternions. The traceless elements thereby correspond to the equatorial  two-sphere of unit quaternions with real part zero, which we denote by $$\tracelessTwoSphere=\{x\in SU(2)~|~\Real (x)=0\},
$$
since it coincides with the conjugacy class of $\bbi$.

\medskip

Evaluation on the generators $\{ a,b,f,h,p,q \}$   embeds
$\Hom(\Pi,SU(2))$ in $SU(2)^6$
 as an affine algebraic set. Similarly, $\Hom(\Pi^0,SU(2))$ and $\Hom(\Pi^1,SU(2))$
 are affine algebraic sets.  Inclusion of  the boundary defines a restriction map
\begin{equation}\label{Hom} \Hom(\Pi,SU(2))\to \Hom(\Pi^0,SU(2))\times\Hom(\Pi^1,SU(2))\end{equation}
which is equivariant with respect to the conjugation action of $SU(2)$.

By definition a {\em meridian} of a tangle is any based loop freely homotopic to the boundary of a small disk meeting the (unenhanced) tangle transversely in a point.

\begin{definition}\label{ptrep}   For any $s\in \RR$, we call a homomorphism $\rho:\Pi\to SU(2)$ an  {\em $s$-perturbed   traceless $SU(2)$ representation    satisfying the $w_2$ condition}\footnote{ 
In this article, the notion of {\em holonomy perturbation} is given a one-parameter, knot-theoretical treatment via the choice of perturbation curves and condition (3) of Definition \ref{ptrep}. We omit any discussion of the origin of this definition and its relation to generically 
perturbing the Chern-Simons function, which is due to Taubes \cite{taubesCI}.  The family of perturbations we use are compatible both with holonomy perturbations in instanton gauge theory, and with Hamiltonian perturbations in Lagrangian-Floer theory  \cite{Heraldthesis, HK, CHK}.
}
  if $\rho$ satisfies 
\begin{align*}
&\text{(1)} &&\rho(w)=-1&& \text{($w_2$ condition)}\\
&\text{(2)} &&0=\Real(\rho(z))\text{ for  any meridian } z \text{ other than $w$}&& \text{(traceless condition)}\\
&\text{(3)} &&\rho(p)=e^{s \Ima(\rho(\lambda_p))},~\rho(q)=e^{s \Ima(\rho(\lambda_q))}&&  \text{($s$-perturbed)}
\end{align*}

For fixed $s\in\RR$, we let $\tNAT_s \subset \Hom(\Pi,SU(2))$ denote the subspace of  $s$-perturbed  traceless representations of $\Pi$ satisfying the $w_2$ condition
and let  $\NAT_s$ denote its orbit space under the conjugation action of $SU(2)$:
$$
\NAT_s=\tNAT_s/_{\rm conjugation}\subset (SU(2)^6)/_{\rm conjugation}.
$$
\end{definition}
 The space $\tNAT_s$  is a real {\em analytic} subset  of $SU(2)^6$, because the perturbation condition is analytic rather than algebraic;  it follows that its orbit space $\NAT_s$ is a  semi-analytic set  \cite{BKR}.

\medskip

\subsection{Restriction maps to the peripheral traceless character variety}
We define $P_0$ to be the space of conjugacy classes of traceless representations of the first peripheral subgroup $\Pi^0$:
$$P_0=\{\rho\in\Hom(\Pi^0,SU(2))~|~ 0=\Real(\rho(a))=\Real(\rho(b))=\Real(\rho(e))=\Real(\rho(f))\}/_{\small \rm conjugation}.$$
The space $P_0$ is called a {\em pillowcase}.  It is a   two-dimensional  real semi-algebraic set, homeomorphic to $S^2$, with four singular points, which we call {\em corners} (see Section \ref{pi0pi1}). 
Denote by $P^*_0$ the complement of the four corner points.  This is a smooth  four-punctured  two-sphere \cite{Lin}  equipped with the Atiyah-Bott-Goldman symplectic form \cite{AB, Gold}.    The smooth stratum $P^*_0$  consists of   irreducible (i.e., non-abelian) representations. The corners correspond to the abelian representations.  
Define 
$$P_1=\{\rho\in\Hom(\Pi^1,SU(2))~|~ 0=\Real(\rho(c))=\Real(\rho(d))=\Real(\rho(g))=\Real(\rho(f))\}/_{\small \rm conjugation}, $$  the traceless character variety of the other  peripheral subgroup
$\Pi_1$.

\bigskip

Restriction of homomorphisms to the peripheral subgroups  defines a map
\begin{equation}\label{rest} u_s:\NAT_s\to P_0  \times P_1,\end{equation}  
whose properties are the central topic of this article.

\medskip

  It is explained in \cite{CHKK}  that the space $\NAT_0$ is not a smooth manifold, but rather 
 a singular space obtained by crushing four singular fibers of a certain Seifert-fibered  three-manifold to points.  The restriction map
 $$u_0:\NAT_0\to P_0\times P_1$$
is not  a  Lagrangian immersion, even on top stratum.  We summarize some properties about $u_s$ for small non-zero $s$ established in \cite{CHKK} and \cite{CHK}.    
\begin{theorem}
 For small non-zero $s$,  $\NAT_s$ is a smooth compact surface of genus three.  The restriction map $u_s$  has image in the top smooth stratum $P_0^*\times P_1^*$ and is a smooth Lagrangian immersion with respect to the Atiyah-Bott-Goldman symplectic form.  The critical set of  $\operatorname{proj}_0\circ u_s:\NAT_s\to P_0^*$ is a disjoint union of four smooth circles, which separates $\NAT_s$ into two four-punctured  two-spheres.\end{theorem}

 This theorem is reproved in Section \ref{coords},
for the convenience of the reader, to establish its extension to the bypass tangle, and to indicate how it follows from  Lemma \ref{estimate2}.

\subsection{Coordinates on the pillowcase and its embedding in $\RR^3$}
 It is useful to identify $P_0$ with the quotient of the  two-torus $T=S^1\times S^1$ by the elliptic involution \begin{equation}\label{elliptic}\iota(e^{\gamma\bbi},e^{\theta\bbi})=(e^{-\gamma\bbi},e^{-\theta\bbi})\end{equation} via the map
\begin{equation}\label{TtoP0}T\to P_0, ~(e^{\gamma\bbi},e^{\theta\bbi})\mapsto \big(a\mapsto \bbi,~ b\mapsto e^{\gamma\bbk}\bbi,~ f\mapsto  e^{\theta\bbk}\bbi,~ e\mapsto
e^{(\gamma+\theta)\bbk}\bbi\big).\end{equation}
We use $(\gamma,\theta)$ to denote either 
a point in $T$ or else in its  universal cover 
 $\RR^2.$
 The image of $(\gamma,\theta)$ in $P_0$ is denoted $[\gamma,\theta]$, which equals $[-\gamma,-\theta].$ It is proven in \cite{CHK} that the Atiyah-Bott-Goldman symplectic form on $P_0^*$ is given in these coordinates by $c d\gamma\wedge d\theta$ for some non-zero constant $c$.

Our standard depiction of the pillowcase uses the $[\gamma,\theta]$ coordinates and views $P_0$ as the identification space obtained from the rectangle $[0,\pi]\times[0,2\pi]$  by folding it along $[0,\pi]\times \{\pi\}$ and identifying edges. In illustrations and our discussion we make the following conventions. The {\em bottom left} corner  is $[0,0]$, and similarly the  {\em bottom right} corner is $[\pi,0]$, the  {\em top left} corner is $[0,\pi]$ and the  {\em top right} corner is $[\pi,\pi]$. These corners are joined by the {\em bottom edge} $\{[\gamma,0]~|~\gamma\in [0,\pi]\}$ as well as the {\em top, left}, and {\em right} edges.

\medskip
An alternative way to view $P_0$ is as a real semi-algebraic subset of $\RR^3$.
The characters of the three peripheral loops $ba^{-1}$, $fa^{-1}$ and $bf^{-1}$ embed   the pillowcase $P_0$ as a singular algebraic surface in $\RR^3$.      Explicitly, if $\rho:\Pi_0\to SU(2)$, then  
$$\Real(\rho(ba^{-1}))=\cos \gamma ,~
\Real(\rho(fa^{-1}))=\cos \theta ,~ \text{ and } \Real(\rho(bf^{-1}))=\cos(\gamma-\theta)$$
are the values of these three characters on $\rho$.
They define the map    \begin{equation}
\label{curvypillow}
P_0\to \RR^3 ,~
[\gamma,\theta]\mapsto(\cos \gamma ,\cos \theta , \cos(\gamma-\theta)),\end{equation}
which {\em embeds}  $P_0$ into $\RR^3$, with image the  two-dimensional singular real semi-algebraic set 
$$\{x^2+y^2+z^2-2xyz=1, x^2\leq 1, y^2\leq 1, z^2\leq 1\}.$$   In this manifestation of $P_0$,  the corners are the four singular points $(1,1,1), (-1,-1,1), (-1,1,-1),$ and $(1,-1,-1)$.
The  bottom edge  corresponds to  the arc
 $\{(\cos\gamma, 1,\cos\gamma)~| ~0\leq \gamma\leq \pi  \}$. The  top edge  corresponds to $\{(\cos\gamma, -1,\cos\gamma)\}$,   the  left edge  corresponds to $\{(1,\cos\theta, \cos\theta)\}$ and the   right  to $ \{(-1,\cos\theta,-\cos\theta)\}$.
 A similar remark applies to the other peripheral subgroup.

 \medskip

 The following lemma, which identifies the restriction map of Equation (\ref{rest}), is an immediate consequence of the definitions.
\begin{lemma}
 \label{mapstor3} The composite 
  $\operatorname{proj}_0\circ u_s:\NAT_s\to P_0$ is determined  by the three characters:
\begin{equation}\label{charcs1}\rho\mapsto \big(\Real(\rho(b\bar a)),\Real(\rho(f\bar a)), \Real(\rho(b\bar f))\big),\end{equation}
and
$\operatorname{proj}_1\circ u_s:\NAT_s\to P_1$ is determined  by the three characters
\begin{equation}\label{charcs2}\rho\mapsto \big(\Real(\rho(c\bar d)),\Real(\rho(f\bar d)), \Real(\rho(c\bar f))\big).
\end{equation}
\end{lemma}

  \medskip

Different analytic models of the pillowcase  can have different formal tangent spaces at the four singular points.  In this article, $P_0$ is understood as the traceless $SU(2)$ character variety of $(S^2,4)$, which has a  four-dimensional formal tangent space at each corner (see e.g.  \cite[Section 4]{CHK})   whereas the hypersurface $x^2+y^2+z^2-2xyz=1$ has a  three-dimensional formal  tangent space  at each singular point.  None of the arguments in this article depend on this distinction.

\subsection{Elimination of variables and partial gauge fixing}\label{eliminating}

 The  semi-analytic space $\NAT_s$ is described above as the orbit space of an $SU(2)$ action on  $\tNAT_s$.    Following Section 4 of  \cite{CHKK}, we introduce 
 a  lower-dimensional model for $\NAT_s$, namely as a quotient of a subspace of  $\tNAT_s$ by an involution.  The following is a self-contained description of this slice-with-involution, for the convenience of the reader. 
 
 \medskip
 
 Extend  the elliptic involution  $\iota$ on $T$ of Equation (\ref{elliptic})  to $T\times \tracelessTwoSphere$  by the formula 
\begin{equation}\label{hatiota}\hat\iota(e^{\gamma\bbi},e^{\theta\bbi}, H)=
(e^{-\gamma\bbi},e^{-\theta\bbi},-\bbi H\bbi).\end{equation}

 \medskip

Consider the smooth $s$-dependent embedding
\begin{equation}\label{s-embed}L_s:T\times \tracelessTwoSphere \xrightarrow{(a,b,f,h,p,q)}  SU(2)^6, ~\end{equation}
 given by 
\begin{equation} \label{Ls} L_s(e^{\gamma\bbi},e^{\theta\bbi }, H)=(a,b,f,h,p,q)=(\bbi, e^{\gamma\bbk}\bbi,  e^{\theta\bbk}\bbi, H,  e^{s\Ima(e^{\gamma\bbk}\bbi H)}, e^{s\Ima(e^{\theta\bbk}H)}).\end{equation}  
The image of $L_s(T\times \tracelessTwoSphere)\subset SU(2)^6$ is a smooth analytic set. Moreover, since $[p,\lambda_p]=1$ and $[q,\lambda_q]=1$ for any point in this image, every point in $L_s(T\times \tracelessTwoSphere)$ determines an  $s$-perturbed traceless representation of $\Pi$. In addition, any   representation satisfying these two    conditions of Definition \ref{ptrep}  can be conjugated so that $(a,b,f, h,p,q)$ lies in   $L_s(T\times \tracelessTwoSphere)$.

To identify the intersection $L_s(T\times \tracelessTwoSphere) \cap\widetilde\NAT_s$, 
we introduce the following function.
Let
\begin{equation}\label{initial F} G(s,-)=(G_1(s,-),G_2(s,-)):T\times \tracelessTwoSphere\to\RR^2,\end{equation}
be given by 
\begin{equation}\label{initial F2} G_1(s,e^{\gamma\bbi},e^{\theta\bbi}, H)= 
\Real\left(h \ba h \bq \bp \right)=\Real(p q \bh  a \bh)\end{equation} 
and
\begin{equation}\label{initial F3} G_2(s,e^{\gamma\bbi},e^{\theta\bbi}, H) =\Real\left(  \ba h \bq \bp  \right)=\Real(pq\bh a),\end{equation} 
using Equation (\ref{Ls}).
 For any $s\in\RR$,  denote by $$\hNAT_s=(G(s,-))^{-1}(0,0)\subset T\times \tracelessTwoSphere$$ the zero set of $G(s,-)$.

\medskip

The pair of functions  $F=(F_2,F_3)=( \Real(\bp a  \bff q f), \Real(h \bp a\bff q f ))$ were introduced in Section 4 of \cite{CHKK}, and in that article, $\widehat\NAT_s$ is defined to be the preimage of zero under $F(s,-)$.   The following lemma
implies that $G(s,-)$ and $F(s,-)$ have the same zero set.

\begin{lemma}\label{shuffle}
The components of $L_s$ satisfy the relation $ \bp a  \bff q f
 =hpq\bh a,$ and
 $$G=-F.$$
 
 Moreover, the following are equivalent.
\begin{itemize}
\item The  representation determined by $L_s(e^{\gamma\bbi},e^{\theta\bbi},H)$ satisfies the $w_2$ condition,
\item $G(s, e^{\gamma\bbi},e^{\theta\bbi},H)=(0,0)$,
\item $[\ba h \bq \bp  , h]=-1$.

\end{itemize}

\end{lemma}
\begin{proof} 
 The commutator of two elements of $SU(2)$ equals $-1$ if and only if the elements are traceless and perpendicular.  Moreover, the inverse of a traceless element is its negative.

Proposition \ref{summ} implies  that $w=[\ba h \bq \bp , h]$. Since $\Real(h)=0$ this implies that $w=-1$ if and only if $\Real(\ba h \bq \bp )=0=\Real (h\ba h \bq \bp)$, that is, if and only if $G=(0,0)$.

Since $[p,bh]=1$ and $[q, f\ba h]=1$ and $\Real(h)=0=\Real(b)$ on $L_s(T\times \tracelessTwoSphere)$, $ph=h \bp,$ and
 $$ \bp a  \bff q f=\bp a \ba h q \bh a=\bp hq\bh a
 =hpq\bh a.
 $$
 It follows that 
 $$G_1=\Real\left(h \ba h \bq \bp \right)=\Real(p q \bh a \bh)=-\Real(h p q \bh a )
= -\Real(\bp a  \bff q f),$$
and
 $$G_2=\Real(\ba h \bq \bp)
 =\Real(p q \bh a)=\Real( \bh  \bp a  \bff q f)=-\Real(h \bp a\bff q f ).$$
\end{proof}

The  simple proof of the following proposition can be found in  \cite[Section 4.2]{CHKK}. 
\begin{proposition}\label{summary}  
The  involution $\hat\iota$ preserves $\hNAT_s$ and acts  freely,  and, 
when $s\ne 0$, the composite 
$$\hNAT_s\stackrel{L_s}{\hookrightarrow} \tNAT_s\to \NAT_s=\tNAT_s/SU(2)$$ induces an isomorphism of semi-analytic sets
$\hNAT_s/\hat\iota\cong \NAT_s.$\qed
\end{proposition}

  \subsection{Extension to the  bypass tangle}

 In this section we provide the counterparts of Propositions \ref{summ} and  \ref{summary} for the bypass tangle, enhanced with the perturbation curves  indicated in Figure \ref{earpertfig}.b.

The following variant of Proposition \ref{summ} is proved similarly. In the statement,  $\Pi'$ denotes the fundamental group of the complement of the enhanced bypass tangle.

\begin{proposition}\label{summ2}
The set of loops
$\{ a,b,  f,p,q, h\}$ generates $\Pi'$, and $\Pi'$ has the presentation
$$\Pi'=\langle  a,b,  f,p,q, h~|~ [p,\lambda_p], ~ [q,\lambda_q]\rangle,
$$
where  $\lambda_p=bh\text { and }\lambda_q=f\bar{a} h$ denote 
the longitudes of the perturbation curves. 
As in the earring case, the remaining labelled loops in Figure \ref{earpertfig}.b  are expressed in terms of these as 
$$c=\bq\bp b p q, ~ e=ba\bff, ~g= \bq e q =\bq ba \bff q,~d=\bc g f=\bq\bp \bb p  ba\bff q f
.$$

The inclusions of the boundary components $$ \FPS^i\subset  S^2\times I,~ i=\rout,\rin$$ induce peripheral subgroups $$\Pi^\rout=\langle a,b,e,f~|~ ba\bff \be\rangle, \qquad \Pi^\rin=\langle c,d,g,f~|~ cd\bff \bg\rangle,$$ 
with the inclusion homomorphisms $\Pi^\rout \to \Pi, ~ \Pi^\rin \to \Pi$ preserving the labeling of the generators.

\end{proposition}

  For any $s\in \RR$, we call a homomorphism $\rho\in \Hom(\Pi',SU(2))$ an  {\em $s$-perturbed traceless $SU(2)$ representation }  if  $\Real(\rho(z))=0
$  for all meridians of the (unenhanced) bypass tangle, and the perturbation condition 
$$\rho(p)=e^{s \Ima(\rho(\lambda_p))},~\rho(q)=e^{s \Ima(\rho(\lambda_q))}$$
holds for the meridians of the perturbation curves.

 \bigskip

 As before, we abuse notation and consider
evaluation on the generators $(a,b,f,p,q,h)$ as 
an inclusion $\Hom(\Pi',SU(2))\subset SU(2)^6$.
Then, for fixed $s\in \RR$,   $$\widetilde\NAT_s'\subset \Hom(\Pi', SU(2))\subset SU(2)^6$$ is defined to be the (real analytic) set of $s$-perturbed traceless $SU(2)$ representations of $\Pi'$ and
 $$\NAT_s'=\widetilde\NAT_s'/SU(2)$$ its quotient by conjugation.

  The  embedding $L_s:T\times \tracelessTwoSphere\subset SU(2)^6$ of Equation (\ref{s-embed}) meets every orbit in $\widetilde\NAT_s'$.  
We introduce the replacement $G'$, in the case of the bypass tangle,   for  the function $G:T\times \tracelessTwoSphere\to \RR^2$ of Equation (\ref{initial F}).

Recall that on $T\times \tracelessTwoSphere$, the relations $[\lambda_p,p]=1$ and $[\lambda_q,q]=1$ hold.  Just as in the proof of Lemma \ref{shuffle},  $[p,bh]=1$, $[q, f\ba h]=1$, $\Real(h)=0=\Real(b)$ on $L_s(T\times \tracelessTwoSphere)$, $ph=h \bp,$ and
 $$ \bp a  \bff q f=\bp a \ba h q \bh a=\bp hq\bh a
 =hpq\bh a.
 $$
This  implies that 
$d$ can be rewritten 
$$
d=\bq\bp \bb p  ba\bff q f=\bq\bp^2 \bb   ba\bff q f=\bq \bp h p q \bh a.
$$

Define $G'(s,-):T\times \tracelessTwoSphere\to \RR^2$ by
\begin{align}\label{F'}G'(s,e^{\gamma\bbi},e^{\theta\bbi}, h)&=(G_1'(s,-),G_2'(s,-))=(\Real( d),\Real(\bh a))\\
\nonumber &= 
(\Real(\bq \bp h p q \bh a),\Real(\bh a)).\\
\nonumber \end{align}
Furthermore, define 
$$\hNAT_s'=(G'(s,-))^{-1}(0,0)\subset T\times \tracelessTwoSphere$$ to be the zero set of $G'(s,-)$.
The easily verified counterpart of Proposition \ref{summary}  is:
\begin{proposition}\label{summary2}  
The  involution $\hat\iota$ preserves $\hNAT_s'$ and acts  freely, and,  when $s\ne 0$,  the composite 
$$\hNAT_s'\stackrel{L_s}{\hookrightarrow}\tNAT_s'\to \NAT_s'=\tNAT_s'/SU(2)$$ induces an isomorphism of semi-analytic sets
$\hNAT_s'/\hat\iota\cong \NAT_s'.$\qed
\end{proposition}

\noindent{\em Remark.} Although $\Pi$ and $\Pi'$ are
 isomorphic, their $s$-perturbed traceless character varieties $\NAT_s$ and $\NAT_s'$ are slightly different.  
 In contrast to the earring tangle, where the arc $w$ represents the obstruction to lifting     an $SO(3)$ bundle to an $SU(2)$ bundle,  traceless $s$-perturbed representations on the bypass tangle determine flat structures on the trivial bundle,  and hence the $w_2$ condition is absent in the definition of $\NAT' _s$.   The perturbation condition takes the same form (Condition (3) of Definition \ref{ptrep})  for 
 $\NAT_s$ and $\NAT_s'$. 
 
 In particular, the trivalent vertices have  different meaning in the earring and bypass markings: at a trivalent vertex on the earring, the meridian $w$ is sent to $-1$ and the meridians of the other two incoming edges are sent to traceless elements of $SU(2)$ by any representation in $\NAT_s$. In contrast, at a trivalent vertex on the bypass, all three meridians of incoming edges are sent   to traceless elements of $SU(2)$ by any representation in $\NAT_s'$.

The maps $G(s,-) $ and $G'(s,-)$, which cut out $\hNAT_s$ and $\hNAT_s'$ in $T\times \tracelessTwoSphere$, are different, but 
 our main results hold for both $u_s$ and $u'_s$, as a consequence of the calculation (Lemma \ref{estimate2})
that their second order expansions in $s$ are similar.

\subsection{The map $\NAT_s\to P_0$ induced by restriction to $\Pi_0$}

Going forward, we use ${\rm proj}$ to denote the projection from a product of spaces to a factor; which factor will be made clear by the context. 

\medskip

Recall from Proposition \ref{summ} that the two peripheral subgroups are given by 
  $\Pi^\rout=\langle a,b,e,f~|~ ba=ef\rangle$,  and $ \Pi^\rin=\langle c,d,g,f~|~ cd=gf\rangle.$ 
  These map to $\Pi$ and $\Pi'$  preserving labels, i.e., by the identical formulas for $c,d,e,$ and $g$ in  Propositions \ref{summ} and \ref{summ2}.
 
    We have the following commutative diagram:
\begin{equation} 
\begin{tikzcd}\label{lifttoP0}
\hNAT_s  
\arrow[r, hook] \arrow[d, "/ \hat \iota"] & T\times  \tracelessTwoSphere \arrow[r, "{\rm proj}"] &T   \arrow[d, "/ \iota"]\\
 \NAT_s= \hNAT_s / \hat \iota 
\arrow[r , "u_s"] 
& P_0 \times P_1 \arrow[r, "{\rm proj}"] & 
P_0=T/ \iota
\end{tikzcd}
\end{equation}
 as well as an identical diagram obtained by replacing $\NAT_s,  \hNAT_s$ by $\NAT_s',  \hNAT_s'$.

\bigskip

We summarize the discussion of this section in the following proposition.   Set $\mathbb{S}_{\bbi}^*=\mathbb{S}_{\bbi}\setminus \{\pm \bbi \}$.  The involution $\hat\iota$ acts freely on $T\times \mathbb{S}_{\bbi}^*$, so that $T\times_{\hat\iota} \mathbb{S}_{\bbi}^*$ is a smooth manifold. (The quotient $T\times_{\hat\iota} \mathbb{S}_{\bbi}$ has eight singular points with neighborhoods cones on $\mathbb{R}P^3$.)

\begin{proposition}\label{master}
  For any fixed $s$, the zero sets $\widehat\NAT_s, ~\widehat\NAT_s'$ of the two analytic functions
 \begin{equation*}
 G(s,-)=(\Real(p q \bh  a \bh), \Real(pq\bh a)), ~ G'(s,-)=(\Real(\bq \bp h p q \bh a),\Real(\bh a)),
\end{equation*}
 are analytic varieties   in $T\times \tracelessTwoSphere^*$.    These are preserved by the involution $\hat \iota$, and $L_s$ induces  isomorphisms $\widehat\NAT_s/\hat\iota\cong\NAT_s$ and $\widehat\NAT_s'/\hat\iota\cong\NAT_s'$.  Restriction to the boundary defines maps
 $$u_s:\NAT_s\to P_0  \times P_1  ~\text{ and }u'_s:\NAT'_s\to P_0  \times P_1.$$
 The composites of $u_s$, $u_s'$ with the projection to $P_0$ are induced by the projection $T\times \tracelessTwoSphere\to T$.
\end{proposition}

\medskip

 \noindent{\em Remark.} The remainder of this article analyzes $u_s$ and $u'_s$. All theorems which follow are true for both $\NAT_s$ and $\NAT_s'$.   In all cases, the calculations and proofs are simpler for $\NAT_s'$ than $\NAT_s$.  The reason is that that $G_2'$  is a simpler function than $G_2$.

\section{Expansions in $s$ and the identification of $\NAT_s,\NAT_s'$ as genus three surfaces.}\label{coords}

The main result of this section is Lemma \ref{estimate2}, which provides   asymptotic expansions in $s$ of the functions $G$ and $G'$ of Proposition  \ref{master}.  These are then used, similarly   to the argument of \cite{CHKK}, to show that 
$(0,0)$ is a regular value of $G(s,-)$ and $G'(s,-)$  for $0<|s|$ sufficiently small, so that $\hNAT_s$ and $\hNAT_s'$ are smooth surfaces.   We show that $\hNAT_s$ and $\hNAT_s'$ are compact genus five surfaces on which the involution $\hat\iota$ acts freely, with quotients $\NAT_s$ and $\NAT_s'$ genus three surfaces.

\subsection{Coordinates on $T\times\tracelessTwoSphere$ }\label{sec3.1}

We  set notational conventions and introduce coordinates $s, \nu, \tau$.  We refer the reader to Equation (\ref{Ls}) where the $s$-dependent   embedding $L_s:T\times C_{\bbi}\hookrightarrow SU(2)^6$, with coordinates $(a,b,f,h,p,q)$,   is defined.

 First, to simplify formulas, we write $$(\gamma, \theta)\in(\RR/2\pi\ZZ)^2 \mbox{ for points } (e^{\gamma \bbi}, e^{\theta \bbi}) \in T.$$   
   Second, observe that when the perturbation parameter $s$ equals zero, $G=G'=(0, \Real(\bh a))$. Hence
there exists an $s_0>0$ so that for any perturbation parameter $s\in [-s_0,s_0]$, $G_1(s,e^{\gamma\bbi},e^{\theta\bbi},h)=0$  and $G'_2(s,e^{\gamma\bbi},e^{\theta\bbi},h)=0$ each imply that  $|\Real (h\bbi)|< \tfrac {1} { 2} $.
Thus, we introduce coordinates 
$(\nu,\tau)\in [-\tfrac1 2 ,\tfrac1 2 ]\times (\RR/2\pi\ZZ)$ on this annulus in $\tracelessTwoSphere$,   and set
\begin{equation*}h(\nu,\tau)=\nu \bbi +\sqrt{1-\nu^2}e^{\tau\bbi}\bbj.
\end{equation*}

Consider the  compact domain 
\begin{equation}
\label{domainD}
\mathcal D= [-s_0,s_0]\times (\RR/2\pi\ZZ)^2 \times [-\tfrac1 2,\tfrac 1 2]\times (\RR/2\pi\ZZ).
\end{equation}
and define
$\bar L:\mathcal D\to SU(2)^6$ by
\begin{equation}\label{Lbar} \bar L (s,\gamma,\theta,\nu,\tau)= (a,b,f,h,p,q)= (\bbi,e^{\gamma\bbk}\bbi, e^{\theta\bbk}\bbi, \nu \bbi +\sqrt{1-\nu^2}e^{\tau\bbi}\bbj,  e^{s\Ima(e^{\gamma\bbk}\bbi h)}, e^{s\Ima(e^{\theta\bbk}h)}).\end{equation}
By construction, the image of $\bar L(s,-)$ lies in $L_s(T\times \tracelessTwoSphere)$ and contains $\widehat\NAT_s$,  $\widehat\NAT_s'$ for $|s|<s_0$.

The involution $\hat \iota$  corresponds in these coordinates to \begin{equation}\label{iota vs tau plus pi}(s,\gamma, \theta, \nu, \tau)\mapsto (s,-\gamma, -\theta, \nu, \tau+\pi).\end{equation}  
We denote equivalence classes by $[s, \gamma, \theta, \nu, \tau]$.

 \medskip

\noindent{\em Remark.} In the remainder of this article, some assertions may hold for $|s|<s_0$  only after  $s_0$ has been shrunk further, and correspondingly restricting the domain $\mathcal{D}$.  We opt to use the phrase ``after shrinking $s_0$ if needed,'' rather than introducing new constants.

  \subsection{The maps $\pi_0$ and $\pi_1$}\label{pi0pi1}
 Define
 \begin{equation}\label{pizero}
\pi_0:\mathcal{D}/\hat{\iota}\to P_0, ~ \pi_0 ([ s,\gamma, \theta, \nu, \tau])=[\gamma, \theta].
\end{equation}
Equivalently, using Lemma \ref{mapstor3}, this map can be written as
$$
 \pi_0([ s,\gamma, \theta, \nu, \tau])=\big(\Real(\rho(b\bar a)),\Real(\rho(f\bar a)), \Real(\rho(b\bar f))\big)=(\cos\gamma,\cos\theta,\cos(\gamma-\theta))\in P_0\subset\RR^3.
$$
Define  \begin{equation}\label{pione}
\pi_1:\mathcal{D}/\hat{\iota}\to   \RR^3, ~ \pi_1([ s,\gamma, \theta, \nu, \tau])=\big(\Real(\rho(c\bar d)), \Real(\rho(f\bar d)),\Real(\rho(c\bar f))\big) \end{equation} where $c=\bq \bp b p q $ and $d=\bq \bp \bar b p b a \bff q f $ express $c,d$ in terms of the components of $\bar L$. 
By Lemma \ref{mapstor3}, the restriction of $\pi_1$ to either $\widehat \NAT_s$  or 
$\widehat \NAT'_s$ takes values in the pillowcase
$$\{x^2+y^2+z^2-2xyz=1, x^2\leq 1, y^2\leq 1, z^2\leq 1\}\subset\RR^3.$$

 The proof of the following lemma is clear.
\begin{lemma}\label{rs=p0p1}
 The map $u_s$  equals the composite
 $$u_s:\NAT_s=\{s\}\times \NAT_s\subset \mathcal{D}/ \hat{\iota}\xrightarrow{\pi_0\times \pi_1}P_0\times P_1.$$
The identical statement holds for $u_s'$ and $\NAT_s'$. \qed
\end{lemma}
 In what follows, we abuse notation slightly write 
 $$u_s=\pi_0\times \pi_1:\NAT_s\to P_0\times P_1$$ or simply $u_s=\pi_0\times \pi_1$ when its domain $\NAT_s\subset \mathcal{D}/ \hat{\iota},$ is understood.

\subsection{Expansions in $s$}  \label{expansions in s}

The  functions $G_1, G_1':\mathcal{D}\to \RR$ vanish when $s=0$ and hence there exist analytic functions 
 $\widehat G_1,~\widehat G_1':\mathcal{D}\to\RR$  satisfying $$G_1=s\widehat G_1\text{ and }G_1'=s\widehat G_1'.$$

An important observation for the sequel is that 
for fixed {\em non-zero} $s$, the zero locus  $\hNAT_s=\{G(s,-)=0\}$  coincides with the zero locus of  $(\widehat G_1(s,-), G_2(s,-)),
$ and similarly   $\hNAT_s'$  coincides with the zero locus of  $(\widehat G_1'(s,-), G_2'(s,-)).
$

\medskip

\begin{lemma} \label{estimate2} After shrinking $s_0$ if needed, there exist analytic real valued functions $R_i$, $i=1,\dots, 6$ on $\mathcal D$  so that for all $|s|\leq s_0$, 
  \begin{align*} \widehat G_1
& = -( \sin\gamma\sin\tau -\sin\theta\cos\tau )   +2 s\cos\gamma\cos\theta +s^2R_1+G_2 R_2,\\
G_2&= (1-sR_6)\big(\nu - s \cos\gamma  +s^2R_3\big),\\
 \widehat G_1'
 &=2( -(   \sin\gamma\sin\tau -\sin\theta\cos\tau )  +2s\cos\gamma\cos\theta) +s^2R_4+G_2' R_5  , \text{ and }\\ G_2'&=\nu.\end{align*} 
\end{lemma}
\begin{proof} 
Given   analytic  functions $\phi_1,\cdots, \phi_n$ on $\mathcal{D}$, we use the  shorthand $O(\phi_1,\cdots, \phi_n)$ to denote  the ideal in the ring of analytic functions on $\mathcal{D}$ generated by the $\phi_1,\cdots, \phi_n$.

Since $p=1+s\Ima(bh)+s^2(\Ima(bh))^2+O(s^3)$ and  $q=1+s\Ima(f\ba h)+s^2(\Ima(f\ba h))^2+O(s^3)$, it follows that  
$$\Real(pq)=1+O(s^2)$$\text{ and }
$$\Ima(pq)=s(\Ima(bh + f\ba h))+ s^2\Ima(\Ima(bh)\Ima(f\ba h))+O(s^3).$$

  Then, \begin{align*}G_2=\Real(pq\bh a)&= \Real(pq)\Real(\bh a) - \Real(\Ima(pq)ha)\\
&=-\Real(ha) -s\Real((\Ima(bh + f\ba h)ha) + O(s^2)\\
&=-\Real(ha) -s\Real((bh + f\ba h)ha) +s\Real(bh+f\ba h)\Real(ha)+ O(s^2)
\\
&=-\Real(ha)(1-s\Real(bh + f\ba h)) -s\Real((bh + f\ba h)ha) +  O(s^2)\end{align*}
Set $R_6=\Real(bh + f\ba h):\mathcal{D}\to \RR$.  Since $\Real(ha)=-\nu$ and
$$\Real((bh + f\ba h)ha)=\Real(bh^2a -f)=-\cos\gamma,$$
$$G_2 =\nu (1-sR_6)-s\cos\gamma + O(s^2)=(1-sR_6)(\nu-s\cos\gamma) + O(s^2)
.$$
 For $s$ small enough,  $1-sR_6$ is non-zero on $\mathcal{D}$, and hence invertible. Therefore, shrinking $s_0$ if needed, 
$$G_2=(1-sR_6)(\nu -s\cos\gamma + O(s^2)),$$
 establishing the second assertion and showing that  
 $$\nu=s\cos\gamma+ O(G_2,s^2).$$
This implies that $\sqrt{1-\nu^2}=1+O(G_2,s^2)$ and therefore the component function $h$ of Equation (\ref{Lbar}), viewed as a quaternion-valued function  $h:\mathcal{D}\to \mathbb{H}$, is given by
\begin{equation}\label{ache} h=s\cos\gamma \bbi + e^{\tau\bbi}\bbj +O(G_2,s^2).\end{equation}
It easily follows that
$$s h \bbi h=s\bbi - 2 s^2 \cos\gamma e^{\tau\bbi}\bbj +O(G_2, s^3).
$$
Hence
\begin{align*}
G_1&=\Real(pqh\bbi h)=\Real(\Ima(pq)h\bbi h) \\
&=s \Real((\Ima(bh)+ \Ima(f\ba h)) ~h \bbi h ) + 
s^2\Real(\Ima(bh) \Ima(f\ba h) h \bbi h )+O(s^3)\\
&=s \Real((\Ima(bh)+ \Ima(f\ba h)) \bbi  ) \\&\qquad -2s^2\cos\gamma \Real((\Ima(bh)+ \Ima(f\ba h))  e^{\tau\bbi}\bbj)\\
&\qquad\qquad  + s^2\Real(\Ima(bh) \Ima(f\ba h)\bbi)+O(G_2,s^3)\\
&=\big(-s(\sin\gamma\sin\tau-\sin\theta\cos\tau)-
s^2\cos\gamma\cos\theta\big)\\
&\quad - 2s^2\cos\gamma(-\cos\theta)\\
&\qquad +s^2\big(\cos\gamma\cos\theta\big)
+O(G_2,s^3)\\
&=-s(\sin\gamma\sin\tau-\sin\theta\cos\tau)+2s^2\cos\gamma\cos\theta +O(G_2,s^3)
\end{align*}
The penultimate equality is a straightforward calculation of each of the three terms, replacing each occurrence of $h$ by the right side of Equation (\ref{ache}). Dividing by $s$ yields the first assertion.

\medskip

We turn now to the third and fourth assertions.
 In the coordinates on $\mathcal{D}$, 
 $$G_2'=\Real(h\ba)=\nu,$$ 
and therefore $h=e^{\tau\bbi}\bbj +O(G'_2)$. Since $\Ima(bh)$ is perpendicular to $h$, $hp= he^{s\Ima(bh)}=\bp h$.  
Similarly,  calculation yields that $qha=ha \bq+ O(G_2')$. Therefore,
\begin{align*}
G_1'&= \Real(\bq \bp h p q \bh a)\\
&=-\Real((\bp h p)(q h a \bq)) \\ 
& =-\Real(\bp^2 h^2 a \bq^2) +O(G_2')\\ 
&  =\Real( \bq^2\bp^2  a )+O(G_2')\\
&=2s\big(\Real(bha)+\Real(f\ba h a)\big) -4s^2\Real(\Ima(f\ba h) \Ima(bh)a)+O(s^3,G_2')\\
&=-2s(\sin\gamma\sin\tau-\sin\theta\cos\tau)\\
&\qquad+ 4s^2\Real(\Ima((\cos\theta + \sin\theta\bbk)h a)\Ima((\cos\gamma\bbi +\sin\gamma\bbj)h)a )+O(s^3,G_2')\\
&=-2s(\sin\gamma\sin\tau-\sin\theta\cos\tau)+ 4s^2\cos\theta\cos\gamma+O(s^3,G_2').
\end{align*}
Dividing by $s$ yields the desired calculation of $\widehat G_1'$.

\end{proof}

  To provide parallel arguments that $\NAT_s$ and $\NAT_s'$ are closed genus three surfaces, we replace $G$ and $G'$ with functions $\bf G$ and ${\bf G}'$, respectively, which have the same zero sets as $G,G'$, but with the convenient property that $\bf G$ and ${\bf G}'$ certain Taylor expansions in $s$ agree.    
Define \begin{align}\label{phis}
{\bf G}(s, \gamma, \theta, \nu , \tau) &=  (\widehat G_1-G_2R_2,G_2(1+sR_6))    \\
&=\big(\sin\theta\cos\tau -\sin\gamma\sin\tau    +2s \cos\gamma\cos\theta + s^2 R_1,  \nu - s \cos\gamma+ s^2R_3 \big) \nonumber    
\end{align}
and  
\begin{align}\label{phisp} 
{\bf G}'(s, \gamma, \theta, \nu , \tau) &=  (\tfrac1 2 \widehat G_1'-\tfrac 1 2 G_2'R_5,G_2')    \\
&= \big(\sin\theta\cos\tau-\sin\gamma\sin\tau +2s \cos\gamma\cos\theta +  s^2 \tfrac{R_4}2,  \nu\big)  \nonumber \end{align}

For any fixed $s$, we'll let  $${\bf G}_s={\bf G}(s,-),~{\bf G}' _s={\bf G}'(s,-)$$
and set $$V_s={\bf G}_s^{-1}(0,0), ~V'_s=({\bf G}_s ')^{-1}(0,0).$$
By construction, 
\begin{equation}\label{Veesubs} V_s=\hNAT_s \text{ and }V'_s=\hNAT'_s \text{  for all non-zero } s.\end{equation}

Setting $s=0$ in Equations (\ref{phis}) and (\ref{phisp}), one sees that \begin{equation}\label{Phi}
{\bf G}_0(\gamma,\theta,\tau,\nu)=(\sin\theta\cos\tau-\sin\gamma\sin\tau,\nu)={\bf G}'_0(\gamma,\theta,\tau,\nu).\end{equation}

\medskip

 The following corollary overlaps  with  Theorem 6.1 from \cite{CHKK} and extends the result to $ \NAT_s'$. 
  In preparation for its statement, define 
\begin{equation}
\label{localtriv}
\mathbb{V}=\{(s,v)~|~ |s|\leq s_0,~ v\in V_s\}={\bf G}^{-1}(0,0)\subset [-s_0,s_0]\times T\times \mathbb{S}_{\bbi}^*.
\end{equation}
and define  $\mathbb{V}'$ similarly.  The involution $\hat\iota$ acts freely on $[-s_0,s_0]\times T\times \mathbb{S}_{\bbi}^*$, preserving $\mathbb{V}$.

\begin{corollary}\label{genus53} 
The  composite map  
\begin{equation} 
V_0\hookrightarrow T\times \tracelessTwoSphere\xrightarrow{{\rm proj}} T \label{unperturbed crit fibers} \end{equation} 
forms a trivial  two-fold covering space over $T^*=T\setminus\{\text{fixed points}\}$.  The preimage of each fixed point is a smooth circle, and the union of these four circles is the critical set of the map (\ref{unperturbed crit fibers}).
The zero locus $V_0=V_0'={\bf G}_0^{-1}(0,0)$ is a smooth, closed, connected, orientable surface  of genus five.  

  After shrinking $s_0$ if needed,   for all $0<|s|\leq s_0$,
$\hNAT_s$ and $\hNAT_s'$ are smooth, closed, connected, orientable surfaces  of genus five, and $\NAT_s, \NAT_s'$ are orientable surfaces  of genus three,   and the projection $\operatorname{proj}:\mathbb{V}/\hat\iota\to [-s_0,s_0]$ is a smooth trivial genus three surface bundle with fiber $V_0/\hat\iota$ over $0$ and $V_s/\hat\iota=\NAT_s$ for $s\ne 0$.
\end{corollary}

\begin{proof}   
The  differential of ${\bf G}_0$  is 
\begin{equation}\label{diffPhi}
d{\bf G}_0=
\begin{pmatrix} -\cos\gamma \sin\tau&  \cos\theta\cos\tau &-\left( \sin\gamma\cos\tau +\sin\theta \sin\tau\right) &0 \\ 
0 &0& 0& 1
\end{pmatrix}.\end{equation}

If ${\bf G}_0(\gamma,\theta, \tau,\nu)=0,$ then $\sin\theta\cos\tau-\sin\gamma\sin\tau=0$, 
 which implies that  one of $\cos\gamma \sin\tau$, $\cos\theta\cos\tau$, or $\sin\gamma\cos\tau +\sin\theta \sin\tau $ is non-zero.  It follows that $(0,0)$ is a regular value for ${\bf G}_0$ and  $V_0$  is a smooth surface and a closed subset of  $T\times S^1\times [-\frac 1 2, \frac 1 2 ]$ which misses its boundary and has trivial normal bundle. Thus $V_0$  is closed, compact,  and orientable.
 It follows from  straightforward differential topology arguments that after perhaps shrinking  $s_0$, the composite
$$\mathbb{V}\subset [-s_0,s_0]\times T\times \mathbb{S}_{\bbi}^*\to [-s_0,s_0]$$
is a trivial  surface bundle.  

\medskip 

The  equation ${\bf G}_0(\gamma,\theta,\tau,\nu)=(0,0)$ is equivalent to 
 $$\nu=0,~ (\sin^2\gamma+\sin^2\theta)^{1/2}   (\cos\tau,\sin\tau)=\pm  (\sin\gamma,\sin\theta) .$$ 
For fixed  $(\sin\gamma,\sin\theta)\ne(0,0)$,
 there are precisely two solutions:  $(\nu,\tau)=\pm(0, ~\tau(\gamma,\theta))$, where
 \begin{equation}\label{sect}
  (\cos(\tau(\gamma,\theta)),\sin(\tau(\gamma,\theta)))=\pm \frac{(\sin\gamma,\sin\theta)}{\sqrt{\sin^2\gamma+\sin^2\theta}}.
\end{equation}
On the other hand, if   $(\sin\gamma,\sin\theta)=(0,0)$, there is precisely a circle of solutions, namely
$$
\nu=0,~ \tau \text{ arbitrary. }
$$
This proves the last statement in the corollary.  

  Hence the smooth, compact surface $V_0$ 
 is the union of two open disjoint four-punctured tori and four circles, the preimages of the fixed points in $T$. It follows that $V_0$ has Euler characteristic negative eight. The continuous path 
 $$
 \gamma=\arcsin(t),  \theta=0, \tau=0, \nu=0,~~ t\in (-\tfrac1 2 , \tfrac1 2 )$$
 lies in $V_0$ and its endpoints lie in different sheets of the covering space.  Hence $V_0$ is connected, and therefore of genus five.
 It now follows by the Riemann-Hurwitz formula that the quotient of $V_0$ by the orientation-preserving free involution $\hat \iota$ has genus three.   We conclude that $\mathbb{V}/\hat\iota\to [-s_0,s_0]$ is a smooth trivial genus three surface bundle.

  Exactly the same argument holds  for $\NAT_s',V_s', \mathbb{V'} $. 
  \end{proof}

 \section{restriction to $P_0^*$} We continue our
analysis of the restriction maps  
   $u_s=\pi_0\times \pi_1:V_s/\hat{\iota}=\NAT_s\to P_0\times P_1$ (resp. $u'_s$) in this section by proving that   for small non-zero  $s$, the map $\pi_0:\NAT_s\to P_0$  misses the corners and has only embedded fold singularities in the sense of Whitney \cite{Whitney}.  In Section \ref{bifolds}, we define appropriate terminology to describe the singularities.  In Section \ref{sect4.2}, we analyze the map to $P_0 $.  The corresponding analysis for the projection to $P_1 $ is carried out  in Section \ref{punchline}.

  \subsection{Fold singularities and bifold Lagrangian immersions of a surface into a product of surfaces.}  
\label{bifolds}

We remind the reader of the definition of a fold singularity and introduce  the concept of {\em bifold Lagrangian immersions} of a surface into a product of two symplectic surfaces.
We begin by recalling aspects of Whitney's theory about smooth maps between surfaces  \cite{Whitney}.  
\medskip

If  $f:F\to G$ is a smooth map between surfaces, we say $x\in F$ is a {\em singularity} if  $\ker df_x\ne 0$.  The map $f:F\to G$ is said to {\em have only fold singularities} provided  
the 1-jet of $f$ misses the rank zero 1-jets and meets the submanifold of rank one 1-jets
transversely, and  the set of singularities, which in this case is called the {\em fold locus}, is a smooth one-dimensional submanifold of $F$ which is immersed by $f$.  
These conditions can be restated as  \begin{enumerate}
\item at each $x$ in the fold locus, $\rank(df_x)=1$,
\item $0$ is a regular value of $\det(df)$, and
\item $f$ immerses the fold locus.
\end{enumerate}
  
  The terminology {\em fold} comes from Whitney's theorem \cite{Whitney} which asserts that if these three conditions hold at a singular point $x$, then there exist local  coordinates around $x\in F$ and $f(x)\in G$   such that $f$ takes the form $$(u,v)\mapsto (u^2,v).$$  The  fold locus  corresponds to the set where $u=0$ in this local form.    
  
 If  $f$   furthermore  self-transversely immerses the fold locus into $G$, then   $f$ is {\em stable}, that is, every smooth map $C^\infty$ close to $f$ is equal to $h_G\circ f\circ h_F$ for some diffeomorphisms $h_F:F\to F, h_G:G\to G$. These facts are stated and proven carefully in \cite[Section III.4]{GG}.  Note that if $f$ has only fold singularities and $f$ embeds the fold locus, then $f$ is stable; in this situation, we say $f$ has only {\em embedded fold singularities}, 
and we call the image of the fold locus under $f$ the {\em fold image}.

\bigskip

  Let $F_0,F_1$ and $F$ be   smooth  surfaces,  
 and let $f=f_0\times f_1:F\to F_0\times F_1$ be a smooth map.
Then  $f$ is called a {\em bifold immersion}  if $f$ is an immersion with the property that  $f_0$ and $f_1$ have only fold singularities, and their fold loci coincide.  For example, the map $(u,v)\mapsto (u,v^2, u+v, v^2)$ is a bifold immersion (in fact, embedding) $\RR^2\to \RR^2 \times \RR^2$, with fold locus $\{(u,0)\}$ and embedded fold image $\{(u,0,u,0)\}.$

 Not every   immersion into a product of surfaces is a bifold immersion.  For example, $(u,v)\mapsto (u^4, v, u^4,u+v)$ is an  embedding of $\RR^2$ into $\RR^2 \times \RR^2$
which is not a bifold immersion because $0$ is not a regular value of $\det df_0$ or $\det df_1$. Another interesting non-bifold is the immersion $(u,v)\mapsto (u\cos v, u\sin v, \tfrac{u^2}{2}, u+v)$ whose projection to the first $\RR^2$  is constant, and hence not an immersion, on the singular set $u=0$, while its projection to the second $\RR^2$ has only embedded fold singularities, with fold locus $v=0$.
We note that  these examples   are Lagrangian immersions into $\RR^2\times \RR^2$ with respect to the symplectic form $ -dx_1\wedge dy_1 + dx_2 \wedge dy_2$.  

\medskip

 While the examples above demonstrate that a Lagrangian immersion of a surface into a product of symplectic surfaces need not be a bifold immersion, the following lemma shows that the Lagrangian immersion property does put constraints on the behaviors of the factor maps.  Given symplectic surfaces $(F_0,\omega_0)$ and $(F_1,\omega_1)$, equip $F_0\times F_1$ with the symplectic form $-{\rm proj}_0^*(\omega_0)+{\rm proj}_1^*(\omega_1).$ We use the abbreviated notation $F_0^-\times F_1$.

\begin{lemma}\label{folding} 

Suppose  $f=(f_0, f_1):F\to F_0^-\times F_1$ is a Lagrangian immersion of a surface into a product of symplectic surfaces.  

Then,
for each $x\in F$, the derivatives $d  f_{0,x}$ and $df_{1,x}$ have equal non-zero rank. For  those  $x\in F$ satisfying 
 $\text{\rm rank } df_{0,x} =1$, the kernels of $df_{1,x}$ and $df_{0,x}$ are transverse.

\end{lemma}

\begin{proof}
Since $f$ is an immersion, the differential of $df_x$ is injective and maps the tangent space $T_xF$ to a subspace of the direct sum of subspaces $T_{f_0(x)}F_0 \oplus T_{f_1(x)}F_1$. Hence:  $$
  0=\ker df_x=\ker df_{0,x}\cap \ker df_{1,x}.
  $$
 
 If $df_{0,x}$ has rank zero or one, then the Lagrangian hypothesis implies that  $$0=f^*(-{\rm proj}_0^*(\omega_0)+{\rm proj}_1^*(\omega_1))_x=f_1^*(\omega_{F_1})_x,$$ and so $df_{1,x}$ cannot be an isomorphism, and similarly with the subscripts $0,1$ reversed.
Therefore, either  $df_{0,x}$ and  $df_{1,x}$ are both isomorphisms, or both have rank one, with transverse kernels.
\end{proof}

  Although it is not used in the present article, we note that  \cite{Zhang}  proves that each Lagrangian immersion of a surface into a product of surfaces  is Lagrangian regular homotopic to one whose compositions with the projection to the factors are stable (in the sense of Whitney).

\subsection{The critical set of 
$  \NAT_s, \NAT_s '\xrightarrow{{\pi_0}} P_0$ for small non-zero $s$}
\label{sect4.2}

 The projection map $V_0\to T$ is complicated, as seen in  Corollary \ref{genus53}; the fibers over the four  fixed points are circles, whereas over other points the map is a trivial   two-fold cover.  We prove below that when $|s|\ne 0$ is small, the map $\widehat \NAT_s=V_s\to T$ has only embedded fold singularities  
  and, furthermore, its image misses the four fixed points.  Composing with the quotient map $T\to P_0$ of the involution $\iota$, one concludes that the  map $\pi_0:\NAT_s\to P_0$ misses the four corner points and has only embedded fold singularities.   Since the arguments only depend on the first order Taylor expansion of the first component of ${\bf G}_s$, the arguments work equally well for ${\bf G}' _s$ with no modification.

 \medskip
 
      We produce coordinate charts on four open subsets of $T\times \tracelessTwoSphere$ that make it easier to analyze  the tubular neighborhoods of the four circles in $V_0$ which are the preimages of the fixed points in $T$.  
 
  To motivate the change of coordinates, note that when $s=0$, all the points in $V_0$ satisfy $\nu=0$. Consider the remaining conditions on  $(\gamma, \theta, \tau)$ in the three-torus.  The four two-tori in the three-torus defined by the equations $\gamma=\pm \tfrac \pi 2$ and $\theta=\pm \tfrac \pi 2$ decompose the three-torus into four disjoint solid tori.  One is 
  $$ \{ (\gamma, \theta, \tau) \mid |\gamma|<\tfrac \pi 2, |\theta | < \tfrac \pi 2 \},$$ and the others are obtained from this one by shifting $\gamma$ or $\theta$  by $\pi$.  The  formula $(\gamma, \theta, \tau)\mapsto (\sin \gamma, \sin \theta, \tau)$ defines a diffeomorphism from each of these solid tori to $(-1,1)\times (-1,1)\times \RR/2\pi \ZZ$.    The surface  $V_0$ intersects each of these solid torus in an annulus.  From Equation (\ref{Phi}), under these diffeomorphisms, these four annuli correspond to twisted annuli 
  $$\{ (u,v,\tau) \mid -u\sin \tau  + v\cos \tau  =0\}=\{ (u,v,\tau) \mid (u,v)\perp (-\sin \tau, \cos \tau) \}.$$  Our new coordinate system on each of these solid tori will be of the form $(x,y,\tau)$ where the twisted annulus is defined by the equation $y=0$.    
  \medskip 
   
 It is convenient in the calculations below to index the four solid tori by pairs of signs, \begin{equation}(\ep_\gamma,\ep_\theta)\in\{\pm 1\}^2 \text{ such that } \cos\gamma=\ep_\gamma \sqrt{1-\sin^2\gamma}\text{ and } \cos\theta=\ep_\theta \sqrt{1-\sin^2\theta}\label{cloverleaf}, \end{equation} on the solid torus.  
 
  Specifically,   on each solid torus, define   functions $x$ and $y$ of $(\gamma, \theta, \tau)$ by 
\begin{equation}
\begin{pmatrix}
 x\\y
\end{pmatrix}
=
\begin{pmatrix}
\cos\tau&\sin \tau\\
-\sin \tau&\cos\tau 
\end{pmatrix}
\begin{pmatrix}
 \sin\gamma \\ \sin\theta
\end{pmatrix}=R_{-\tau}
\begin{pmatrix}
 \sin\gamma \\ \sin\theta
\end{pmatrix},
\end{equation} 
where $$
R_t=
\begin{pmatrix}
\cos t&-\sin t\\
\sin t &\cos t
\end{pmatrix}
$$
  Then the map $(\gamma, \theta, \tau)\mapsto (x,y,\tau)$ is a diffeomorphism from each of the open solid tori in $T\times \RR/2\pi \ZZ$ to an open solid torus in $\RR^2 \times \RR/2\pi \ZZ$.  (In terms of the coordinates $(x,y,\tau,\nu)  $, the zero set of Equation (\ref{Phi}) is given by $\nu=0, y=0$.)  
   The inverse diffeomorphism expresses $\gamma $ and $\theta$ in terms of $x,y,\tau$, namely:
\begin{equation}\label{locinv}
 \begin{pmatrix}
\gamma\\ \theta
\end{pmatrix}
=A\circ
R_{\tau}
\begin{pmatrix}
x  \\ y
\end{pmatrix},
\end{equation} 
with  $A=\arcsin\times\arcsin$ interpreted appropriately depending on $(\epsilon_\gamma, \epsilon_\theta)$.  

\medskip
Going forward, fix $(\epsilon_\gamma, \epsilon_\theta)\in \{\pm 1\}^2$.  On the part of $\mathcal D$ where $(\gamma, \theta, \tau)$ is in the corresponding solid torus,  Equation (\ref{phis})  implies that, in the new coordinates,
 $${\bf G}(s, x,y,\tau, \nu) =  ( y ,\nu)+sR$$
  for some analytic function $R(s,x,y,\tau, \nu)$.
 The analytic implicit function theorem now implies the following lemma. 
\begin{lemma}\label{LP}  After shrinking $s_0$ if needed, 
 for each  $(\epsilon_\gamma, \epsilon_\theta)\in \{ \pm 1\}^2$, there exist an $x_0>0$,   and two analytic functions 
 $$Y(s,x,\tau),N(s, x,\tau):[-s_0,s_0]\times [-x_0 ,x_0] \times \RR/2\pi \ZZ\to \RR,
 $$
 so that, for all $|s|\leq s_0$ and $|x|\leq x_0$,  
 $${\bf G}_s(x,y,\tau,\nu)=0 \text{ if and only if } (y,\nu)=(Y(s,x,\tau), N(s, x,\tau)).$$
 
Moreover, there exists an analytic function     $R_6 (s,x,\tau)$  so that 
$$Y(s,x,\tau) +2s\cos \gamma(x, Y(s,x,\tau),\tau) \cos \theta(x, Y(s,x,\tau),\tau) = s^2R_6 (s,x,\tau). 
$$
\qed
\end{lemma}

We denote by $k_s:V_s\to T$ the composition map in the following diagram 
\begin{equation} 
\begin{tikzcd}\label{perturbed crit fibers}
V_s ={\bf G} _s^{-1}(0,0)
\arrow[r, hook] \arrow[rr, bend right, "k_s"] & T\times  \tracelessTwoSphere \arrow[r, "{\rm proj}"] &T   
\end{tikzcd}
\end{equation}

 When $s= x =0$, for the four $(\epsilon_\gamma, \epsilon_\theta)$ choices, Lemma \ref{LP} describes a parameterization by $\tau\in \RR/2\pi \ZZ$ of the circle in $V_0$ projecting to the fixed point in $T$  corresponding to $(\epsilon_\gamma, \epsilon_\theta)$.  Specifically, these four circles lie in the critical set of $k_0$, but $k_0$ is a covering map away from these circles, with each sheet of the covering a graph of a smooth function.  Thus the critical set of $k_0$ is precisely the union of these four circles.

  We now identify the critical set of $k_s$, for $|s|$ small but nonzero,  
and identify the image of this critical set in $T$.   Recall that the smooth family $V_s={\bf G}_s^{-1}(0,0)$ coincides with $\hNAT_s$ when $s\ne 0$. This allows us to analyze the map $\pi_0:\NAT_s \to P_0$, due to the commutative diagram (\ref{lifttoP0}).

\begin{proposition}\label{thmfold}  After shrinking $s_0$ if needed, then  the image of the composition $k_s$ in Equation (\ref{perturbed crit fibers})
misses the four elliptic fixed points for all $0<|s|\leq s_0$.  Moreover, the critical set of this map is a disjoint union of four circles, transversely cut out by the 1-jet of the map $k_s$.     The four critical circles are embedded   under $k_s$ with images  that are simple closed curves encircling  the four elliptic fixed points.   In other words, $k_s$ has  only embedded fold singularities and the four circles making up the fold image enclose the four elliptic fixed points.

\end{proposition}
\begin{proof}  Because $k_0$ is a local diffeomorphism except along the critical circles, it is sufficient to work in neighborhoods of the critical circles.  
 Consider the pair    $(\epsilon_\gamma, \epsilon_\theta)\in \{ \pm 1\}^2$ associated to a fixed point.   Referring to Equation (\ref{locinv}), since each of the four branches $A$ is a local diffeomorphism near $(0,0)$, it suffices   to prove that, for small enough $s$,  
$$\Gamma_s:[-x_0,x_0]\times S^1\to \RR^2,~ \Gamma_s(x,\tau)=
R_{\tau}\begin{pmatrix}
x  \\ Y(s,x,\tau)
\end{pmatrix}
=:\begin{pmatrix}
\sin\gamma(s,x,\tau)  \\ \sin\theta(s,x,\tau)\end{pmatrix}
$$
has  	a circle fold locus which embeds in $\RR^2$ enclosing the origin.

\medskip

The second assertion in Lemma \ref{LP}  implies that \begin{equation} \label{when s is zero} Y(0,x,\tau)=0 \text{ for all }(x,\tau).\end{equation}  Moreover,  $\cos(\gamma(0,0,\tau))\cos(\theta(0,0,\tau))=\epsilon_\gamma \epsilon_\theta= \pm 1$ for all $\tau$, so that   $\tfrac{\partial Y}{\partial s}(0,x,\tau)\ne 0$ 
for all $(x,\tau)$.  Hence for all small enough non-zero $s$, $Y(s,x,\tau)\ne 0$. This implies the first statement in Proposition \ref{thmfold}.

\medskip

We turn next to the examination of the critical set  of $k_s$, restricted to the solid torus indexed by $(\ep_1,\ep_2)$.   Denote this critical set by $C_s^{\epsilon_\gamma,\epsilon_\theta}$. By definition, $C_s^{\epsilon_\gamma,\epsilon_\theta}$  is the zero set of  
\begin{equation}\label{crit2}
\det(d\Gamma_s)=\det(\tfrac{\partial \Gamma_s}{\partial x},\tfrac{\partial \Gamma_s}{\partial \tau})=
\det
\left(R_{\tau}\begin{pmatrix} 1 &-Y \\  \tfrac{\partial Y}{\partial x}& \tfrac{\partial Y}{\partial \tau} + x \end{pmatrix}\right)=\tfrac{\partial Y}{\partial \tau} + x + Y \tfrac{\partial Y}{\partial x}.
\end{equation}
Equation (\ref{when s is zero})  implies that $\tfrac{\partial{ \left( \det(d\Gamma_s)\right)}}{\partial x}=1$ when $s=0$, so $0$ is a regular value of $\det(d\Gamma_s)$  for all small enough $s$.  Equivalently, the 1-jet of $\Gamma_0$  cuts out $C_0^{\epsilon_\gamma,\epsilon_\theta}$ transversely (see Section \ref{bifolds}).  
The implicit function theorem implies the existence of an analytic function $\widehat X(s,\tau)$, defined for all $s$ small enough and all $\tau\in S^1$,  so that $$\det(d\Gamma_s)(s,x,\tau)=0\text{ if and only if }x=\widehat X(s,\tau).$$  
 Shrink $s_0$ if needed so that $0$ is a regular value of $\det(d\Gamma_s)$ and $\widehat X(s,\tau)$ is defined for all $|s|\leq s_0$.

Set $$\widehat Y(s,\tau)=Y(s, \widehat X(s,\tau),\tau).$$
Then, for fixed $s$ with $|s|\leq s_0$, the map 
\begin{equation} \label{param critical circ}
e^{\tau\bbi}\mapsto(s,\widehat X(s,\tau),\widehat Y(s,\tau),\tau, N(s,\widehat X(s,\tau),\tau))\in V_s
 \end{equation} 
 parameterizes the critical circle $C_s^{\epsilon_\gamma,\epsilon_\theta}$. 
It follows that the restriction of $\Gamma_s$ to the critical circle $C_s^{\epsilon_\gamma,\epsilon_\theta}$ is parameterized as a map $\alpha_s:S^1\to T$ and expressed in the local coordinates $(\sin\gamma,\sin\theta)$ near the fixed point as
 \begin{equation}\label{416}\alpha_s(e^{\tau\bbi})=R_{\tau}\begin{pmatrix}
\widehat X(s,\tau)  \\ \widehat Y(s,\tau)
\end{pmatrix}.\end{equation}

We note here that, without any further analysis, the critical set when $s=0$ is given by $x=y=0$, parameterized by $\tau$, so \begin{equation} \label{easy part} 
\widehat X(0,\tau)=0 \text{ and } \widehat Y(0,\tau)=Y(0,\widehat X(0,\tau), \tau)=0.\end{equation}

 We use the following  lemma to finish the proof of Proposition \ref{thmfold}.

\begin{lemma}\label{fuglylemma}  Given any $(\epsilon_\gamma, \epsilon_\theta)\in \{ \pm 1\}^2$  and $|s|\leq s_0$, there exist analytic functions $T_1(s, \tau)$ and $T_2(s,\tau)$  so that 
$$\widehat X= s^2T_1 \text{ and } \widehat Y= -  2 \ep_\gamma\ep_\theta s + s^3T_2.$$   \end{lemma}
\begin{proof}

From the formula in Equation (\ref{locinv}) for $\gamma$ and $\theta$ in terms of $x,y,\tau$, 
$$\gamma(x,Y(x,s,\tau),\tau)=\ell_1 \pi +O(x,s) \mbox{ and }\theta(x, Y(x,s,\tau),\tau)=\ell_2 \pi +O(x,s) \mbox{ for some } \ell_i\in \ZZ.$$    Therefore, $$\cos \left( \gamma(x,Y(x,s,\tau),\tau)\right)  \cos \left( \theta(x, Y(x,s,\tau),\tau)\right) = \epsilon_\gamma \epsilon_\theta +O(x^2, sx, s^2).$$  It follows from the implicit formula for $Y(s,x,\tau)$ in Lemma \ref{LP} that \begin{equation} \label{better Y exp} 
Y(x,s,\tau)=-2s\epsilon_\gamma \epsilon_\theta +O(sx^2, s^2 x, s^3),\end{equation}  which implies that both $\tfrac{\partial Y}{\partial x}(s,x,\tau)$ and $\tfrac{\partial Y}{\partial \tau}(s,x,\tau)$ are $O(sx, s^2)$.

Equation (\ref{easy part}) implies that   
$\widehat X(s,\tau)=O(s)$, from which we conclude that $$\tfrac{\partial Y}{\partial x}(s,\widehat X(s,\tau),\tau)=O(s^2) \mbox{ and }\tfrac{\partial Y}{\partial \tau}(s,\widehat X(s,\tau),\tau) =O(s^2).$$ Using the defining equation for $\widehat X$ (i.e., Equation (\ref{crit2}) equal to zero), it follows that 
$$\widehat X(s,\tau)=-\tfrac{\partial Y}{\partial \tau}(s,\widehat X(s,\tau),\tau) - Y(s,\widehat X(s,\tau),\tau) \tfrac{\partial Y}{\partial x}(s,\widehat X(s,\tau),\tau)=O(s^2).$$ Also, Equation (\ref{better Y exp}) now implies 
$$ \widehat Y(s,\tau) = Y(s, \widehat X(s,\tau), \tau) = -2s \epsilon_\gamma \epsilon_\theta + O(s^3).$$

\end{proof}

With Lemma \ref{fuglylemma} in place,  Equation (\ref{416}) can be written in  the form
\begin{equation}\label{eq417}\alpha_s(\tau)=2 s \ep_\gamma\ep_\theta\begin{pmatrix}
-  \sin \tau    \\ \cos \tau  \end{pmatrix}+s^2 R_{\tau} \begin{pmatrix}T_1\\ sT_2\end{pmatrix}.
\end{equation}

By further shrinking $s_0>0$ if necessary, $0<|s|\leq s_0$ implies
\begin{enumerate}
\item $\alpha_s$ misses the fixed     point,   and 
\item the smooth map
$$\tfrac{\alpha_s(\tau)}{\|\alpha_s(\tau)\|}=\tfrac{\alpha_s(\tau)/s}{\|\alpha_s(\tau)/s\|}:S^1\to S^1$$
is defined and is a diffeomorphism.
 \end{enumerate}

Therefore, under the mapping $k_s:V_s\to T$, the critical circle $C_s^{\ep_\gamma,\ep_\theta}$ encircles the fixed point corresponding to $(\ep_\gamma,\ep_\theta) $ via   Equation (\ref{cloverleaf}).
This concludes the proof of Proposition \ref{thmfold}. 
\end{proof}

 In the following corollary, a {\em ray emerging from a corner} refers to  a path in $P_0$ of the form $[\gamma(t),\theta(t)]=[\ell_1\pi+t x,\ell_2\pi+ty]$, $t\in[0,\tfrac \pi 2 )$, $\ell_i\in \ZZ, (x,y)\in S^1$.

\begin{corollary} \label{JamesWattdied} After shrinking $s_0$ if needed, for all $s$ with $0<|s|\leq s_0$, the map $$ \pi_0:\NAT_s \to P_0$$ misses the corners  and  has only embedded fold singularities.  The fold locus consists of four critical circles. Each corner is enclosed by its corresponding   the  fold image   circle.  Any ray in $P_0$ emerging from a corner intersects the image circle encircling that corner transversely. 
The restriction of $\pi_0$ to the complement of the fold locus is a trivial  two-fold cover of its image in $P_0^*$.
\end{corollary} 
\begin{proof}  The fact that $\pi_0$ misses the corners follows from the fact that $k_s$ misses the  fixed points.   Given that the quotient $T^*=T\setminus \{ \mbox{fixed points} \} \to P_0 ^*=T^*/\hat\iota$ is a local diffeomorphism, it follows from the diagram (\ref{lifttoP0}) that $\pi_0$ has only   fold singularities.

The diagram: \begin{equation} 
\begin{tikzcd}\label{Vsquare}
V_s  
\arrow[r,"\hat\iota"] \arrow[d, "k_s"] & V_s\arrow[d, "k_s"]\\
T\arrow[r,"\iota"]
& T
\end{tikzcd}
\end{equation}
commutes for all $s$, and hence, when $s\ne 0$,  the critical set  of $k_s$ is invariant under $\hat\iota$. 
Since the  fold image $k_s(C_s^{\ep_\gamma,\ep_\theta})$ converges to the corner corresponding to $(\ep_\gamma,\ep_\theta)$, $C_s^{\ep_\gamma,\ep_\theta}$ must be invariant under 
the involution $\hat \iota$.   Hence each of the four quotient circles $C_s^{\ep_\gamma,\ep_\theta}/\hat\iota$ maps diffeomorphically to $P_0^*$, encircling its corresponding  corner.

 Because of the local structure of a fold map, for any non-zero $s$ there exists an open  neighborhood  $W$ of the  the fold locus such that the restriction of $\pi_0$ to $W$   is two-to-one off the fold locus. When $s=0$, $\pi_0:\NAT_0\to P_0$ is two-to-one away from the critical set.  The last claim  about $\pi_0:\NAT_s \to P_0$ follows.   \end{proof}

  \begin{corollary} \label{Wattprime} 
  The statement in Corollary \ref{JamesWattdied} holds as well when $\NAT_s$ is replaced by  $\NAT_s'$.
  \end{corollary} 
  \begin{proof} To begin with, the proof of Lemma \ref{LP} depends only on the first order expansion of ${\bf G}_s$, so it works for ${\bf G}' _s$, too.    It follows that the statement in Proposition \ref{thmfold} holds with $V_s={\bf G}_s ^{-1} (0,0)$ replaced by $V_s '=({\bf G}' _s)  ^{-1} (0,0).$  
  \end{proof}

\section{Restriction to $P_0\times P_1$ }
\label{punchline}

  Having described the restriction $\pi_0:\NAT_s\to P_0$, we now describe the restriction
$\pi_1:\NAT_s\to P_1$ to the second pillowcase, making use of an involution of the tangle,   $U$,  which exchanges its boundary components. 
 The strategy used is to express $\pi_1$ as a composite $\pi_1=(\Psi\circ\Theta)\circ \pi_0\circ U_s$
where $U_s:\NAT_s\to \NAT_s$ is the involution induced by $U$,  $\pi_0:\NAT_s\to P_0$ continues to denote the composition of projection onto $T$ with the quotient $T\to T/\iota$,  and $ \Theta:P_0^*\to P_0^*$, $\Psi:P_0^*\to P_1^*$ are $s$-independent explicit symplectomorphisms. 
In addition, we introduce a  pair of involutions $W_1,W_2$ on $\NAT_s$ and $\NAT_s'$, used to streamline the proof of the main result.

 \subsection{The involution $U$}
 Denote by 
$$U:S^2\times [0,1]\to S^2\times [0,1]$$
the involution 
\begin{equation}
\label{defU}
U(x,y,z,t)=(-x,y,z,1-t).
\end{equation}
The involution $U$ is orientation-preserving, has fixed point set the circle $\{x=0\}\times\{\tfrac 1 2\}$,  and interchanges the two boundary components.  Moreover, the  two-sphere $S^2\times \{\tfrac 1 2\}$  is $U$-invariant.

 \begin{figure}[ht]
\begin{center}
\includegraphics[width=1\textwidth]{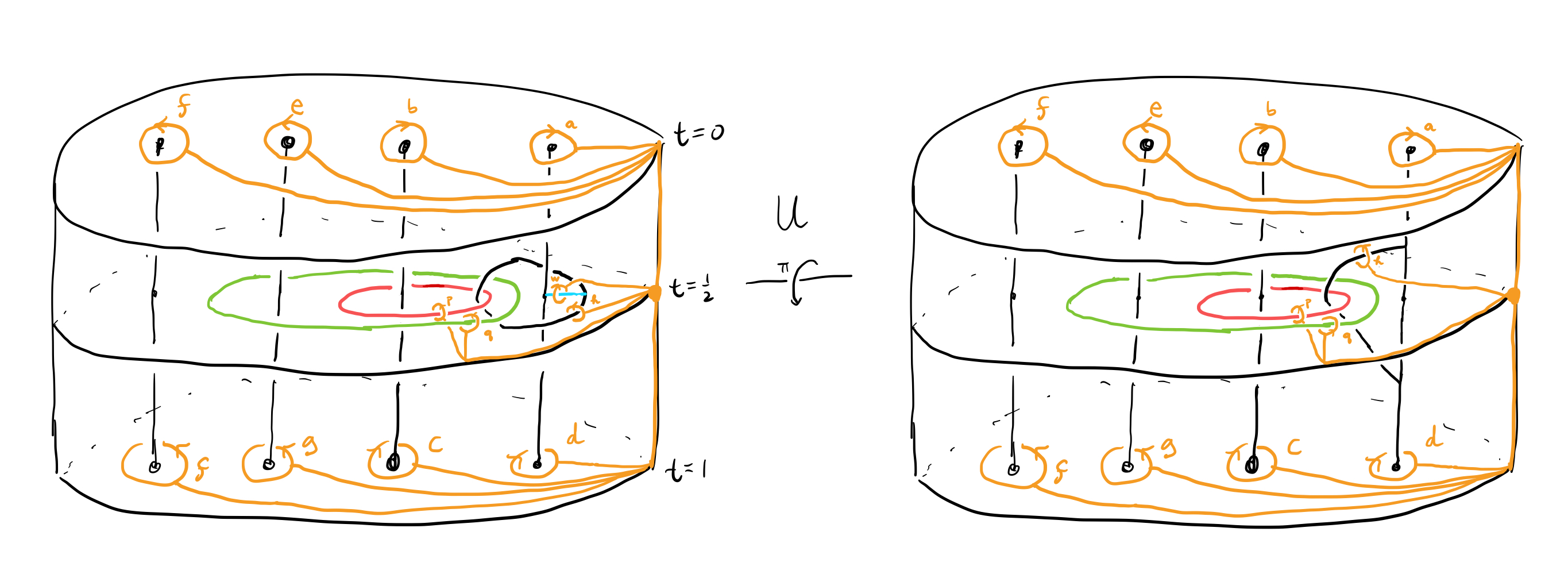}
 \caption{The earring and bypass tangles with their perturbation curves, loops generating $\Pi$ and $\Pi'$, and the involution $U$.\label{invoTfig}}
\end{center}
\end{figure}

The earring and  bypass tangles, endowed with their perturbation curves, may be isotoped to be $U$-invariant and to lie in $D^2\times I$, as illustrated in Figure \ref{invoTfig}.

Figure \ref{invoTfig} is simply a different perspective of Figure \ref{earpertfig}, and the  illustrated labelled   loops based at a fixed point of $U$  correspond to the labelled based loops in Figure \ref{earpertfig}.  In particular, generators and relations in the presentations given in Propositions \ref{summ} and \ref{summ2} agree.

\medskip

Rotating  Figure \ref{invoTfig} along the horizontal axis by the angle $\pi$ and keeping track of where homotopy classes of based loops are sent establishes the following lemma, which identifies the induced automorphism $U_\#$ of $\Pi$ (resp. $\Pi'$).

\begin{lemma}\label{actofT} In terms of the presentations given in Propositions \ref{summ} and \ref{summ2},  for both tangles,
$$
U_\#(a)=\bar d=\bff \bq  f \bar a \bar b \bp  b  p q, ~
U_\#(b)=\bar d \bar c d=(\bar f \bq f \bar a  \bar b \bp)\bar b (p b a \bar f q f),~U_\#(f)=\bar f,
$$
Since $U$ preserves the two  tori which bound tubular neighborhoods of the perturbation curves, there exist loops $\gamma_i$ so that  $$U_\#(p)=\gamma_1 \bar p\gamma_1^{-1},~
 U_\#(\lambda_p)=\gamma_1 \lambda_p^{-1}\gamma_1^{-1} 
,~U_\#(q)=\gamma_2 \bar q\gamma_2^{-1},~
 U_\#(\lambda_q)=\gamma_2 \lambda_q^{-1}\gamma_2^{-1} $$

In addition, 
\begin{itemize}

\item for the earring tangle, $U_\#(w)$ is conjugate to $w$ and $U_\#(h)=\bh \bar w$, and
\item for the bypass tangle $U_\#(h)=\ba h \bar q \bar p \bar h  p  q \bh a$. \qed
\end{itemize}

 \end{lemma}

  Lemma \ref{actofT}  implies that the action of $U$ on $\Hom(\Pi, SU(2))$ (resp. on $\Hom(\Pi', SU(2))$) preserves the $s$-perturbed flat traceless representations, and, in the case of the earring tangle, the $w_2$ condition, and thereby determines an involution which we denote by $U_s:\NAT_s\to \NAT_s$ (and similarly $U_s:\NAT'_s\to \NAT_s'$).    
 Explicitly, if $x\in \Pi$ and $[\rho]\in \NAT_s$
$$
(U_s(\rho))(x)=\rho(U_\#(x)).
$$

  Denote by
\begin{equation}\label{hatU}\widehat U:P_0\to P_1\end{equation}
the isomorphism on character varieties induced by the restriction $U|_{S^2\times\{1\}}: S^2\times\{1\}\to S^2\times\{0\}$.  
Functoriality immediately implies the following 
(recall that we write $u_s=\pi_0\times\pi_1$; see the comment following Lemma \ref{rs=p0p1}).
 \begin{lemma}\label{nice} The diffeomorphism 
 $U$ induces an analytic involution $U_s:\NAT_s\to \NAT_s$ for any $0<|s|<s_0$.  The diagram
 \begin{equation} 
\begin{tikzcd}\label{Usquare}
\NAT_s  
\arrow[r,"U_s"] \arrow[d, "\pi_0"] & \NAT_s\arrow[d, "\pi_1"]\\
 P_0\arrow[r,"\widehat{U}"]
& P_1
\end{tikzcd}
\end{equation}
commutes.
The identical statement holds for $\NAT'_s$.\qed
\end{lemma}

 In the following, $C_s\subset \NAT_s$ denotes the fold locus,
\begin{equation}\label{foldLocus}C_s =\{c\in \NAT_s~|~\ker (d\pi_0)_c\ne 0\}.\end{equation}
 Hence $C_s$ is the union of the four quotient circles $C_s^{\ep_\gamma,\ep_\theta}/\hat \iota$. The same holds for the  fold locus  $C_s'\subset \NAT_s'$.  

\medskip

\subsection{Relationship of $U_s$ to $\pi_1$}
We next express $\pi_1:\NAT_s\to P_1^*$ in terms of $U_s$, $\pi_0$, and a fixed identification $\Psi:P_0\cong P_1$.

Denote by $\psi:\Pi_1\to \Pi_0$ the isomorphism of peripheral subgroups taking the ordered triple of free generators $d,c,f$ to the free generators $a,b,f$.    Note that $\psi$ is induced by the obvious homeomorphism from $(S^2, 4)\times \{1\}$ to $(S^2,4)\times \{0\}$, which is orientation-reversing with respect to the boundary orientation on these two components.  Then set
\begin{equation}\label{Psi}\Psi:P_0\to P_1, ~  \Psi=\psi^*\end{equation}
the induced map on traceless character varieties. 

Also define 
\begin{equation}\label{Theta}\Theta:P_0\to P_0, ~ \Theta([\gamma,\theta])=[\gamma,-\theta]. \end{equation}
   Notice that $\Theta$ is orientation-reversing on $P_0 ^*$.   In terms the coordinates $P_0^*\subset \RR^3$ given by the characters of the loops $a,b,f$  (see Equation (\ref{curvypillow}):
  \begin{equation}
\label{theta2}
\Theta(x,y,z)=(x,y,2xy-z).
\end{equation}

\begin{lemma} \label{factor}\hfill
\begin{enumerate}
\item The map $\widehat{U}:P_0\to P_1$ factors as
$$
\widehat{U}= \Psi\circ  \Theta.$$
\item The maps $\pi_1:\NAT_s,\NAT' _s\to P_1^*$ factor as
\begin{equation*}\pi_1=\Psi\circ\Theta\circ\pi_0\circ U_s.
\end{equation*}
\item    The map
$u_s$ factors as \begin{equation*}
u_s:\NAT_s \xrightarrow{\pi_0\times (\Theta\circ\pi_0\circ U_s)} P_0^{*}\times P_0^*\xrightarrow{{\rm Id}\times \Psi}P_0^{*}\times P_1^*
\end{equation*}
and similarly for $u'_s:\NAT' _s\to P_0^* \times P_1 ^*$. 
\end{enumerate}
\end{lemma}
 
\begin{proof}  We show that $\Psi^{-1}\circ \widehat{U}=\Theta$, which implies the first statement. 
Let $\widehat{U}_\#:\Pi_1\to \Pi_0$ denote the isomorphism induced on peripheral subgroups  by the involution $U$, so that  $(\widehat{U}_\#)^*=\widehat{U}$.  Using Lemma \ref{actofT}, we compute the action of $\widehat{U}_\#\circ \psi^{-1}$ on the peripheral subgroup $\Pi_0$:
$$\widehat{U}_\#\circ \psi^{-1}(a)=\widehat{U}_\#(d)=\widehat{U}_\#(\widehat{U}_\#^{-1}(\bar a))=\bar a
$$
$$\widehat{U}_\#\circ \psi^{-1}(b)=\widehat{U}_\#(c)=\widehat{U}_\#(\widehat{U}_\#^{-1}(\bar a \bar b a))=\bar a \bar b a,~\text{ and }
$$
$$\widehat{U}_\#\circ \psi^{-1}(f)=\widehat{U}_\#(f)=\bar f.
$$
Hence 
$$(\widehat{U}_\#\circ \psi^{-1})^*=\Psi^{-1}\circ \widehat{U}:P_0\to P_0$$ takes $[\gamma,\theta]\in P_0$ to the conjugacy class of 
$$a\mapsto \bar \bbi, ~b\mapsto \bar\bbi e^{\gamma\bbk}\bar \bbi\bbi= e^{-\gamma\bbk}\bar\bbi, ~f\mapsto e^{\theta\bbk}\bar\bbi.
$$
Conjugating by $\bbj$ yields $[\gamma,\theta]\mapsto[\gamma,-\theta]$, which equals $\Theta$.  This proves the first statement.  The second and third statements follow  from the first and Lemma \ref{nice}.

\end{proof}
Since $\Psi$ is induced by an orientation-reversing diffeomorphism  between  the two four-punctured  two-sphere boundary components,  $\Psi$ restricts to a symplectomorphism $\Psi:P_0^{*-}\to P_1^*$.
 For all $0<|s|<s_0$, the first map in the third statement of Lemma \ref{factor}, $\pi_0\times (\Theta\circ\pi_0\circ U_s):\NAT_s\to P_0^{*-}\times P_0^{*-}$,   is a   Lagrangian  immersion (see Theorem \ref{thm6.1}).  The second map,  ${\rm Id}\times \Psi:P_0^{*-}\times P_0^{*-}\to P_0^{*-}\times P_1^*$  is an ($s$-independent) symplectomorphism induced by the natural identification of the two boundary components of $S^2\times I$.

  \medskip

\begin{proposition}\label{finepeopleonbothsides}  For $0<|s|\leq s_0$,
 the induced morphism $U_s:\NAT_s\to \NAT_s$ is an orientation-reversing involution of a closed genus three surface. Each of the four circles  in the fold locus $C_s$ is $U_s$-invariant. 
 The fold locus separates $\NAT_s$ into two  surfaces with boundary,
$$\NAT_s=\NAT_s^+\cup_{C_s}\NAT^-_s,$$ 
each homeomorphic to a  two-sphere with four open discs removed.  Each surface $\NAT_s^\pm$ maps homeomorphically into $P_0^*$ by $\pi_0$ with image the complement of a neighborhood of the corners,   with boundary the fold image. Moreover, $U_s(\NAT_s^+)= \NAT_s^+$ and $U_s(\NAT_s^-)= \NAT_s^-$.  The identical statement holds for $\NAT'_s$.\end{proposition}

\noindent{\em Remark. }  For concreteness, define $\NAT^+_s$ to be the (closure of) the path component of $\NAT_s\setminus C_s$ which contains the interior point $[\gamma,\theta,\nu, \tau]=[\tfrac\pi 2,0,0,0]$;   Corollary \ref{two points}  below shows that this point lies in both $\NAT_s$ and $\NAT_s'$ for any $s$, and is fixed by $U_s$.

\begin{proof} Note that  $U_s:\NAT_s\to \NAT_s $ is  an involution. 
Lemma \ref{folding} shows that the critical sets of $\pi_0:\NAT_s\to P_0^*$ and $\pi_1:\NAT_s\to P_1^*$ coincide. Lemma \ref{nice} then shows that  $U_s(C_s)=C_s$.   
Corollaries \ref{genus53} and \ref{JamesWattdied} imply that $\NAT_s^\pm$ are  four-punctured spheres, mapped homeomorphically by $\pi_0$ into $P_0^*$,
with image the complement of a neighborhood of the corners, and boundary the fold image.

\medskip

It remains to prove three assertions: $U_s$ is orientation-reversing, $U_s$ preserves the path components of $C_s$,  and $U_s(\NAT^+_s)=\NAT_s^+$.
These three properties are stable under small deformations. Therefore it suffices to check them when $s=0$, since 
  Equation (\ref{localtriv}) and Corollary \ref{genus53}  show that   $\{V_s/\hat\iota\} $ is a smooth family of genus three surfaces  over $[-s_0,s_0]$ with $V_s/\hat\iota =\NAT_s$ when $s\ne 0$.

Recall that $V_0/\hat\iota= \{{\bf G}_0=0\}/\hat\iota$. Lemma \ref{estimate2} identifies $V_0$ as
$$V_0=\{\sin\gamma\sin\tau-\sin\theta\cos\tau=0, \nu=0\},$$ and $\hat\iota(\gamma,\theta,\tau)=(-\gamma,-\theta,\tau+\pi)$  (Equation (\ref{iota vs tau plus pi})). 
 
Using Lemma \ref{actofT}, a straightforward calulation shows that, 
$$U_0(a)=\bar a=\bar \bbi, ~ U_0(b)=e^{-\gamma\bbk}\bar\bbi,~ U_0(f)=\bar f=e^{\theta\bbk}\bar\bbi, ~ U_0(h)=h
=e^{\tau\bbi}\bbj.$$  
Conjugating by $\bbk$ shows that the involution $U_0:V_0/\hat\iota\to V_0/\hat\iota$ is given, in the coordinates of Equation (\ref{iota vs tau plus pi}), by
\begin{equation}\label{Tatzero} U_0([\gamma,\theta,0,\tau])=[-\gamma,\theta,0,-\tau+\pi].\end{equation}
In particular, the  path components of the critical set $C_0=\{[\ell_1\pi,\ell_2\pi,0,\tau]\}$ are preserved by $U_0$.

The point $[\gamma,\theta,\nu,\tau]=[\tfrac \pi 2,0,0,0]$ lies in $V_0/\hat\iota$, since $\sin\gamma\sin\tau-\sin\theta\cos\tau=0$,  is fixed by $U_0,$  and does not lie in  $C_0$.  Thus $U_0$ preserves the path components of the decomposition along $C_0$.  

  Since $\Theta:P_0^{*-}\to P_0^{*-}$ is orientation-reversing and $\Psi:P_0^{*-}\to P_1^*$ is a symplectomorphism, Lemma \ref{factor} implies that $\widehat{U}$ is orientation-reversing.  Since $\pi_0$ and $\pi_1$ are local diffeomorphims away from $C_s$, the diagram (\ref{Usquare})  shows that  $U_s:\NAT_s\to\NAT_s$ is    orientation-reversing.

Since $V_0=V_0'$, the same argument applies to $\NAT_s'$
\end{proof}

\begin{theorem}\label{thm6.1} For $0<|s|<s_0$ , the restriction map $u_s:\NAT_s \to P_0^{*-}\times P_1^*$  is a Lagrangian bifold immersion with embedded fold singularities. The fold locus $C_s$ (Equation (\ref{foldLocus})), a disjoint union of four circles, separates $\NAT_s$ into two four-punctured two-spheres: $\NAT_s=\NAT_s^+\cup_{C_s} \NAT_s^-$.  

Each of the four composites 
$$\pi_i:\NAT^\pm_s\to P_i^*, i=0, 1$$
is a homeomorphism onto its image, and a diffeomorphism on the interior $\NAT_s^\pm\setminus C_s$, with image 
 the complement of  four disk neighborhoods of the corners.   As $s\to 0$, the circles converge to the corners.

 The identical statement holds for the bypass tangle and $\NAT_s'$.
\end{theorem}

\begin{proof} The fact that $u_s$ is a Lagrangian immersion into the traceless moduli space of the oriented boundary of the earring or bypass tangle (with respect to the Atiyah-Bott-Goldman symplectic form)  is a general feature of the restriction-to-the-boundary map for holonomy perturbed character varieties of  three-manifolds; an explicit theorem which applies in this case is proved in \cite{CHK}.

Commutativity of the Diagram   \ref{Usquare}, together with the fact that the horizontal maps are real-analytic isomorphisms,  implies that $U_s$ sends the critical set of $\pi_0$ bijectively onto the critical set of $\pi_1$.   The map $\widehat{U}$ sends corners to corners, and hence the image of $\pi_1$ misses the corners of $P_1$ and the four critical circles are sent by $\pi_1$ to  circles which encircle the four corners of $P_1$.  The property of being a fold map is coordinate independent, and therefore $\pi_1$ has only fold singularities.  The remaining statements  follow from Proposition \ref{finepeopleonbothsides} and Corollaries  \ref{JamesWattdied} and \ref{Wattprime}. \end{proof}

\medskip

We finish this section with a lemma which 
expresses the  map $\pi_0\circ U_s:\NAT_s\to P_0$ along the bottom edge in terms of the  characters associated to the loops $b\ba$, $ f\ba$,  and $b\bar f$  which embed $P_0$ in $\RR^3$.  Recall that (by definition)
\begin{align*}(\pi_0\circ U_s)(\rho) &=\big(\Real((U(\rho))(b\ba)), \Real((U(\rho))(f\ba)), \Real((U(\rho))(b\bar f))\big)\\
&=\big(\Real(\rho\circ U_\#(b\ba)),\Real(\rho\circ U_\#(f\ba))
,\Real(\rho\circ U_\#(b\bar f))\big)\end{align*}
(see Equation (\ref{curvypillow}) and Lemma \ref{mapstor3}).

   \begin{lemma} 
\label{actofT2}
Suppose that $[\rho]\in \NAT_s$ or $[\rho]\in \NAT'_s$ satisfies $\rho(f)=\rho(a)=\bbi$.  Then 
$$(\pi_0\circ U_s)(\rho) 
=(\Real(\rho(  \bar b   q \bbi\bq)) ,
  \Real( \rho(\bp^2) ), \Real( \rho(\bar b \bp^2 q \bbi \bq     ) ).$$

\end{lemma}

\begin{proof}  We compute each term, referring to Lemma \ref{actofT}, Proposition \ref{summ}, and  Proposition \ref{summ2}, and plugging in $a=f=\bbi$ to simplify:
\begin{align*} 
\Real((U(\rho))(b\ba))&=\Real(\rho(U(b\ba))) 
 =\Real(\rho(\bar d \bar c d d))  
 =\Real(\rho(  \bar c d))\\ &= \Real(\rho(\bq b \bar a \bar f q f))   
 = \Real(\rho(\bq \bar b  q f))  = \Real(\rho( \bar b  q \bbi \bq))
\end{align*}
Similarly, since $\rho(\bar b p)=\rho(\bp \bar b)$,
\begin{align*} 
\Real((U(\rho))(f\ba))&=\Real(\rho(U(f\ba))) 
 =  \Real(\rho(\bar f d))\\
&=\Real(\rho(\bar f \bq \bp \bar b  p b a \bar f q f)= \Real(\rho(\bp^2 a\bar f))=\Real(\rho(\bp^2 )).
\end{align*}
Finally
\begin{align*} 
\Real((U(\rho))(b\bar f))& 
 =  \Real(\rho(\bar d \bar c d f))=\Real(\rho(\bq f \bar a \bar b \bp \bar b pba\bar f q f))\\
&=\Real(\rho(\bq f \bar a \bar b \bp^2a\bar f q f))=
\Real(\rho(\bq  \bar b \bp^2  q f))=\Real(\rho( \bar b \bp^2  q \bbi \bq )).
\end{align*}
\end{proof}

\subsection{  Two further involutions  $W_1, W_2$ on $\NAT_s$ and $\NAT' _s$
}\label{Wi}
The following lemma describes two more   involutions on $\NAT_s$ and $\NAT' _s$, induced by multiplying by a pair of center-valued representations.

\medskip

\begin{lemma}\label{2invols} The involutions $W_1,W_2$ on $T\times \tracelessTwoSphere$ given by 
$$W_1(\gamma,\theta,\nu,\tau)= (\gamma+\pi, \theta+\pi, -\nu,\tau+\pi) \mbox{,   } W_2(\gamma,\theta,\nu,\tau)=(\gamma+\pi,\theta, -\nu, -\tau+\pi)$$
preserve the subset $\widehat \NAT_s$ and they commute with the involution $\hat \iota$, so they define involutions (which we continue to call $W_1,W_2$) on $\NAT_s$.
Furthermore, the involutions preserve the fold loci, 
 $$W_i(C_s)=C_s, \mbox{ for }i=1,2,$$
 and commute with $U_s$
 $$W_i\circ U_s=U_s\circ W_i:\NAT_s\to \NAT_s.$$
 
 The identical statements hold with $ \widehat \NAT _s, \NAT_s, C_s$ replaced by $\widehat \NAT '_s, \NAT ' _s, C'_s$.  
\end{lemma}

\begin{proof}  

 In the notation of Proposition \ref{summ}, define the representations
 $$\chi_1,\chi_2:\Pi\to \{\pm 1\}\subset SU(2)$$ by
$$\chi_1: a\mapsto 1, b\mapsto -1, f\mapsto -1, h\mapsto -1, p\mapsto 1, q\mapsto 1$$
 and
 $$\chi_2: a\mapsto -1, b\mapsto 1, f\mapsto -1, h\mapsto 1, p\mapsto 1, q\mapsto 1.$$
Since these are central, the pointwise product $\chi_i\cdot \rho:\Pi\to SU(2)$ with any representation $\rho$ is again a representation. The resulting involutions   on $\Hom(\Pi,SU(2))$  preserve the perturbation condition, the $w_2$ condition, and conjugacy classes, so they descend to involutions $\chi_i$ on $\NAT_s$. Since $\chi_i \circ U_\#=\chi_i$, then $U$ commutes with the action of $\chi_i$ on $\Hom(\Pi,SU(2))$ and on $\NAT_s$.

\medskip 

 Given the formulas for $W_i$ in the lemma, if  $L_s(\gamma, \theta, \nu, \tau) =(a,b,f,h,p,q)$, then $L_s( W_1(\gamma, \theta, \nu, \tau)) = (a,-b, -f, -h, p, q)$ and  $L_s( W_2(\gamma, \theta, \nu, \tau)) = (a,-b, f, \bbk h \bar \bbk, \bbk p \bar \bbk, \bbk q \bar \bbk)$.  The $\chi_1$ action on $\tNAT_s$ preserves the partially gauge-fixed subset $\hNAT_s$, and the action on $\hNAT_s$ agrees with $W_1$.  The $\chi_2$ action on $\tNAT_s$ does not preserve $\hNAT_s$, but the modification $\rho\mapsto \bbk (\chi_2 \cdot \rho) \bar \bbk$ preserves $\hNAT_s$ and agrees with $W_2$.   

We leave the reader to check, using Equation (\ref{iota vs tau plus pi}), that $W_1$ and $W_2$ commute with $\hat \iota$.  It follows that $W_i$ and $\chi_i$ induce the same action on $\NAT_s$ and the action commutes with $U_s$.

To verify the claim that $W_1, W_2$ preserve the fold locus, set  \begin{equation*}
\widehat{W}_1([\gamma,\theta])=[\gamma+ \pi, \theta+\pi],~~ \widehat{W}_2([\gamma,\theta])=[ \gamma+\pi ,\theta]; \end{equation*}
 then the diagrams
\begin{equation} \label{W1square}
\begin{tikzcd} 
\NAT_s  
\arrow[r,"W_i"] \arrow[d, "\pi_0"] & \NAT_s\arrow[d, "\pi_0"]\\
 P_0^*\arrow[r,"\widehat{W}_i"]
& P_0^*
\end{tikzcd}
\end{equation}
 commute with the horizontal maps diffeomorphisms. Hence $W_i$ preserves the critical set $C_s$. 
 
 Identical arguments work for the bypass case with primes added throughout.  \end{proof}

The homeomorphisms  $\widehat W_1, \widehat W_2:P_0\to P_0$ covered by $W_1, W_2$ will occasionally be useful in the rest of the paper, so we pause here to note  that 
 $\widehat{W}_1$ 
 lifts  to a rotation of $\RR^2$ about the point $(\tfrac \pi 2 , \tfrac \pi 2)$ by angle $\pi$, and therefore is orientation-preserving on $P_0^*$.
Similarly, $   \widehat{W}_2$
lifts to a rotation  of $\RR^2$ about the point $(\tfrac \pi 2 ,0)$ by angle $\pi$,  and so $\widehat{W}_2$   is also orientation-preserving.  By inspection, both $\widehat{W}_1,\widehat{W}_2$ commute with $\Theta$.

  \section{The correspondence $(\NAT_s)_*:\Lag(P_0^*)\to\Lag(P_1^*)$}   \label{corr}

Given a symplectic manifold $(M,\omega)$, let $\Lag(M)$ denote the set of Lagrangian immersions $g:L\to M$.
 If $M$ is compact we assume $L$ is compact.
When $M$ is non-compact, additional restrictions are placed on the Lagrangian immersion, including restricting to   proper Lagrangian immersions with appropriate behavior near the ends of $M$; we refer to the literature  for careful descriptions of those Lagrangian immersions which qualify as
 {\em objects of the immersed wrapped Fukaya category} of a non-compact symplectic manifold  \cite{seidel, fukaya,abou}.   An elementary definition of $\Lag(P_0^*)$ appropriate for the current article  is given in the following subsection.

 \subsection{Lagrangian composition}\label{composition1}

We recall a fundamental observation of Weinstein \cite{weinstein,WW2}.  A Lagrangian immersion into a product of symplectic manifolds
$$\phi=(\phi_0, \phi_1):N\to (M_0\times M_1,-\omega_{M_0}\oplus \omega_{M_1}).$$
defines   a function  
$$\phi_*:\Lag(M_0)^{\pitchfork_\phi}\to \Lag(M_1),$$
called {\em composition with the Lagrangian correspondence $\phi$}, where 
the domain $$\Lag(M_0)^{\pitchfork_\phi}\subset \Lag(M_0)$$
is defined to be  the subset of Lagrangian immersions $g:L\to M_0$ such that
$$g\times {\rm Id}_{M_1}:L\times M_1\to M_0\times M_1 
 ~~\text{ and }~~
 \phi:N\to M_0\times M_1$$
 are transverse. When $g:L\to M_0$ lies in $  \Lag(M_0)^{\pitchfork_\phi}$, then $$X:=\{(\ell,m_1,n)\in L\times M_1\times  N~|~  \phi(n)=(g(\ell),m_1) \}$$ is a smooth submanifold of $L\times M_1\times N$ called the {\em fiber product} of $g\times {\rm Id}_{M_1}$ and $\phi$. 
 
 Define
 $\phi_*(g)$  to be the restriction of the projection ${\rm proj}_1:L\times M_1\times N\to M_1$ to $X$, 
 $$\phi_*(g): X\xrightarrow{{\rm proj}_{M_1}} M_1,~~\phi_*(g)(\ell,m_1,n)=m_1.
 $$
 A foundational observation in this framework, observed by Weinstein, is the following theorem.
 \begin{theorem} [\cite{weinstein}] \label{Cain died, sadly}If $g\in\Lag(M_0)^{\pitchfork_\phi}$, then $\phi_*(g)$ is a Lagrangian immersion.\end{theorem}  
 When $M_1$ and $M_0$ are compact, the assignment $g\mapsto \phi_*(g)$ defines a function
\begin{equation}\label{eq6.2}\phi_*:\Lag^{\pitchfork_\phi}(M_0)\to \Lag(M_1).\end{equation}
 In the non-compact case,  the definitions of $\Lag^{\pitchfork_\phi}(M_0)$  and  $\Lag(M_1)$  may require additional restrictions to ensure that the function $\phi_*$  of Equation (\ref{eq6.2}) be defined.   
 
 A familiar instance of this construction is when 
 $N=M_0$, $\phi_-={\rm Id}_{M_0}$, and $\phi_1:M_0\to M_1$  a symplectomorphism.  Then $   \Lag(M_0)^{\pitchfork_\phi}=   \Lag(M_0)$ and $\phi_*(g)$ equals the composition (of functions) $\phi_1\circ g$.

\subsection{Composition with $\NAT_s$}\label{pitchforks}

 In what follows, we  apply the construction of the Section \ref{composition1} in the instance when the map $\phi$ is 
$$u_s=\pi_0\times \pi_1:\NAT_s\to P_0^{*-}\times P_1^*\in \Lag(P_0^{*-}\times P_1^*),$$
for small, non-zero $s$.

 When $J$ is a 1-manifold and  $\alpha:J\to P_0^*$  an immersion, the {\em Weinstein composition of $\alpha$ and $u_s:\NAT_s\to P_0^*\times P_1^*$} is given by the top horizontal row in the diagram
\begin{equation}
\begin{tikzcd}\label{Jsquare}
F
\arrow[r, "\alpha^*"] \arrow[d] & \NAT_s \arrow[d, "\pi_0"] \arrow[r,"\pi_1"]& P_1^* \\
 J\arrow[r,"\alpha" ]
 & P_0^*&
\end{tikzcd}
\end{equation}
with the square a pullback and $F\subset J\times \NAT_s$ the fiber product.  If $\alpha$ is transverse to the fold image, $F$ is a 1-manifold and $\pi_1\circ \alpha^*$ an immersion.  In this way composition with $u_s$ induces
 $$(u_s)_*:\Lag^{\pitchfork_{u_s}}(P^*_0)\to \Lag(P^*_1),$$   which  assigns an immersed curve in $P_1^*$ to any  immersed curve in $P^*_0$ satisfying the   transversality condition with respect to $u_s$.

 \medskip

\subsection{Immersed curves in the pillowcase}\label{LaginP}  We  provide an elementary definition of $\Lag(P_0^*)$, and some of its subspaces, suitable for the rest of this article.

\medskip

  Let  $$\WNb=S^1\times (-1,1),$$  equipped with the involution $(e^{x\bbi},t)\mapsto (e^{-x\bbi},-t)$,
 and denote by $\Nb$ its orbit space $$\Nb=\WNb / \{ \pm 1\}.$$ The involution on the annulus $\WNb$ has two fixed points; its orbit space $\Nb$ is a disk with two orbifold points (i.e.,   corners). The center circle $S^1\times\{0\}\subset \WNb$ is invariant and maps to an arc in $\Nb$ joining the  corners.    Denote by $\WNb^*$ the complement of these two fixed points
 and by $\Nb^*$ the complement of the two corners.

\medskip
The bottom edge of $P_0=T/\iota $ is parameterized
by \begin{equation}\label{bedge}\beta:[0,\pi]\to P_0, ~ \beta(x)=[x,0].\end{equation}  The map  \begin{equation}\label{eqn6.3}
\tilde\beta:\Nb\to P_0\text{ given by 
 } [e^{x\bbi},t]\mapsto[x,t]\end{equation}
is a continuous embedding, covered by  the smooth equivariant embedding $\WNb\to T, (e^{x\bbi},t)\mapsto (e^{x\bbi},e^{t\bbi})$. Moreover, $\beta(x)=\tilde\beta(e^{x\bbi},0)$ for $0\leq x\leq \pi$.
Thus $\tilde\beta$  is a  tubular (orbifold) neighborhood
of the the bottom edge $\beta$.

\medskip

These observations motivate the following definition.

Equip $S^1$ with the involution $e^{x\bbi}\mapsto e^{-x\bbi}, $ with fixed points $\pm1$.

 \begin{definition}\label{goodarc} Call  a map   $g:[0,\pi]\to  P_0$ {\em good} if it satisfies 
\begin{enumerate}
\item $g(0)$ and $g(\pi)$ are corners, and $g(x)\in P_0^*$ for $0<x<\pi$,  and
\item  there exists an equivariant immersion $\hat g:S^1\to T$ so that $[\hat g(e^{x\bbi})]=g(x)$ for $0\leq x \leq\pi$.

 \end{enumerate}
\end{definition}
In alternative language, a good map is a smooth orbifold immersion of the closed 1-dimensional  orbifold $S^1/\{\pm 1\}$  into  $P_0=T/\{\pm 1\}$.

\begin{definition}\hfill
\begin{enumerate}
\item Two proper immersions $g_i:J_i\to P_0^*, ~ i=1,2$  of 1-manifolds $J_1,J_2$ are called  {\em equivalent} if there exists a 
diffeomorphism $h:J_1\to J_2$ so that $g_2=g_1\circ h$.
\item 
An {\em immersed curve  in} $ P_0 ^* $ is defined to be an equivalence class of  immersions into $P_0^*$  with domain  a finite disjoint union of circles and open intervals such that the restriction  to each  open interval extends to   a good map of the closed interval.

\item  The set of  immersed curves in $P_0^*$ is denoted by $\Lag(P_0^*)$.
\end{enumerate}
 \end{definition}

  It can be proved, using the results of  \cite{HK,FKP}, that any  two-stranded tangle in a $\ZZ$-homology ball admits arbitrarily small holonomy perturbations so that its holonomy perturbed traceless $SU(2)$ character variety  lies in $\Lag(P_0^*)$.

   Denote by $$\Lag_{\rm cir}(P_0^*)\subset \Lag(P_0^*)$$  those immersed curves whose domain is a finite union of circles.  Denote by $$\Lag_{\rm int}(P_0^*)\subset \Lag(P_0^*)$$ those immersed curves whose domain is a finite union of open intervals. Note that $\Lag_{\rm cir}(P_0^*)$ is precisely the set of  immersed curves with compact domain.

\begin{definition}  Call two good curves $ g_1,  g_2:[0,\pi]\to P_0$ {\em regularly homotopic},   and write $[g_1]=[g_2]$,   if they admit equivariant lifts $\hat g_1,\hat g_2:S^1\to T$ that are equivariantly regularly homotopic.  This extends the usual notion of  regular homotopy classes for immersed circles  in $\Lag_{\rm cir}(P_0^*)$  to all of $\Lag(P_0^*)$.
Denote the set of regular homotopy classes by $$[\Lag(P_0^*)]=\Lag(P_0^*)/{\rm  regular~ homotopy}.$$

\end{definition}
\medskip

 The next  lemma provides an alternative   description of the transversality condition needed to compose with $u_s$.
 
 \begin{lemma}\label{preimageisgood} An    immersed curve   $g\in \Lag(P_0^*)$ lies in  $\Lag^{\pitchfork_{u_s}}(P^*_0)$  if and only if $g$ is transverse to the fold image $\pi_0(C_s)\subset P_0^*$. In particular, any $g\in \Lag(P_0^*)$ lies in $\Lag^{\pitchfork_{u_s}}(P^*_0)$ for all small enough non-zero $s$.
\end{lemma}
 \begin{proof} The first statement follows  from the local description of a fold and the fact that $u_s$ is a local diffeomorphism off the fold locus. The second is immediate from the definition of a good map.
 \end{proof}

\subsection{Two geometric operations on immersed curves}
\label{operations}
\subsubsection{Doubling interior curves}

We define two  doubling operations $$D,\widetilde D:\Lag_{\rm cir}(P_0^*)\to\Lag_{\rm cir}(P_0^*).$$

\begin{definition}\label{double} Suppose that $ g:\sqcup_i S^1\to P^*_0$ lies in $\Lag_{\rm cir}(P_0^*)$.    Define the {\em double, $D( g)$,  of $ g$},  to be the composition of $ g$ with the trivial  two-fold covering of its domain:
$$D( g):\sqcup_i (S^1\sqcup S^1) \to \sqcup_i S^1\xrightarrow{ g} P^*_0.$$

Define the {\em twisted double, $\widetilde D( g)$, of $ g$}, by precomposing with the non-trivial  two-fold cover of each circle in the domain $$\widetilde D( g):\sqcup_i  S^1  \to\sqcup_i S^1\xrightarrow{ g} P^*_0.$$

    \end{definition}

Although not needed for the
  statements or proofs of our results, $\widetilde D$ provides a convenient way to describe certain elements of $\Lag_{\rm circ}(P_0^*)$ which show up in character variety examples.

\medskip

Clearly, $D$ and $\widetilde D$ induce well-defined functions
 $$D, \widetilde D:[\Lag_{\rm cir}(P_0^*)]\to [\Lag_{\rm cir}(P_0^*)]$$
on regular homotopy classes.

\subsubsection{Replacing  a proper immersed arc by a nearby  Figure Eight} \hfill

Next we define the {\em Figure Eight doubling operation} 
$$F8:[\Lag_{\rm int}(P_0^*)]\to[\Lag_{\rm cir}(P_0^*)].$$
We begin with a standard model.  For any $|s|\leq 1$, define $F8_s:S^1\to \mathcal{N}$ by the formula
$$  F8_s(e^{\bbi \sigma})=[\sigma, -2s\cos \sigma].$$
When $s\ne 0$, this is an immersed circle in $\mathcal{N}^*$ with   a single transverse double point at $F8_s(\tfrac\pi 2)=F8_s(-\tfrac\pi 2)=[\tfrac\pi 2, 0],$
 $F8_s$ and $F8_{-s}$ are equal as unparameterized immersions.
As $s\to 0$,  $F8_s$ converges to  a path which traverses from one corner of $\Nb$ to the other and then returns along the same path.

\medskip

Any   equivariant immersion $\hat g:S^1\to T$ lifting a good path $ g$ admits an equivariantly immersed tubular neighborhood, that is, an equivariant immersion 
$\tau_{\hat g}:\widetilde{\mathcal{N}}\to T$
satisfying
 $\tau_{\hat g}(e^{\sigma\bbi},0)=\hat g(e^{\sigma\bbi})$  and
$\tau_{\hat g}(e^{-\sigma\bbi},-t)=\iota(\hat g(e^{\sigma\bbi},t)).$   Taking the quotient by the involution gives an orbifold tubular neighborhood of $ g$, which we denote as
$$\tilde g:\mathcal{N}\to P_0.$$

\begin{definition} \label{deffig8} Given    $ g\in \Lag_{\rm int}(P_0^*)$ and $0<|s|\leq1$, we call an immersed circle in $P_0 ^*$ a {\em Figure Eight curve supported by $g$} if the immersion  can be expressed as $F8_s(\tilde g)$ for some immersed orbifold tubular neighborhood $\tilde  g:\mathcal{N}\to P_0$.   See Figure \ref{Fig5fig}.
\end{definition}

 \begin{figure}[ht]
\begin{center}
\def\svgwidth{2in}
\includegraphics[width=.9\textwidth]{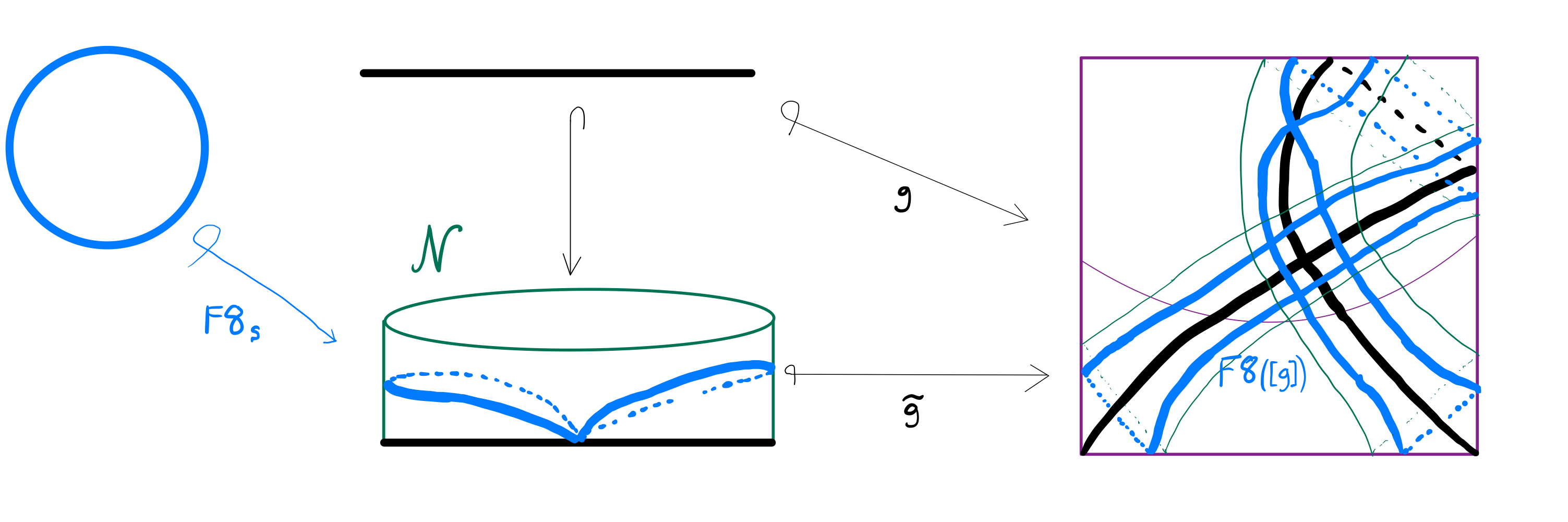}
 \caption{A good immersion $g$ of an  interval into $P_0^*$, its extension  $\tilde g:\mathcal{N}\to P_0$,  and   $F8([g])$.  \label{Fig5fig} }
\end{center}
\end{figure}

Note that $ F8_s(\tilde g)$ and $F8_{-s}(\tilde g)$ are equivalent as unparameterized immersions. Hence, the class of $F8_s(\tilde g)$ in $[\Lag_{\rm cir}(P_0^*)]$ is independent of (non-zero) $s$. Moreover, the class of $F8_s(\tilde g)$ depends only on the regular homotopy class of $g$. 
  Hence,  $F8_s$ well defines a map
 $$F8:[\Lag_{\rm int}(P_0^*)]\to [\Lag_{\rm cir}(P_0^*)].$$

\medskip

As a tautological example, for the bottom edge $\beta$ and the choice of extension $\tilde\beta$ given in Equations (\ref{bedge}) and (\ref{eqn6.3}),
\begin{equation}
\label{fig8inP}
F8_s(\tilde \beta)(\sigma)=
  [\sigma, -2s\cos\sigma]. 
\end{equation}
This can be rewritten in  terms of the three characters which embed $P_0$ into $\RR^3$ (see Equation (\ref{curvypillow})) as:
\begin{equation}
\label{fig8inr3}
 F8_s(\tilde\beta)(\sigma)=
(\cos\sigma,\cos(2s\cos\sigma), \cos(\sigma+2s\cos\sigma))
\end{equation}

\subsection{The correspondence on $[\Lag_{\rm cir}(P_0^*)]$}

The following proposition describes the action of the correspondence $(u_s)_*$ on $[\Lag_{\rm cir}(P_0^*)]$.

\begin{theorem} 
 \label{interiorc}
 Let $ g\in \Lag_{\rm cir}(P_0^*)$.  Then for all sufficiently small non-zero $s$, \begin{itemize} 
 \item $ g\in \Lag^{\pitchfork_s}(P_0^*)$, 
 \item the regular homotopy class of 
 $(u_s)_*( g)$ depends only on the class of $ g$ in  $[\Lag_{\rm cir}(P_0^*)]$,  and 
 \item $[(u_s)_*(g)] =[\Psi\circ D( g)]$  in $[\Lag_{\rm cir}(P_1^*)]$.
 \end{itemize} 
\end{theorem}

We prove a  more general result, Lemma \ref{lem6.7}.  In the case $J$ is a circle the statement of the lemma is exactly Theorem \ref{interiorc}.  The other case, when $J$ is a closed interval, is used in the calculation of the action of 
$(u_s)_*$ on $[\Lag_{\rm int}(P_0^*)]$ below.

\begin{lemma}\label{lem6.7}
Let $  g:J\to P_0^* $ be an immersion.  Assume either that $J$ is a circle, or that $J$ is a closed interval.

 Then for all $s$ small enough, $ g$ misses the fold image.  The fiber product of $g\times {\rm Id}_{P_1^*}$ and $u_s$ is a disjoint union of two copies of $J$, 
and the Weinstein composition   $(u_s)_*( g)$ (see Equation (\ref{Jsquare})) is  given as a composite
\begin{equation}\label{eq6.a}(u_s)_*( g):J\sqcup J\xrightarrow{ g^+_s\sqcup  g^-_s} \NAT_s\xrightarrow{\pi_1}P_1^*,\end{equation}
where $ g_s^\pm:J\to \NAT_s^\pm$ are immersions satisfying $\pi_0\circ  g^\pm_s= g$ and which depend smoothly on $s$.   When $s=0$,  the map (\ref{eq6.a}) equals $   \Psi\circ D( g ).$

\end{lemma}

\begin{proof}  There exists an $s(g)>0$ so that if $0<|s|<s(g)$, 
 the image of $ g$ lies in the interior of the image of $\pi_0:\NAT_s\to P_0^*$ for all $s\leq s(g)$. 
For $s\leq s(g)$, the restriction $\NAT_s\setminus C_s\to P_0^*$ is a trivial  two-fold cover of its image, which contains $ g(J)$; the two sheets are ${\rm Int}(\NAT^\pm_s)$.  This makes sense at $s=0$ also, provided we substitute $V_0/\hat\iota$ for $\NAT_0$, which we do for the rest of this proof
 (see Corollary \ref{genus53} and Proposition \ref{finepeopleonbothsides}).    Hence there exist   $s$-dependent    immersions $$ g_s^\pm:J\to {\rm Int}(\NAT_s^\pm)\subset T\times_{\hat\iota}\tracelessTwoSphere$$ 
 which lift $ g$  for all $|s|<s(g)$. It follows that  the fiber product  is given by  $F=J\sqcup J$, and the Weinstein composition is given by  (see Equation (\ref{Jsquare}))
\begin{equation*} 
 \label{Jsquare2}
J\sqcup J\xrightarrow{g_s^+\sqcup g_s^-}\NAT_s\xrightarrow{\pi_1}P_1^*. 
\end{equation*}
 Lemma \ref{factor} implies  that the map of Equation (\ref{eq6.a}) is given by
$$
\pi_1\circ ( g_s^+\sqcup  g_s^-) =\Psi\circ \Theta\circ\pi_0\circ U_s\circ ( g_s^+\sqcup  g_s^-).
 $$
 The maps $\pi_1\circ  g_s^\pm$ are immersions. Corollary \ref{genus53}  implies that the  $ g_s^\pm$ vary smoothly in $s$.  In particular, $
\pi_1\circ ( g_s^+\sqcup  g_s^-)$ is a regular homotopy (if $J$ is an interval its endpoints need not be fixed).

\medskip

Fix $\ell\in J$ and write $ g(\ell)=[\gamma, \theta]$. Equation (\ref{sect}) implies that $ g_0^\pm$ is given by 
  $$ g_0^\pm(\ell)=[\gamma,\theta, 0,  \tau^\pm],$$ where $$ e^{\tau^\pm\bbi}=\pm\tfrac{1}{\sqrt{\sin^2 \gamma  +\sin^2 \theta   }}
\big(\sin \gamma   + \sin \theta  \bbi\big).$$
Hence, using   Equations (\ref{Theta})   and 
  (\ref{Tatzero}),
\begin{align*}
 \Theta\circ\pi_0\circ U_0\circ  g^+_0(\ell)&=  
 \Theta\circ\pi_0\circ U_0([\gamma , \theta ,0,\tau^+ ])  \\
 &=  
 \Theta\circ\pi_0([-\gamma , \theta ,0,\pi -\tau^+ ])  \\
 & 
  =[\gamma ,\theta ]= g(\ell) \end{align*}
and similarly for $ g^-_0$. 
It follows that  varying $s$ produces a regular homotopy from $\pi_1\circ ( g_s^+\sqcup  g_s^-)$  to $$\pi_1\circ ( g_0^+\sqcup  g_0^-)=\Psi( g\sqcup  g).$$
\end{proof}

\subsection{Special case: action of $(u_s)_*$ or $(u' _s)_*$ a particular good arc}

Recall from  Equation (\ref{bedge}) that
 $\beta\in \Lag_{\rm int}(P_0^*)$ denotes the  bottom edge: $\beta(\gamma)=[\gamma,0],~0<\gamma<\pi.$
  In this section we prove Proposition \ref{theta0}, which states that    $(u_s)_*(\beta)=\Psi\circ F8(\beta)$ in $[\Lag_{\rm cir}(P_1^*)]$. 
 
 \medskip

\subsubsection{Four   points in $\NAT_s$ and $\NAT_s'$}
 \begin{proposition} \label{two points} The four points $$[\gamma,\theta,\nu,\tau]=[\tfrac\pi 2,0,0,0], [\tfrac{\pi} {2},0,0,\pi], [\tfrac\pi 2,\pi,0,0],\text{ and } [\tfrac{\pi} {2},\pi,0,\pi]$$ lie in $\NAT_s\setminus C_s$ for  any $|s|<s_0$.  
 Moreover, 
  $U_s$ fixes each of  these points.  
The points $[\tfrac\pi 2,0,0,0],  [\tfrac\pi 2,\pi,0,0]$ lie in the path component $\NAT_s^+$ and the points  
 $[\tfrac\pi 2,0,0,\pi],  [\tfrac\pi 2,\pi,0,\pi]$ lie in $\NAT_s^-$. 
 The identical statement holds for $\NAT_s'$.
 
 \end{proposition}
\begin{proof} 
Let $\ep=(\ep_1,\ep_2)\in\{\pm 1\}^2$ and denote by
$\rho_\ep$ the assignment
\begin{equation}\label{eqn5.10} a\mapsto \bbi,~ b\mapsto \bbj,~f\mapsto \ep_1\bbi,~ h\mapsto \ep_2 \bbj,~ p\mapsto 1, ~q\mapsto e^{s\ep_1\ep_2\bbj}.\end{equation}
Since $q$ commutes with $h$, it is easy to check, using Proposition \ref{master},  that
these satisfy $G=0=G'$ and hence lie on $\NAT_s$ and $\NAT_s'$ for any $s$.

\medskip

We turn to the second assertion. 
Note that $\rho_\ep(p)=1$ and  $\rho_\ep(q)=e^{\ep_2 s\bbj}$.
In addition, Propositions \ref{summ} and \ref{summ2}  imply that 
$$\rho_\ep(c)=\rho_\ep(\bq\bp b p q)=\bbj\text{ and }\rho_\ep(d)=\rho_\ep(\bq\bp \bb p  ba\bff q f)=\rho_\ep(
\bq a \bff q f)=\bbi
.$$
Lemma \ref{actofT} shows that $$U(\rho_\ep)(a)=\rho_\ep(\bar d)=-\bbi,~ U(\rho_\ep)(b)=\rho_\ep(\bar d \bar c d)=\bbj,~
U(\rho_\ep)(f)=\rho_\ep(\bar f)=-\ep_1\bbi,$$
and 
$$
U(\rho_\ep)(h)=\ep_2\bbj.
$$
Conjugating by $\bbj$ yields
$$U(\rho_\ep):a\mapsto\bbi, ~b\mapsto\bbj, f\mapsto\ep_1\bbi, h\mapsto\ep_2\bbj
$$
proving $U([\rho_\ep])=[\rho_\ep]$.

\medskip

  By definition, $\NAT^+_s$ denotes the path component of $\NAT_s\setminus C_s$ containing $[\tfrac\pi 2,0,0,0]$.   If $\tau(\theta)$ is defined by
$$e^{\tau(\theta)\bbi}=\frac{1-\sin\theta\bbi}{\sqrt{1+\sin^2\theta}}$$
then the path $[0,\pi]\ni\theta\mapsto [\tfrac \pi 2, \theta, 0 , \tau(\theta)]$  lies in $V_0/\hat\iota$ and misses the critical set  
$C_0$.  Hence its endpoints $[\tfrac\pi 2,0,0,0]$ and $[\tfrac\pi 2,\pi,0,0]$
lie in the same path component of  $(V_0/\hat\iota)\setminus C_0,$ and this remains true for the small deformations $\NAT_s\setminus C_s$.   Since $\tau(0)=\tau(\pi)=0$, $   [\tfrac\pi 2,\pi,0,0]\in \NAT_s^+$. 

Since $[\tfrac\pi 2,0,0,\pi]$ and $[\tfrac\pi 2,0,0,0]$ map by $\pi_0$ to the same point $[\tfrac\pi 2,0]$ in $P_0^*$, $[\tfrac\pi 2,0,0,0]\in \NAT_s^+$ implies $[\tfrac\pi 2,0,0,\pi]\in \NAT_s^-.$ Similarly $[\tfrac\pi 2,\pi,0,\pi]\in \NAT_s^-.$

\end{proof}
 
\begin{corollary}\label{preserves}
 The involutions $U_s$, $W_1$, $W_2$ satisfy
 $$U_s(\NAT_s^\pm)=\NAT_s^\pm,~
 W_1(\NAT_s^\pm)=\NAT_s^\pm,~\text{ and }
 W_2(\NAT_s^\pm)=\NAT_s^\pm.$$
 The identical statement holds for $\NAT'_s$.
\end{corollary}
\begin{proof}
Proposition \ref{finepeopleonbothsides} contains the first statement. By definition, $[\tfrac\pi 2, 0, 0, 0]\in \NAT_s^+$.  Lemma \ref{2invols} implies 
 $$W_1[\tfrac\pi 2, 0, 0, 0]=[\tfrac{3\pi}{2}, \pi, 0,\pi]=[\tfrac\pi 2, \pi,0,0]\in \NAT_s^+.$$
 Hence $W_1$ takes $\NAT_s^+$ to itself.
 Similarly, $W_2[\tfrac\pi 2, 0, 0, 0]=[\tfrac\pi 2, 0,0,0]$. Thus $W_2(\NAT^+_s)=\NAT^+_s$, as claimed.
\end{proof}

\subsubsection{A parameterization of the preimage of the bottom edge}

 For non-zero $s$, denote by
 $$K_s=\{[\gamma,0, \nu, \tau]\in \NAT_s\},$$ 
 the preimage  of $\beta$ under $\pi_0:\NAT_s\to P_0^*$.
 Lemma \ref{preimageisgood} and Equation (\ref{eq417}) imply that for all small enough non-zero $s$, $\beta\in \Lag^{\pitchfork_{u_s}}(P_0^*)$ and intersects the fold image transversely twice.   Therefore, for such $s$, $K_s$ is a smooth circle which meets $C_s$ transversely in two points.  The identical discussion applies to  $K_s'=\pi_0^{-1}(\beta)\cap\NAT_s'$.

\begin{lemma}
 \label{magicformula}\hfill

   \begin{enumerate}
\item[$(\NAT'_s)$] Assigning, to $\sigma\in \RR/2\pi \ZZ$, the conjugacy class of 
$$a\mapsto \bbi, ~ b\mapsto e^{sh}  e^{-\sigma e^{\sigma\bbi}\bbk}\bbi e^{-sh}, ~~ f\mapsto\bbi, h\mapsto e^{\sigma\bbi}\bbj,
   ~~ p\mapsto e^{sh}e^{s \cos\sigma e^{\sigma\bbi}\bbk}e^{-sh},~~ q\mapsto e^{sh},$$  defines an embedding $S^1\to \NAT_s'$ parameterizing  $K'_s$,  for  sufficiently small non-zero $s$.
 
 \item[$(\NAT_s)$] Assigning, to $\sigma\in \RR/2\pi \ZZ$, the conjugacy class of 
$$a\mapsto  \bbi, ~
b\mapsto  e^{sh} e^{-\sigma e^{\sigma\bbi}\bbk}\bbi e^{-sh}, ~f\mapsto  \bbi, ~
h\mapsto  e^{\eta e^{\sigma\bbi}\bbk}e^{\sigma\bbi}\bbj e^{-\eta e^{\sigma\bbi}\bbk},
~p\mapsto  e^{sh}e^{-2\eta e^{\sigma\bbi}\bbk }e^{-sh}, ~q\mapsto  e^{sh},$$
with $\eta=\eta(s,\sigma)$ the unique solution to $2\eta= -s\cos(\sigma+2\eta)$,
 defines a defines an embedding $S^1\to \NAT_s$ parameterizing  $K_s$,  for  sufficiently small non-zero $s$.

\end{enumerate}
\end{lemma}

\noindent{\em Remark:} These circles of representations are not partially gauge fixed in the form described in Section \ref{eliminating}, because $\Real(b\bbk)$ need not  be zero.  
\begin{proof}

 $(\NAT'_s)$.  Since $f=a$, $\Ima(f\bar a h)=h$ and so $ e^{sh}=e^{s\Ima(\bar f a h)}.$ Therefore $q$ satisfies the perturbation condition.
 Furthermore, 
\begin{equation}\label{imbha}bh=e^{sh} e^{-\sigma e^{\sigma\bbi}\bbk}\bbi e^{-sh}h=e^{sh}e^{\sigma h\bbi}\bbi he^{-sh}=\sin\sigma +e^{sh}(\cos\sigma e^{\sigma\bbi}\bbk)e^{-sh},
\end{equation}
 and so 
 $e^{sh}e^{s \cos\sigma e^{\sigma\bbi}\bbk}e^{-sh}=e^{s\Ima(bh)}$. Thus the perturbation condition holds for $p$ as well.
 
Note that $ph=h\bar p$ and $qa=a\bar q$, and that, by construction, $hq=qh$. Using Proposition \ref{master}, $G_2'=\Real(ah)=0$ and
\begin{align*}
G_1'&=\Real(\bq \bp h p q h a) =\Real(\bq^2\bp^2 \bar a)  \\
&=\Real(e^{-2sh}e^{sh}e^{-2s \cos\sigma e^{\sigma\bbi}\bbk} e^{-sh}\bar \bbi)\\
&=\Real(e^{- sh}e^{-2s \cos\sigma e^{\sigma\bbi}\bbk} \bar \bbi e^{sh})\\
&=\Real(e^{-2s \cos\sigma  e^{\sigma\bbi}\bbk}\bar \bbi)
=0,
\end{align*}
and so this smooth circle of representations, parameterized by $\sigma$, maps  into $\NAT_s'$. Since $f=a=\bbi$ for all $\sigma$, the image lies in $K'_s$.

To see that this circle parameterizes $K_s'$, we calculate the characters 
$\Real(b\bar a)$ and $\Real(b\bar h)$ on this circle of representations:
$$\Real(b\bar a)=\Real(e^{sh}  e^{-\sigma e^{\sigma\bbi}\bbk}\bbi e^{-sh}\bar\bbi)=\Real(e^{2sh}  e^{-\sigma e^{\sigma\bbi}\bbk})=\cos(2s)\cos\sigma,
$$
and, using (\ref{imbha}),
$$\Real(b\bar h)=-\sin\sigma.$$
These factor through $\NAT_s'$ (since they are characters) and embed the circle into $\RR^2$ if $|s|<\tfrac\pi4$. Therefore, the circle embeds in $\NAT_s'$ with image in $K_s'$, and must therefore be a parameterization of $K_s'$ for small enough $s$.

\bigskip

$(\NAT_s)$.   Notice first that for $|s|<1$, there is a unique solution $\eta=\eta(s,\sigma) $ to the equation $2\eta= -s\cos(\sigma+2\eta)$ for any $\sigma$; $\eta(s,\sigma)$ varies smoothly with $\sigma$ and $s$.

For the given representation,  $e^{sh}=e^{s\Ima(\bar f a h)}$.   Thus the perturbation condition holds for $q$.
Moreover, $$h=e^{\eta e^{\sigma\bbi}\bbk}e^{\sigma\bbi}\bbj e^{-\eta e^{\sigma\bbi}\bbk}=e^{2\eta e^{\sigma\bbi}\bbk}e^{\sigma\bbi}\bbj =e^{(2\eta+\tfrac\pi 2) e^{\sigma\bbi}\bbk}\bbi.$$
From this one calculates:
\begin{align}\label{imbhb}
bh&=e^{sh} e^{-\sigma e^{\sigma\bbi}\bbk}\bbi e^{-sh}h=e^{sh}e^{-\sigma e^{\sigma\bbi}\bbk}\bbi   e^{(2\eta+\tfrac\pi 2) e^{\sigma\bbi}\bbk}\bbi e^{-sh}=
e^{sh} e^{ (-\sigma -2\eta+\tfrac \pi 2)e^{\sigma\bbi}\bbk}           e^{-sh}\\
&=\sin(\sigma+2\eta)+e^{sh}\cos(\sigma+2\eta)e^{\sigma\bbi}\bbk e^{-sh}.\nonumber
\end{align}
and hence
$e^{s\Ima(bh)}=e^{sh}e^{s \cos(\sigma+2\eta)e^{\sigma\bbi} \bbk}e^{-sh}= 
e^{sh}e^{-2\eta e^{\sigma\bbi} \bbk}e^{-sh}$ and the perturbation condition holds for $p$ also.

Since $\Real(h)=0$, $\bar h=-h=-e^{\eta e^{\sigma\bbi}\bbk}e^{\sigma\bbi}\bbj e^{-\eta e^{\sigma\bbi}\bbk}=
-e^{2\eta e^{\sigma\bbi}\bbk}e^{\sigma\bbi}\bbj $, and so
$$
 pq\bar h a = -e^{sh}e^{-2\eta e^{\sigma\bbi}\bbk }
e^{2\eta e^{\sigma\bbi}\bbk}e^{\sigma\bbi}\bbj
  \bbi
 = e^{sh}e^{\sigma\bbi}\bbk.
$$
 Hence,
 $$G_2=\Real(pq\bar h a)=\sin s \Real(h e^{\sigma\bbi}\bbk)=\sin s\Real(e^{\sigma\bbi}\bbj e^{\sigma\bbi}\bbk)=0$$
and 
$$G_1=
\Real(pqhah)=-\Real(e^{sh}e^{\sigma\bbi}\bbk h)
=-\Real(he^{sh}e^{\sigma\bbi}\bbk )=-\cos s \Real( h e^{\sigma\bbi}\bbk)=0,$$
and therefore this circle of representations maps into $\NAT_s$ and, since $f=a=\bbi$ for all $\sigma$, into $K_s$.

The   characters of $b\bar h$ and $b\bar a$ evaluated along this circle satisfy $$\Real(b\bar h)= -\sin(\sigma+2\eta)=-\sin\sigma+O(s),\text{ and }\Real(b\bar a)=\cos\sigma + O(s),$$
and therefore this circle must parameterize  $K_s$ for all small enough $s$,  completing the proof of Lemma \ref{magicformula}.
\end{proof}

\subsubsection{ The action of $(u_s)_*$ and $(u' _s)_*$ on the bottom edge}

\begin{proposition}\label{theta0} Let  $\beta\in \Lag_{\rm int}(P_0^*)$ be the  bottom edge, $\beta(\gamma)=[\gamma,0],~0<\gamma<\pi.$ Then, for sufficiently small non-zero $s$,   
 $\beta\in \Lag^{\pitchfork_{u_s}}(P_0^*)$ and $(u_s)_*(\beta):K_s\to P_1^*$ is an immersion of a circle with a single transverse double point.  Moreover,  
 $$(u_s)_*(\beta)=\Psi\circ F8(\beta)\text{ in } [\Lag_{\rm cir}(P_1^*)].$$ 
The identical statement holds for $(u'_s)_*$.
\end{proposition}

\begin{proof}  
Lemma \ref{preimageisgood} implies  that $\beta\in \Lag^{\pitchfork_{u_s}}(P_0^*)$ and intersects the fold image transversely twice for sufficiently small  non-zero $s$.

Lemma \ref{factor} and the top row of Diagram (\ref{Jsquare})  imply that  
 $(u_s)_*(\beta)$ is given by the composite
\begin{equation}\label{nats} (u_s)_*(\beta):K_s\subset \NAT_s\xrightarrow{\pi_0\circ U_s}P_0^*\xrightarrow{\Psi\circ \Theta}P_1^* \mbox{ and similarly for $u'_s$, $K'_s$}.\end{equation}

Apply Lemmas \ref{actofT2} and  \ref{magicformula} to calculate  $\pi_0\circ U_s:K_s\to P_0^*\subset \RR^3$. The first coordinate is given by
\begin{equation}\label{cosghat}\Real(U(\rho)(b\bar a)) =\Real( \bar b q  \bbi \bq )=
\Real(\bq  \bar b q \bbi )=\Real(e^{-\sigma e^{\sigma\bbi}\bbk}\bar\bbi\bbi)=\cos\sigma. 
\end{equation}

All assertions and calculations made so far in this proof are true, for identical reasons, for $u_s'$ and $\NAT_s'$. The calculations of the remaining two coordinates differ slightly in the two cases of $u_s$ and $u_s'$.    Lemma \ref{magicformula}  and two calculations in its proof, Equations (\ref{imbha}) and (\ref{imbhb}), yield:
\begin{equation}\label{costhat}\Real(U(\rho)(f\bar a))=\Real(\bar p^2)
= 
\begin{cases}\cos(2s\cos\sigma)&\NAT_s'\\
 \cos(2s\cos(\sigma+ 2\eta))&\NAT_s
\end{cases}
\end{equation}
and:
\begin{equation}\label{cosgthat}
\Real(U(\rho)(b\bar f))=
\Real( \rho(\bar b \bp^2 q \bbi \bq )    )=\begin{cases} \cos(\sigma+2s\cos\sigma)&\NAT_s'\\
\cos(\sigma +2s\cos(\sigma+2\eta))&\NAT_s.
\end{cases}
\end{equation}
 (with $\eta(s,\sigma)$ the unique solution to $2\eta= -s\cos(\sigma+2\eta)$). 
 
 \medskip

Suppose that $\sigma_1\ne \sigma_2\in \RR/2\pi\ZZ$ are sent to the same point by $\pi_0\circ U_s:K'_s\to P_0^*$.   From Equation (\ref{cosghat}), we can assume $\sigma_2=-\sigma_1$ with $0<|\sigma_1|<\pi$.    Then, by Equation (\ref{cosgthat}), 
$$\cos(\sigma_1+2s\cos\sigma_1)=\cos(\sigma_2+2s\cos\sigma_2).
$$ 
This implies that that 
$$\sin\sigma_1\sin(2s\cos\sigma_1)=\sin\sigma_2\sin(2s\cos\sigma_2)=-\sin\sigma_1\sin(2s\cos\sigma_1)
$$
and therefore $\sigma_1=\pm \tfrac\pi 2$.  Hence 
$\pi_0\circ U_s:K'_s\to P_0^*$ has a single double point.

The argument for the earring tangle is similar. 
Suppose that $\sigma_1\ne \sigma_2\in \RR/2\pi\ZZ$ are sent to the same point by $\pi_0\circ U_s:K_s\to P_0^*$.
Then 
$$\cos\sigma_1=\cos\sigma_2,~\cos(2s\cos(\sigma_1+2\eta_1))=\cos(2s\cos(\sigma_2+2\eta_2)),$$
and $$\cos(\sigma_1+2s\cos(\sigma_1+2\eta_1))=\cos(\sigma_2+2s\cos(\sigma_2+2\eta_2)),
$$ 
where $2\eta_i=-s\cos(\sigma_i+2\eta_i)$.  These three equations are equivalent, through angle addition formulas, to 
$$\cos\sigma_1=\cos\sigma_2,~\cos(4\eta_1)=\cos(4\eta_2), 
\text{ and  }\cos(\sigma_1-4\eta_1)=\cos(\sigma_2-4\eta_2),
$$

It follows that $\sigma_2=-\sigma_1\ne0,\pi$,  $\eta_1=\pm \eta_2$, and 
$$\sin\sigma_1\sin(4\eta_1)=\sin\sigma_2\sin(4\eta_2)=-\sin\sigma_1\sin(4\eta_2),
$$
and therefore  $\eta_2=-\eta_1$.   Since $2\eta_i=-s(\cos\sigma_i\cos2\eta_i-\sin\sigma_i\sin2\eta_i),$
$$
2\eta_1=-s(\cos\sigma_1\cos2\eta_1-\sin\sigma_1\sin2\eta_1)=
 -s(\cos\sigma_2\cos2\eta_2-\sin\sigma_2\sin2\eta_2)=2\eta_2,
$$
which implies that $\eta_1=\eta_2=0$ since $|s|$ is small and non-zero.  Hence
$\sigma_1=\pm \tfrac\pi 2$.  Therefore, 
$\pi_0\circ U_s:K_s\to P_0^*$ also has a single double point.
A straightforward calculation, using Equations (\ref{cosghat}), (\ref{costhat}), and (\ref{cosgthat})  
 shows that 
\begin{equation}\label{transverse}
\left. \frac{d}{dt}\right|_{t=0} \pi_0\circ U_s(\pm (\tfrac{\pi} {2}+t))
=\begin{cases} (-1,0,-1\pm 2s)& \NAT'_s ,\\
 (-1,0,-1\pm 2s-4s\tfrac{d\eta}{dt})&\NAT_s.\end{cases}
\end{equation}
 Since $\eta=O(s)$, this shows $\pi_0\circ U_s$ on both $K_s$ and $K'_s$ have single transverse double points (after possibly shrinking the bound  $s_0$ in the case of the earring).   
Since $\Psi\circ\Theta:P_0^*\to P_1^*$ is a diffeomorphism,   $\sigma=\pm \tfrac\pi 2$ is a transverse self-intersection point of $(u'_s)_*(\beta)$ and $(u_s)_*(\beta)$.  This completes the proof of the first statement in the proposition in both cases.

 \medskip

   We first prove the second statement for the bypass case.  It follows from Equation (\ref{fig8inr3}) that the 
  composition of the parameterization of $K_s'$ from Lemma \ref{magicformula} with $\pi_0\circ U_s:\NAT_s'\to P_0^*$ equals
  $\tilde\beta\circ F8_s$,
with $\tilde \beta$ the extension of $\beta$ given in Equation (\ref{eqn6.3}).  Therefore,
 $\pi_0\circ U_s:K_s'\to P_0^*$ equals $F8(\beta)$ in $[\Lag_{\rm cir}(P_0^*)]$.
Equations (\ref{Theta}) and (\ref{fig8inP}) imply that  $\Theta:P_0\to P_0$ preserves $F8(\beta)$ (as an unparameterized immersion),  completing the proof of Proposition \ref{theta0} for the bypass tangle.
 
 \medskip

 For the earring tangle, the formulas are less explicit so we construct a regular homotopy $H:S^1\times [0,1]\to P_0$ from $\pi_0 \circ U_s$ on $K_s$ to $\pi_0 \circ U_s $ to $K_s '$.   Define
  $$H(\sigma,u)=\big(\cos\sigma, \cos(2s\cos(\sigma+(1-u)2\eta)),\cos(\sigma +2s\cos(\sigma+(1-u)2\eta) )\big), ~u\in[0,1].
$$

 Since $|s|\neq 0$ is small,   if $|\cos\sigma|=1$, then
$$\cos(\sigma +2s\cos(\sigma+(1-u)2\eta)) 
=\cos\sigma\cos(2s\cos(\sigma+(1-u)2\eta))\ne \pm 1.$$
Hence the image of $H$ misses the corners  of $P_0$.

We prove $H$ is a regular homotopy, that is, $H(-,u)$ is an immersion for each $u\in [0,1]$.  For fixed $u$, the  derivative of the first coordinate of $H(-,u)$ is non-zero unless $|\cos\sigma|=1$.

The derivative of the third component of $H(-,u)$ equals
$$
\tfrac{\partial}{\partial\sigma}\big(\cos(\sigma +2s\cos(\sigma+(1-u)2\eta)\big)=
\sin(\sigma +2s\cos(\sigma+(1-u)2\eta))\cdot (1+2s\tfrac{\partial}{\partial\sigma}(\cos(\sigma+(1-u)2\eta)).
$$
When $s\ne 0$ is small and $\sigma=0$ or $\pi$, 
the second factor is non-zero and the  first factor equals 
$
 \pm\sin(2s\cos((1-u)2\eta)). 
$
 This is non-zero for $s\neq 0$ sufficiently small, since $ \eta(\sigma,s)=O(s)$.
 It follows that $H$ is a regular homotopy in $P_0^*$ from  
$$(\pi_0\circ U_s):K_s\to P_0^*\text{ to } F8(\beta)=(\pi_0\circ U_s):K_s'\to P_0^*.$$ 
Composing with $\Psi\circ \Theta$ gives the  regular homotopy between 
$
(u_s)_*(\beta)$ and $(u_s')_*(\beta)=\Psi(F8(\beta))$.
  \end{proof}

  This completes the analysis of the special case of a good arc (the bottom edge)  that we set out to describe in this section.  Before moving on to discuss more general good arcs,  we note that the proof of Lemma \ref{theta0} also gives a characterization of the action of $(u_s)_*,(u' _s)*$ on any closed sub-arc of the bottom edge containing the bottom left corner.

    \begin{lemma}\label{fig8p} Fix some $0<\ep< \pi  $ and
 let $\beta_{\rm bl}:[0,\ep]\to P_0$ be given by $\beta_{\rm bl}(\sigma)=[\sigma,0]$.
  Then the fiber product of $\beta_{\rm bl}\times{\rm Id}$ and $u_s$ is a closed interval,     the Weinstein composition $(u'_s)_*(\beta_{\rm bl})$  can be parameterized as   
the composite $$[-\ep,\ep]\ni \sigma\mapsto\Psi \big([\sigma, -2s\cos \sigma])$$ and  the Weinstein composition  
$(u_s)_*(\beta_{\rm bl})$ can be parameterized as   
 $$[-\ep,\ep]\ni \sigma\mapsto\Psi  \big(  [\sigma, -2s\cos( \sigma+\eta(s,\sigma))]).$$
  \qed  
    \end{lemma}
    
  The notation $\beta_{\rm bl}$ is used to highlight  that this is an embedded arc in $P_0$ which starts at the bottom left corner and lies in the bottom edge.  Restricting the parameterization in Lemma \ref{magicformula} to the interval $[-\epsilon, \epsilon]$ defines a parameterization   
 $\alpha_{{\rm  bl},s}:[-\ep, \ep]\to \NAT_s$  
 for the preimage  of $\beta_{\rm bl}$ which is transverse to the fold locus in one point. 
 One can check by plugging in $\sigma=\pm \tfrac \pi 2$  that the parameterization  $\alpha_{{\rm  bl},s}$  starts in $\NAT_s^+$ and ends in $\NAT_s^-$ (note that $\sigma=\pm \tfrac \pi 2$ implies  $\eta(s,\sigma)=0$).   The formula 
\begin{equation} \label{alphas}  \pi_1\circ \alpha_{{\rm  bl},s}(\sigma)=
  \Psi\big( [\sigma, -2s\cos( \sigma+\eta(s,\sigma))]\big)\end{equation} 
  gives the parameterization of the Weinstein composition $(u_s)_*(\beta_{\rm bl})$ in Lemma \ref{fig8p}.

Using the symmetries $W_1,W_2$,  we can easily extend this to sub-arcs of the top or bottom edge starting at either corner.   
For example,  $\beta_{\rm br}=\widehat{W}_2\circ \beta_{\rm bl}$ is a sub-arc of the bottom edge starting at the bottom right corner and $\alpha_{{\rm br},s}=W_2\circ \alpha_{{\rm bl},s}$     parameterizes the preimage of $\beta_{\rm br}$.  Again, 
$ \alpha_{{\rm  br},s}$ starts in $\NAT_s^+$ and ends in $\NAT_s^-$ (by Corollary \ref{preserves}), and $ \alpha_{{\rm  bl},s}$ meets the fold locus $C_s$ transversely in one point. Moreover, the Weinstein composition $(u_s)_* (\beta_{\rm br}) $ is parameterized by 
$\pi_1\circ W_2 \circ \alpha_{{\rm bl}, s}.$
Similarly, define $\beta_{\rm tr} =\widehat{W}_1\circ \beta_{\rm bl}$, 
 $\alpha_{{\rm tr},s}=W_1\circ \alpha_{{\rm bl},s} $,  
 $\beta_{\rm tl} =\widehat{W}_1\circ \widehat{W}_2 \circ \beta_{\rm br}$,  and 
 $\alpha_{{\rm tl},s}=W_1\circ  W_2\circ  \alpha_{{\rm br},s} $.  This notation is illustrated in Figure \ref{alphasfig}.

A similar approach works for the bypass tangle. In particular, the parameterization of $(u'_s)_* (\beta_{{\rm bl},s})$ analogous to Equation (\ref{alphas}) is  $\Psi \left( [ \sigma, -2s\cos(\sigma)]\right).$

\begin{figure}[ht]
\begin{center}
\includegraphics[width=.95\textwidth]{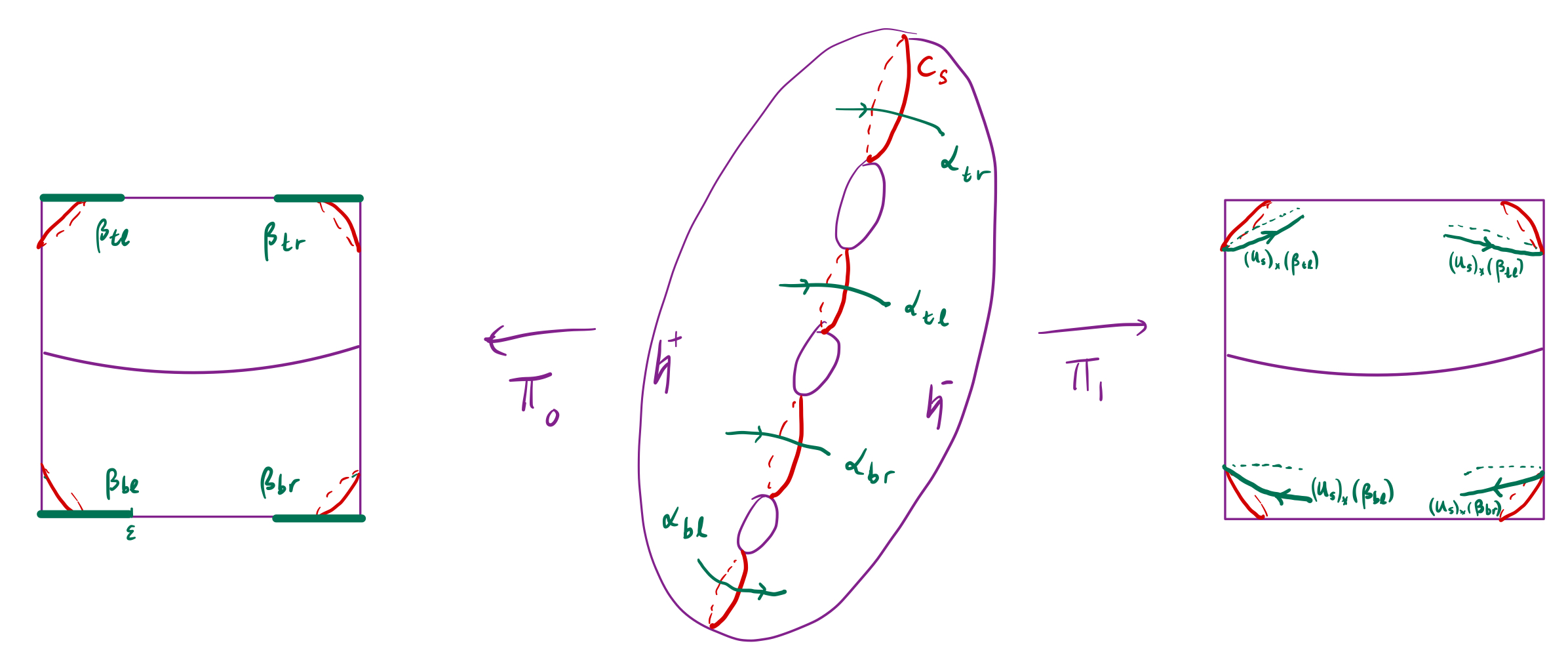}
 \caption{The map from $\NAT_s$ to the product of two pillowcases.     The four arcs $\alpha_{\rm br},\alpha_{\rm bt},\alpha_{\rm tr},\alpha_{\rm tl} $, making up the  preimage of the four horizontal arcs  $\beta_{\rm br},\beta_{\rm bt},\beta_{\rm tr},\beta_{\rm tl} $ under $\pi_0$ as well as their image under $\pi_1$ are illustrated. Also illustrated is the fold locus $C_s$ and the fold images $\pi_0(C_s)$ and $\pi_1(C_s)$. \label{alphasfig}}
\end{center}
\end{figure}

\subsection{Correspondence theorem for arbitrary good arcs}

In this section we establish a similar result to  Proposition   \ref{theta0}  for arbitrary good arcs.

\begin{theorem}\label{FIG8baby} Let $g\in \Lag_{\rm int}(P_0^*)$.  Then for all sufficiently small non-zero $s$, \begin{itemize} 
 \item $ g\in \Lag^{\pitchfork_s}(P_0^*)$, 
 \item the regular homotopy class of 
 $(u_s)_*( g)$ depends only on the class of $ g$ in  $[\Lag_{\rm int}(P_0^*)]$,  and 
 \item $[(u_s)_*(g)] =\Psi\circ F8( [g])$  in $[\Lag_{\rm cir}(P_1^*)]$.
 \end{itemize} 
 \end{theorem}

\begin{proof}  It suffices to provide a proof for  $g:[0,\pi]\to P_0^*$ so that the restriction to $(0,\pi)$ is   a good immersion. 
Choose an $s_1\in (0,s_0)$ and $\epsilon\in (0,\tfrac \pi 2)$ so that $g$ meets the fold image transversely in two points  and the fold image lies in an $\epsilon$ neighborhood of the corners for all $0<|s|<s_1$.  
 Then $g$  can be  regularly homotoped, through good immersions which meet the fold circles transversely in two points, to a path $g'$ which is a concatenation of three paths,  
 $g'=g_1*g_2*g_3$  where, as illustrated on the right side in Figure \ref{reghomfig},

\begin{enumerate}
\item The path $g_1$ is one of $\beta_{\rm bl}, \beta_{\rm br},\beta_{\rm tl},\beta_{\rm tr}$, with domain $[0,\epsilon]$.  In particular $g_1$   is an embedded path in the top or bottom edge starting at a corner.
\item The path $g_3$ is obtained from one of $\beta_{\rm bl}, \beta_{\rm br},\beta_{\rm tl},\beta_{\rm tr}$
by precomposing with the linear orientation-reversing homeomorphism $[\pi-\ep, \pi]\to[0,\ep]$. In particular $g_3$   is an embedded path in the top or bottom edge ending at a corner.
\item The path $g_2$ has domain $[\epsilon, \pi-\epsilon]$, starts at the endpoint of $g_1$, ends at the initial point of $g_3$, and stays outside the $\epsilon$ neighborhoods of the corners containing the fold image.
\end{enumerate}

\begin{figure}[ht]
\begin{center}
\includegraphics[width=.75\textwidth]{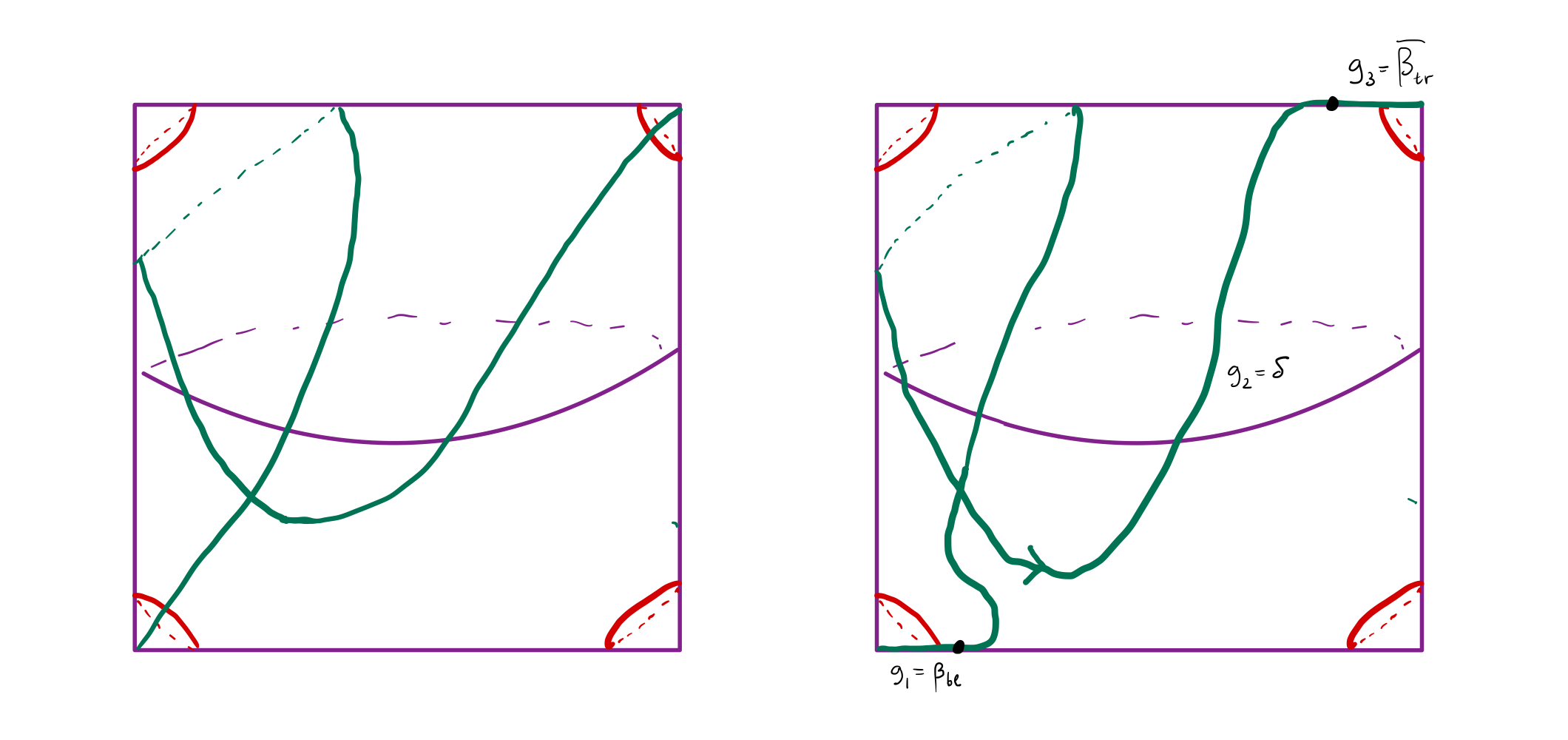}
 \caption{A    regular homotopy  in $\Lag_{\rm int}^{\pitchfork_{u_s}}(P_0^*)$,   starting from a good path and ending at $g=g_1*g_2*g_3=\beta_{\rm bl}*\delta*\overline{\beta_{\rm tr}}$.\label{reghomfig}}
\end{center}
\end{figure}

Let $M(t,u)$ be such a regular homotopy from $g(t)$ to $g'(t)$.   The local form of a fold map and the fact that, away from the fold, $\pi_0:\NAT_s\to P_0$ is a trivial cover over its image, imply that the preimage $L_u\subset \NAT_s$ of the immersed path $M(-,u)$ under $\pi_0:\NAT_s\to P_0$ is a smoothly immersed circle transverse to the fold locus for each $u$.   These preimages can be parameterized to depend smoothly on $u$ to create a regular homotopy $\widetilde{M}:S^1\times[0,1] \to \NAT_s$ from $L_0$ to $L_1$ such that $\widetilde{M}(-,u) $ parameterizes $L_u$. Since $u_s$ is a Lagrangian immersion, Theorem \ref{Cain died, sadly} implies the composite 
$\pi_1\circ \widetilde{M}(-,u) :S^1\to \NAT_s\to P_1$  
is an immersion for each $u\in [0,1]$ and so  $\pi_1\circ \widetilde{M}$ is a regular homotopy from $(u_s)_*(g)$ to  
$(u_s)_*(g')$.   Hence it suffices to prove the theorem for  good curves of the form $g=g_1*g_2*g_3$, with $g_1,g_2,g_3$ satisfying the conditions listed above.

\medskip

\medskip

 To keep the exposition simple, we will discuss the earring case and we assume that 
 \begin{equation*} g=g_1*g_2*g_3=\beta_{\rm bl} * \delta * \overline{\beta_{\rm tr}}.\end{equation*}  An example is illustrated on the right  in Figure \ref{reghomfig}.   Other cases are handled similarly.
 With these choices, the fiber product of $g$ with $u_s$ is an immersed circle in $\NAT_s$ (or the analogous object in the $\NAT'_s$ case), which can be smoothly parameterized as follows.  Recall that $\delta_s ^{\pm}$ are the two lifts of $\delta$ in $\NAT_s^\pm\subset \NAT_s$, so that $(u_s)_*(\delta)=\pi_1\circ\delta_s^+\sqcup \pi_1\circ\delta_s^-$, as in Lemma \ref{lem6.7}.
For $\sigma\in \RR/2\pi\ZZ$ and $|s|<s_1$, set
\begin{equation} \label{fourfold concat2} L(s,\sigma)=
\begin{cases}
\alpha_{{\rm bl},s}(\sigma)& \text{ if } \sigma\in[-\ep, \ep],\\
\delta^-_s(\sigma)& \text{ if } \sigma\in[\ep, \pi-\ep],\\
W_1\circ \alpha_{{\rm bl},s}(\pi-\sigma)& \text{ if } \sigma\in[\pi-\ep, \pi+\ep],\\
\delta^+_s(2\pi-\sigma)&  \text{ if } \sigma\in[\pi+\ep, 2\pi-\ep].
\end{cases}
\end{equation}

The argument is finished by proving that 
the immersed circle of Equation (\ref{fourfold concat2})  
lies in the regular homotopy class $F8([g])$. To further ease the notation, we denote the paths in the concatenation $$L(s,\sigma)=L_1(s,\sigma)*L_2(s,\sigma)*L_3(s,\sigma)*L_4 (s,\sigma).$$

\medskip

 Let $\mu:\RR/2\pi\ZZ\to [0,1] $ be a smooth function which equals 0 on $[-\tfrac\ep 2,\tfrac\ep 2]\cup[\pi-\tfrac\ep 2,\pi+\tfrac\ep 2]$ and equals 1 outside 
 $[- \ep,  \ep ]\cup[\pi- \ep  ,\pi+ \ep ]$.
 Then, for $0<|s|<s_1$, the smooth homotopy 
 $$
H(\sigma, t)= L(s(1-t\mu(\sigma)), \sigma).$$
satisfies $$ H(\sigma,0)=L(s,\sigma)$$   and   
 \begin{equation}\label{H1}H(\sigma, 1)=L_1(s(1-\mu(\sigma)), \sigma)*\delta_0 ^- (\sigma)* L_3(s(1-\mu(\sigma)), \sigma) * \delta^+ _0 (2\pi -\sigma).\end{equation}

Lemma \ref{lem6.7} implies that the restriction of $H(-,t)$ to $[\epsilon, \pi-\epsilon] \cup [\pi+\epsilon, 2\pi -\epsilon]$  is a {\em regular} homotopy (and hence so is its restriction to a slightly larger open set).  
The restriction of $H(-,t)$  to $[-\tfrac \epsilon 2, \tfrac \epsilon 2]\cup [\pi -\tfrac \epsilon 2, \pi+ \tfrac \epsilon 2]$ is a stationary homotopy, hence a regular homotopy because $H(-,0)$ is an immersion.  By Theorem \ref{Cain died, sadly}, $\pi_1\circ H(-,t)$ is a regular homotopy on this portion of the domain.

On the transition regions $\sigma\in [-\epsilon, -\tfrac \epsilon 2] \cup [\tfrac \epsilon 2, \epsilon]$, we take advantage of the explicit description of $\alpha_{{\rm bl},s}$ in Equation (\ref{alphas}) to observe that 
\begin{equation}\label{so close}
\pi_1\circ H(\sigma,t)=\pi_1 \circ \alpha_{{\rm bl}, s(1-t\mu(\sigma))} (\sigma) = \Psi([\sigma, -2s(1-t\mu(\sigma))\cos( \sigma+\eta(s,\sigma))])\end{equation} is a regular homotopy, because $\Psi$ is a diffeomorphism and the first pillowcase coordinate is $\sigma$, it is clear that this is an immersion for each $t$ so this is a regular homotopy.
On the other transition regions $\sigma\in [\pi-\epsilon, \pi-\tfrac \epsilon 2] \cup [\pi+\tfrac \epsilon 2, \pi+\epsilon]$, 
\begin{eqnarray*} 
\pi_1\circ H(\sigma,t)&=&\pi_1 \circ W_1 \circ \alpha_{{\rm bl},(1-t\mu(\sigma))}(\pi-\sigma)\\
&=& \Psi\circ \Theta \circ \pi_0\circ U_{s(1-t\mu(\sigma))} \circ W_1 \circ \alpha_{{\rm bl},s(1-t\mu(\sigma))}(\pi-\sigma)\\
&=& \Psi\circ \widehat{W}_1\circ  \Theta \circ  \pi_0\circ U_{s(1-t\mu(\sigma))}  \circ \alpha_{{\rm bl},(1-t\mu(\sigma))}(\pi-\sigma)\\
&=& \left( \Psi\circ \widehat{W}_1\right) ( [  \pi -\sigma, -2s(1-t\mu(\sigma))\cos( \pi-\sigma+\eta(s,\pi-\sigma)) ]) 
\end{eqnarray*}  As above, this is a regular homotopy since the first pillowcase coordinate is $\pi- \sigma$.

\medskip

It remains to prove that  the smooth immersion $\pi_1\circ H(\sigma,1)$ lies in the regular homotopy class $\Psi (F8([g]))$, or equivalently, that $\Psi^{-1}\circ \pi_1\circ H(\sigma,1)$ lies in the regular homotopy class $F8([g])$.    Note first that  by Lemma \ref{lem6.7}, $\Psi^{-1}\circ\pi_1\circ H(\sigma,1)$ sends the second and fourth intervals to $ \delta$ and $\overline{\delta}.$

Equation (\ref{so close}) shows that $\Psi^{-1}\circ \pi_1\circ L_1(s(1-\mu(\sigma)), \sigma)$ is  a path in   $P_0^*$   of  
 the form \begin{equation}\label{soon1}[-\ep,\ep]\ni\sigma\mapsto [\sigma, -2s(1-\mu(\sigma))\zeta_1(\sigma,s)]\end{equation} with $\zeta_1(\sigma,s)$ strictly positive. 
  Therefore, the image of $\Psi^{-1}\circ \pi_1\circ H(-,1)$ on a neighborhood of $[-\ep,\ep]$ is teardrop shaped; it travels left along the bottom edge, then 
 wraps around the bottom left corner and returns to the bottom edge, traveling  right. This is illustrated   in Figure \ref{fig11fig}.

Similarly, 
$\Psi^{-1}\circ  \pi_1\circ L_3(s(1-\mu(\sigma)), \sigma)$ is  obtained by applying the orientation-preserving $\widehat{W}_1$ to a path in $P_0^*$ of  
 the form \begin{equation}\label{soon2}[\pi-\ep,\pi+\ep]\ni\sigma\mapsto [\pi-\sigma,-2s(1-\mu(\sigma))\zeta_2(\sigma,s)]\end{equation} with $\zeta_2(\sigma,s)$ strictly positive.  In this case the teardrop wraps around the top right corner in the opposite direction, since the first coordinates in Equations (\ref{soon1}) and (\ref{soon2}) move in opposite directions.

Using Equation (\ref{H1}), we conclude that, as $\sigma$ increases on  $[-\ep ,2\pi-\ep]$, $\Psi^{-1}\circ \pi_1\circ H(\sigma,1)$ is a path in $P_0^*$ which moves left along  the bottom edge, then winds monotonically around the bottom left corner (the initial point of $g$) and returns to bottom edge forming a teardrop shape.  It then follows $  \delta$ to the top edge. Next, it wraps monotonically around the top right corner (the terminal point of $g$), again in a teardrop shape, in the opposite direction  from the wrapping around the bottom left corner.  Finally the path returns to its starting point, following $  \delta$ backwards.   Such a path is clearly regularly homotopic to $F8([g])$. 
 \end{proof}
\begin{figure}[ht]
\begin{center}
\includegraphics[width=.65\textwidth]{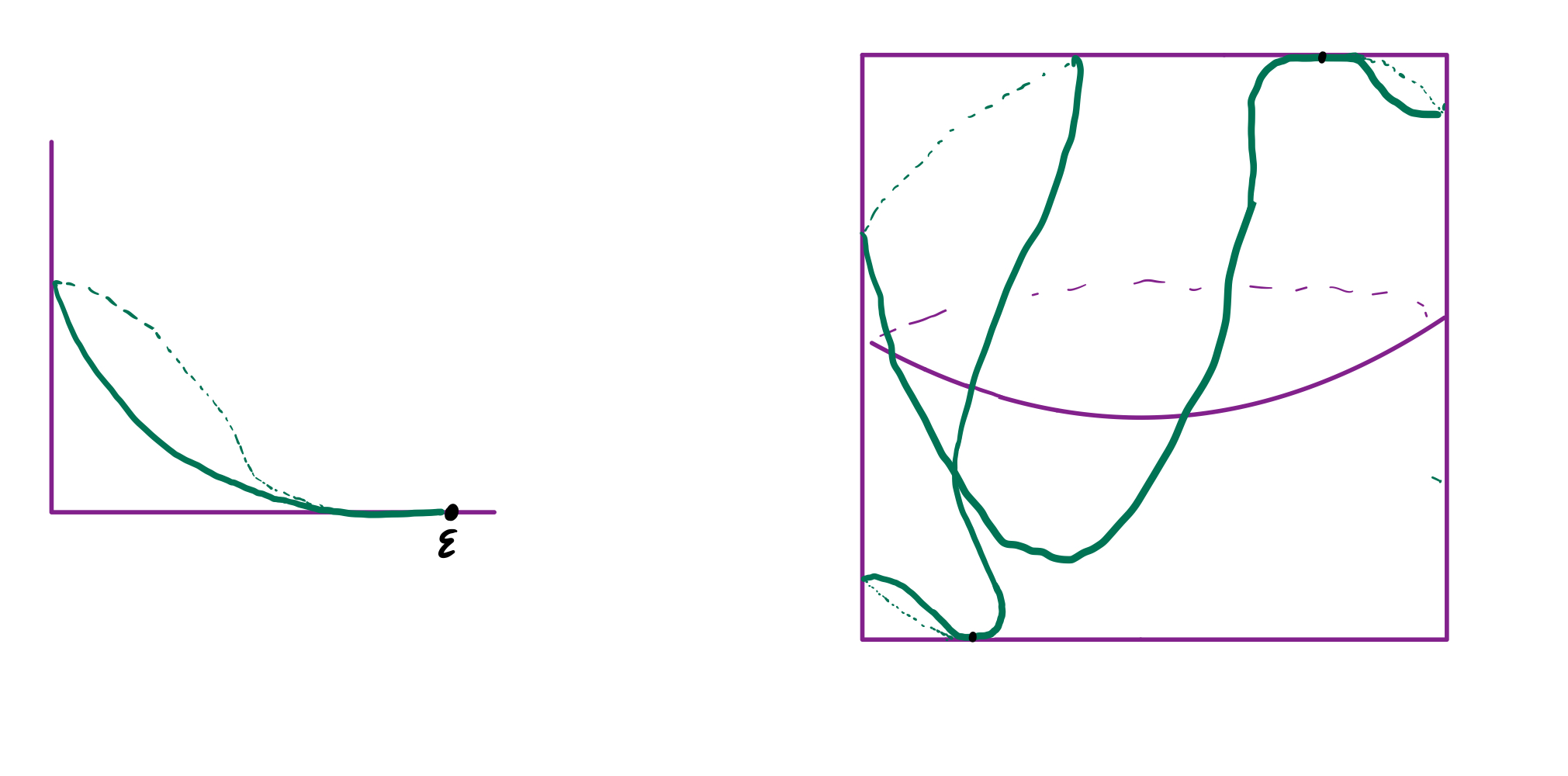}
 \caption{The image of $\Psi^{-1}\circ \pi_1\circ H(1,\sigma), ~\sigma\in\RR/2\pi\ZZ$ on the right, for the good arc  illustrated on the right in Figure \ref{reghomfig} . A close-up of the teardrop $\Psi^{-1}\circ\pi_1\circ H(1,\sigma)=\Theta\circ \pi_0\circ U_{s(1-\mu(\sigma)) } \circ  \alpha_{{\rm bl},{s(1-\mu(\sigma)) } }, \ \sigma\in[-\ep,\ep] $  on the left.  \label{fig11fig}}
\end{center}
\end{figure}

\section{Illustration by example}\label{IBE}
Following Floer, Kronheimer-Mrowka  construct the  {\em reduced singular instanton homology} $I^\nat(K)$ of a knot $K$ as the Morse homology of an appropriate  Chern-Simons function on a space  of  connections which are singular along $K$ with prescribed asymptotic  
holonomy.  Their construction begins by   placing a strong marking, e.g. an earring or bypass, at a base point on $K$.
The Atiyah-Floer conjecture  proposes that $I^\nat(K)$ should equal the appropriate Lagrangian-Floer homology of flat moduli spaces associated to  a decomposition of $(S^3,K)$ along a disjoint union of punctured surfaces. 
We explore this statement,   by assigning a generalized Lagrangian of length three to a   decomposition of $(S^3,K)$ along   the boundary of a tubular neighborhood of a Conway sphere passing through the marking. There are  two ways to obtain a pair of Lagrangians in a pillowcase using Weinstein composition.  We compare the two pairs of Lagrangians obtained by composing along in these two different ways.

In \cite{FKP, HHK2}  a decomposition of the $p,q$ torus knot $K_{p,q}$ along a Conway sphere is constructed
$$(S^3, K_{p,q})=(B^3, W_{p,q})\cup_{(S^2,4)}(B^3,V),$$
where $(B^3,V)$ denotes the trivial  two-stranded tangle  in a  three-ball and $(B^3,W_{p,q})$ denotes a more complicated  two-tangle in the  three-ball.  The traceless character varieties of $(B^3,V)$ and $(B^3, W_{p,q})$ and their image in the pillowcase $P=R(\partial(B^3,V))=R(S^2,4)$ associated to the Conway sphere are identified for many pairs $p,q$.

The traceless character variety of $(B^3,V)$ is an arc, parameterized by the character associated to the product of the two meridians of   $V$.   Denote this arc by $A_1$.   The  restriction to its boundary pillowcase 
\begin{equation}\label{r0}u_V:A_1\to P\end{equation}
maps $A_1$ linearly with slope one, joining the bottom left  and the top right corners. Thus, $A_1$ is a good arc, and so  $ u_V \in \Lag_{\rm int}(P^*).$   This is illustrated on the left in Figure \ref{fig9fig}.

\medskip

The traceless character variety of $(B^3, W_{3,7})$ is computed in \cite{FKP} to be a disjoint union of an arc and two circles.  The restriction map to its boundary pillowcase   takes the arc, which we call $A_2$, to an embedded arc of slope two limiting to the two bottom corners. It takes each of the two circles to an immersed circle regularly homotopic to $\widetilde D(B_{\rm ver})$, where $B_{\rm ver}$ denotes the  embedded vertical circle  $\{[\tfrac\pi 2,\theta]~|~ \theta\in \RR/2\pi\}$, and $\widetilde D(B_{\rm ver})$ denotes the twisted double of $B_{\rm ver}$ (see the paragraph after Definition \ref{double}).  Hence \begin{equation}\label{r1}u_{W_{3,7}}:A_2\sqcup D(\widetilde D (B_{\rm ver})) \to P\end{equation}
This is illustrated on
 the right in Figure \ref{fig9fig} (this is equivalent to Figure 7 of \cite{FKP}).  

\begin{figure}[ht]
\begin{center}
\includegraphics[width=.5\textwidth]{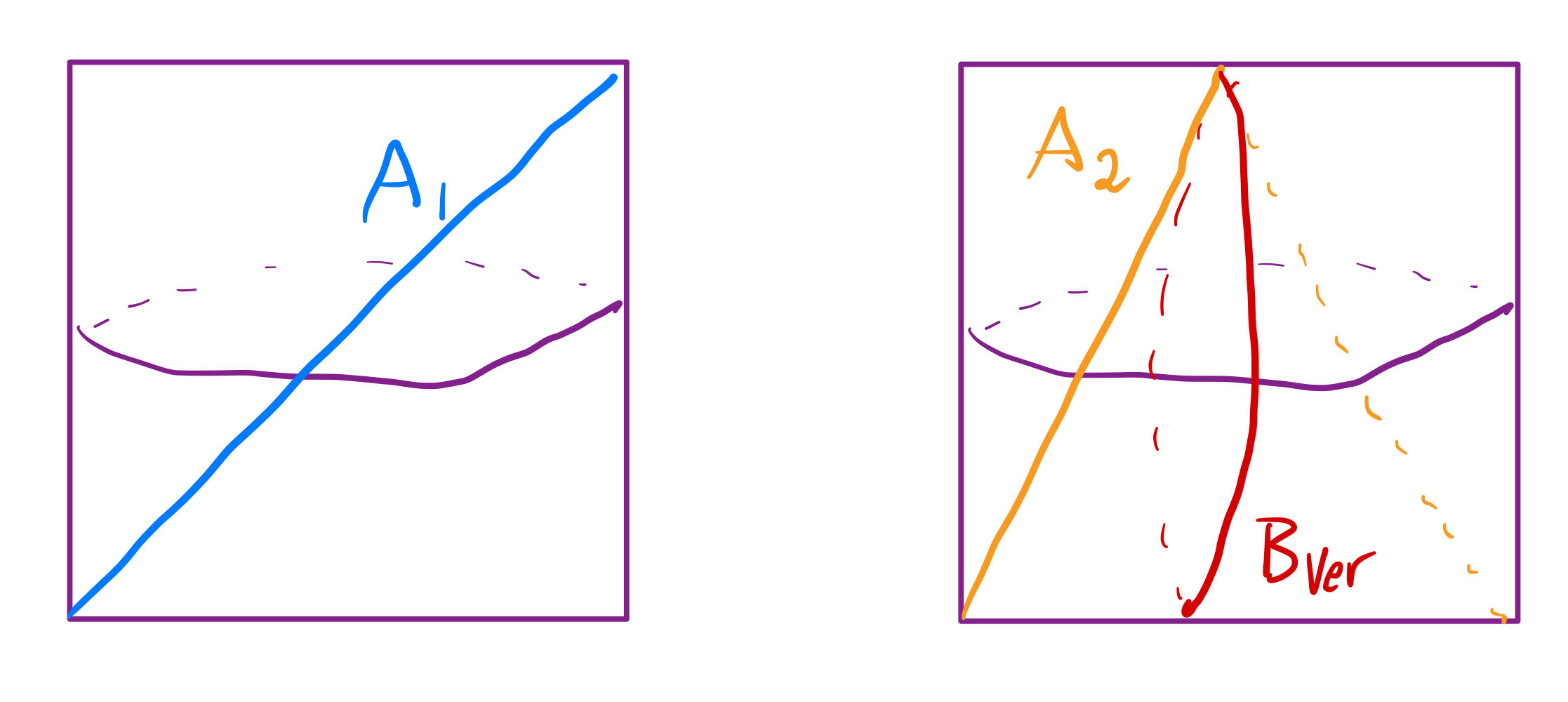}
 \caption{ The traceless character varieties of two tangles that decompose the torus knot $K_{3,7}$.  The embedded arc $A_1=R(B^3,V)$ in blue on the left. The embedded arc $A_2$ in orange and immersed  circle $B_{\rm ver}$ in red on the right,  with $R(B^3,{W_{3,7}})=A_2\sqcup D(\widetilde{D}(B_{\rm ver}))$. \label{fig9fig}}
\end{center}
\end{figure}

Placing the bypass or earring strong marking at a point where the knot intersects the  Conway sphere corresponds to replacing a neighborhood of this  two-sphere  with the bypass or earring tangle of Figure \ref{atomfig2}. 
This decomposition  of the marked knot $(S^3,K_{3,7})$ into three components,
$$(B^3,V),~\text{ the earring (or bypass) tangle, }~(B^3,{W_{3,7}}),$$ produces an immersed  generalized Lagrangian of length three  \cite{WW1}
\begin{equation}\label{2comp}\big(
u_V:A_1\to P_0^*, ~
u_s:\NAT_s\to P_0^*\times P_1^*,~u_{W_{3,7}}:A_2\sqcup D(\widetilde D(B_{\rm ver}))\to P_1^*
\big).
\end{equation}

 Theorem \ref{thm7.1} 
 describes the result of Weinstein composition corresponding to gluing along either   Conway sphere.
These two gluings 
give rise to two  pairs  of Lagrangians in the  two-dimensional pillowcases $P_0^*$ or $P_1^*$, 
 namely (with the proper interpretation of the notation):
 $$((u_s)_*(u_V), u_{W_{3,7}}) \text{ in } P_1^*$$
 and
 $$(u_V, (u_s)^*(u_{W_{3,7}})) \text{ in } P_0^*.$$

The conclusion of Theorem \ref{thm7.1} for the two compositions is
 illustrated  in Figure \ref{fig10fig}.   In either case,   the resulting pair of Lagrangians   have Lagrangian-Floer homology of rank nine (in the figure these have nine transverse intersections and there are no bigons.) These  are isomorphic to $I^\nat(K_{3,7})$.
\begin{figure}[ht]
\begin{center}
\includegraphics[width=\textwidth]{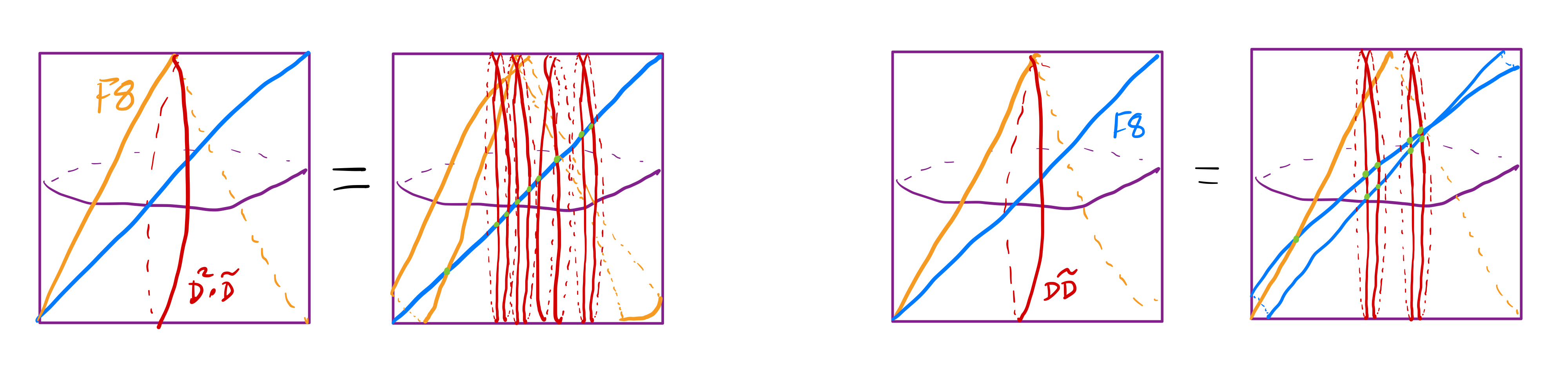}
 \caption{ The two compositions of the generalized Lagrangian of length three of Equation (\ref{2comp}). On the left the pair $(u_V, (u_s)^*(u_{W_{3,7}}))$ in $P_0$, on the right the pair $((u_s)_*(u_V),u_{W_{3,7}})$ in $P_1$ (each is presented in simplified form and in detail).
\label{fig10fig}}
\end{center}
\end{figure}

 \section{Some questions and problems}
 
We close this article with a discussion of some open questions and problems.

\subsection{} Is the Fukaya $\mathcal{A}_\infty$ algebra of  the immersion $u_s$ flat  or curved \cite{fukaya}? Is it unobstructed and, if so, what are the possible bounding cochains?
These questions concern  
properties of compactifications of the moduli spaces of $J$-holomorphic polygons in $P_0^*\times P_1^*$ with consecutive boundary segments in self-transverse push offs of $u_{s}(\NAT_{s})$.

\subsection{} Find is a suitable definition of {\em good Lagrangian immersions} into a product of pillowcases, $g:L\to P^{*,-}\times P_0 ^*$, so that the composition map 
\begin{equation} \label{enhancedB} \lim_{s\to 0} (u_s)_*: [\Lag(P ^{*,-}\times P_0 ^*)]\to [\Lag(P ^{*,-}\times P_1 ^*)]\end{equation} 
is well-defined and identify this endomorphism.  This definition should include, besides compact immersed surfaces,  some proper immersions of punctured surfaces with ends limiting to $\{\mbox{corners}\}\times \{\mbox{corners}\}$.  We expect that the top stratum of the perturbed traceless character varieties of any tangle in a homology cylinder with two Conway sphere boundary components to have this kind of structure, by arguments similar to \cite{Heraldthesis, HK, CHK}.  
Identifying the map in Equation (\ref{enhancedB}), i.e., generalizing  Theorem \ref{thm7.1} to this context, would describe the effect of  adding a marking to one Conway sphere of such a tangle. Note that when $g:L\to P^{*,-}\times P_0 ^*$ has compact domain and $s$ is small and non-zero,  $(u_s)_*(g)$ simply doubles $g$.  

Both the question about Weinstein composition above and the application to adding markings to tangles extend to more pillowcase factors and tangles with more Conway sphere boundary components, as those considered by \cite{KaiSmith}. 

 \bigskip

\newcommand*{\arxivPreprint}[1]{ArXiv preprint \href{http://arxiv.org/abs/#1}{#1}}
\newcommand*{\arxiv}[1]{ArXiv:\ \href{http://arxiv.org/abs/#1}{#1}}


\end{document}